\documentclass[twoside]{elsarticle}
\usepackage{hyperref}
\usepackage[utf8]{inputenc}
\usepackage[T1]{fontenc}
\usepackage{bbm}
\usepackage{graphicx}
\usepackage{amssymb}
\usepackage{amsmath}
\usepackage{amsthm}
\usepackage{thmtools}
\usepackage{cleveref}
\usepackage{mathtools}
\usepackage{multirow}
\usepackage{url}
\usepackage{ifpdf}
\ifpdf
\usepackage[all,pdf,2cell]{xy}\UseAllTwocells\SilentMatrices
\else
\usepackage[all,xdvi,2cell]{xy}\UseAllTwocells\SilentMatrices
\fi
\usepackage{todonotes}
\usepackage{bm} 
\usepackage{stackengine}
\usepackage{enumerate}
\usepackage{float}
\usepackage{framed}
\usepackage{marginnote}
\newtheorem{theorem}{Theorem}[section]
\newtheorem{lemma}[theorem]{Lemma}
\newtheorem{proposition}[theorem]{Proposition}
\newtheorem{corollary}[theorem]{Corollary}
\theoremstyle{definition}
\newtheorem{definition}[theorem]{Definition}
\newtheorem{example}[theorem]{Example}
\newtheorem{problem}{Problem}

\newcommand*{\vcbox}[1]{\begingroup
\setbox0=\hbox{#1}\parbox{\wd0}{\box0}\endgroup}

\newcommand{\id}{\mathrm{id}}
\newcommand{\newcategory}[1]{\expandafter\newcommand\csname #1\endcsname{\mathbf{#1}}}
\newcategory{Set}
\newcategory{Rel}
\newcategory{Fin}
\newcategory{FHilb}
\newcategory{FinSet}
\newcategory{FinRel}
\newcategory{Pos}
\newcategory{Frob}
\newcategory{DagFrob}
\newcategory{Act}
\newcategory{Sch}
\newcategory{Conf}
\newcategory{Cat}
\newcategory{FinVec}

\newcommand{\MatR}{\mathbf{Mat}_{\mathbb R_0^+}}
\newcommand{\MapR}{\mathbf{Map}_{\mathbb R_0^+}}
\newcommand{\MatS}{\mathbf{Mat}_\Sigma}
\newcommand{\MatRel}{\mathbf{Mat}_\mathbbm{2}}
\newcommand{\FinVecC}{\FinVec_{\mathbb C}}
\newcommand{\C}{\mathcal C}

\newcommand{\val}[1]{\|{#1}\|}
\newcommand{\valphi}[1]{\|{#1}\|_\phi}
\renewcommand{\>}{\rangle}
\newcommand{\unit}{e}

\newcommand{\nablaDot}{{\stackinset{c}{.1ex}{c}{.1ex}{$\cdot$}{$\bm{\nabla}$}}}

\newcommand{\based}[1]{\underbracket[0.5pt][1.5pt]{#1}}
\newcommand{\supp}{\mathrm{supp}}
\newcommand{\theO}{\Dnabla\circ\Dnabla^\dagger}

\newcommand{\LCM}{\operatorname{lcm}}
\newcommand{\Spec}{\operatorname{Spec}} 

\newcommand{\Dnabla}{{\nablaDot}} 
\newcommand{\dnabla}[3]{{\Dnabla}_{#1 #2}^{#3}}  
\newcommand{\ddnabla}[3]{\sqrt{\frac{\val{{#3}}}
{\val{#1}\val{#2}}}\nabla_{{#1}{#2}}^{#3}} 

\newlength{\marginrulesep}
  {\endMakeFramed\endgroup}
\newenvironment{proposal}{}{}

\newenvironment{remark}{\par\noindent{\bf Remark. }}{\par}

\NewDocumentEnvironment{diagram}{ob}{
  \IfNoValueTF{#1}{\xymatrix{#2}}{\xymatrix@#1{#2}}%
}{}

\makeatletter
\newcommand*\savelabel[2]{%
   \immediate\write\@auxout{%
     \noexpand\global\noexpand\@namedef{mylabel@#1}{#2}}%
   #2%
}
 
\newcommand*\mycellref[1]{%
   \ifcsname mylabel@#1\endcsname
     \@nameuse{mylabel@#1}%
   \else
     ??%
   \fi
}
\makeatother

\newcounter{mycellno}[equation]
\DeclareRobustCommand{\mycell}[1]{\stepcounter{mycellno}\savelabel{#1}{(\theequation.\themycellno)}}
\newcommand{\putmycell}[2]{\ar@{}[#1]|{\mycell{#2}}}
\newcommand{\lscellrd}{\ar@{}[rd]|{\lesssim}}
\newcommand{\lscellr}{\ar@{}[r]|{\lesssim}}
\newcommand{\gscellrd}{\ar@{}[rd]|{\gtrsim}}
\newcommand{\gscellrdd}{\ar@{}[rdd]|{\gtrsim}}
\newcommand{\lscellrdd}{\ar@{}[rdd]|{\lesssim}}
\newcommand{\simeqcellrd}{\ar@{}[rd]|{\simeq}}
\newcommand{\simeqcellr}{\ar@{}[r]|{\simeq}}
\newcommand{\simeqcellrrd}{\ar@{}[rrd]|{\simeq}}
\newcommand{\simeqcellrdd}{\ar@{}[rdd]|{\simeq}}

\begin{document}
\title{Coherent configurations and Frobenius structures}
\author[3]{Dominik Lachman}
\ead{dominik.lachman@upol.cz}
\affiliation{
organization={Department of Algebra and Geometry,
Palacký University Olomouc},
addressline={17.~Listopadu~12},
city={Olomouc},
country={Czech Republic}
}
\fntext[3]{The author received support from the Czech Science Foundation grant 25-20013L.}
\author[1]{Gejza Jenča}
\ead{gejza.jenca@stuba.sk}
\affiliation{
\organization={Department of Mathematics and Descriptive Geometry,
Faculty of Civil Engineering,
Slovak University of Technology,},
addressline={Radlinského~11},
postcode={810 05},
postcodesep={},
city={Bratislava},
country={Slovakia}}
\fntext[1]{Supported by grants VEGA 2/0128/24 and 1/0036/23.}
\author[2]{Anna Jenčová}
\ead{jencova@mat.savba.sk}
\affiliation{
organization={Mathematical Institute, Slovak Academy of Sciences},
city={Bratislava},
country={Slovakia}
}
\fntext[2]{Supported by the grant VEGA 2/0128/24.}
\begin{abstract}
We prove that coherent configurations can be represented as modules over Frobenius structures in the category
of real nonnegative matrices. We generalize the notion of admissible morphism from association
schemes to coherent configurations. We show that the Frobenius structure associated to a coherent configuration
can be modified to become a dagger Frobenius structure, and use this to connect the coherent configurations to
groupoids and $H^*$-algebras. We examine the properties of the dagger Frobenius structure
with respect to admissible morphisms. We introduce the matrix $O$ obtained as the
composition of comultiplication and multiplication of the dagger Frobenius structure and
prove that we may obtain the valencies of colors, and thus recover the original coherent
configuration, as an eigenvector of $O$. In the last part of the paper, we examine the
spectrum of $O$ and show how it relates to surjective admissible morphisms of
association schemes.
\end{abstract}
\maketitle

\section{Introduction}

The notion of an association scheme arose in the theory of experimental designs in
statistics in the first half of the twentieth century
\cite{bose1939partially,bose1952classification}. Later on, association schemes were
used in coding theory, graph theory, combinatorics, and elsewhere.  We refer to
\cite{bailey2004association,zieschang2005theory} for more information on association  
schemes and related topics. Every group gives rise to an association scheme, so 
the theory of association schemes can be considered as an extension of group theory.

Coherent configurations were introduced by Higman in \cite{higman1975coherent} as a
generalization of association schemes. They are a straightforward and natural
generalization of 2-orbits of a group acting on a finite set. Many
parts of the theory of association schemes can be extended  to the theory of coherent
configurations.

\begin{proposal}
A \emph{morphism of association schemes} \cite{zieschang1996} is naturally defined as a mapping between
vertices and colors that preserves the coloring of edges. As this notion appears to
be too weak to generalize important constructions from group theory like subgroup or
quotient group, Zieschang \cite{zieschang2005theory} introduced a stronger notion of
\emph{homomorphisms} for this purpose.  Unfortunately, homomorphisms of association
schemes are not closed with respect to composition, hence they do not form a
category. To solve this issue, French in \cite{french2013functors} proposed a less
restrictive notion of \emph{admissible morphisms}, which do form a category. He then
proved that several standard constructions used in the theory
of association schemes are functorial with respect to admissible
morphisms. In particular, admissible morphisms between association schemes 
induce morphisms between their adjacency algebras. Moreover, the category of groups
is equivalent to a full subcategory of the category of association schemes equipped
with admissible morphisms. 

The notion of an admissible morphism of association
schemes and some of the results from \cite{french2013functors} were generalized
by Kuribayashi and Matsuo in \cite{kuribayashi2015association} to the class of \emph{association (quasi-)schemoids}. This class includes all coherent configurations.
When specialized to coherent configurations, their definition of an admissible
morphism corresponds to our \Cref{def:stradmorphism} (strongly admissible morphism).
Our definition of an admissible morphism of coherent configurations in
\Cref{def:admorphism} is more general.

Let us also mention the category of association schemes with base points
considered by Hanaki in \cite{hanaki2010acategory}, attempting to generalize some of the
group constructions to the class of association schemes in another way.
\end{proposal}

The main aim of this paper is to show that coherent configurations (and hence also
association schemes) can be represented as modules over Frobenius structures in the
category of nonnegative real matrices $\MatR$ and examine several aspects of this
representation. The category $\MatR$ is a dagger compact category, just as the
category of finite dimensional Hilbert spaces $\FHilb$.
This is the main reason why our perspective and presentation are 
strongly influenced by \emph{categorical quantum mechanics}, initiated by Abramsky 
and Coecke in their seminal paper \cite{abramsky2009categorical} and recently summarized in the books
\cite{heunen2019categories,coecke2017picturing}, in which the Frobenius structures in
dagger compact categories play an important role. Whenever appropriate, we try to use
the standard language of categorical quantum mechanics -- string diagrams.

\begin{proposal}
The relations/colors of every coherent configuration, when considered as
$0,1$-matrices, generate a matrix algebra called the \emph{adjacency algebra of the
configuration}
\cite{higman1987coherent}. Since a coherent algebra is a matrix algebra, it can be
equipped with a Frobenius structure via the canonical irreducible form given by the
trace of a matrix. The naturally arising questions about the exact relationship
between the Frobenius structures investigated in this paper and the Frobenius
structure arising from the usual matrix trace is discussed at the end of
\Cref{sec:3}.
\end{proposal}

The paper is structured as follows. In the preliminaries, we introduce the main definitions and basic results concerning
coherent configurations, Frobenius structures, and dagger compact categories. We also prove some basic results about
monoids in $\MatR$.  In Section 3, we prove that coherent configurations can be represented as modules over Frobenius
structures in $\MatR$. The most important result here is \Cref{thm:connections}, which
shows how some of the classes of coherent configurations studied in the literature 
connect to existing special classes in the theory of Frobenius structures. The following two sections are devoted to an attempt to generalize the notion of an admissible morphism from
association schemes to coherent configurations. It turns out that there are at least two possible meaningful
generalizations, \Cref{def:admorphism} and \Cref{def:stradmorphism}.  In Section 6, we prove that the Frobenius
structure we associated in Section 3 to a coherent configuration can always be ``normalized'' so that it becomes a
\emph{dagger Frobenius structure}. This allows us, in a straightforward way, to connect coherent configurations to
$H^*$-algebras in \Cref{coro:hstar} and present the connection of thin coherent
configurations to groupoids (already discovered by Hanaki and Yoshikawa in
\cite{hanaki2015thin}, later reproved in \cite{kuribayashi2015association}) in an appropriate category-theoretical context in
\Cref{coro:thinCCIsGroupoid}. In Section 7, we return to the notion of (strongly) admissible morphisms and examine how
they behave with respect to the dagger Frobenius structures associated with coherent configurations. It is clear that the
normalization procedure from Section 6 can be reversed if we know the valencies of colors. In Section 8, we prove that
the information about valencies (and thus the structure constants of the coherent configuration) can be retrieved from
the dagger Frobenius structure in the form of a positive eigenvector of a certain symmetric matrix, which we denote by
$O$. In Section 9, we examine the spectral decomposition of $O$ and compute several examples.  In \Cref{thm:genLagrange}
we then show how the spectrum of the matrix $O$ transfers along a surjective
admissible morphism of association schemes. In the last section, we sketch some possible directions of
future research, and we close the paper with several open problems.

\section{Preliminaries}

We assume some knowledge of the basics of category theory. For notions undefined here, see the standard monograph
\cite{mac1998categories}. 

\subsection{Dagger compact categories and string diagrams}

Let us review elements of the graphical language of string diagrams
that we will use throughout the paper. We mostly follow the conventions
established in the book \cite{heunen2019categories}. An object $A$ is denoted by
an arrow (called string), a morphism is denoted by a labeled box, circle or a triangle.
A morphism $f\colon A\to B$ is drawn like this:
\begin{center}
\includegraphics{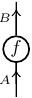}
\end{center}
If the morphisms and objects are clear from the context, we omit their labels in the
diagram. The composition $g\circ f$ of morphisms $f\colon A\to B$ and
$g\colon B\to C$ is drawn by connecting the codomain of $f$ to the
domain of $g$:
\begin{center}
\includegraphics{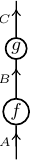}
\end{center}
Moving to monoidal categories, the monoidal product $f\otimes g\colon A\otimes C\to B\otimes D$ of morphisms $f\colon A\to B$ and $g\colon C\to D$
is drawn like this:
\begin{center}
\includegraphics{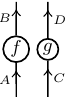}
\end{center}
The monoidal unit object $I$ is drawn as the empty diagram. For an object $A$, a
morphism $s\colon I\to A$ is called \emph{a state}, or \emph{a vector} and drawn like this:
\begin{center}
\includegraphics{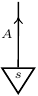}
\end{center}

An endomorphism of the unit object $t\colon I\to I$ is called \emph{a scalar}. The set of all
scalars forms a commutative monoid with respect to the composition, see \cite[Lemma 2.3]{heunen2019categories}.
There is a well-defined operation of multiplication of a morphism by a scalar:
for $f\colon A\to B$, $t\cdot f\colon A\to B$ is given by the composite
\[
\begin{diagram}
A
    \ar[r]^-{\lambda_A^{-1}}
&
I\otimes A
    \ar[r]^{t\otimes f}
&
I\otimes B
    \ar[r]^{\lambda_B}
&
B
\end{diagram}
\]
For a scalar $f$, this $t\cdot f$ coincides with $t\circ f$.
In string diagrams, we draw $t\cdot f$ like this:
\begin{center}
\includegraphics{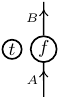}
\end{center}
A {\em right dual object} to an object $L$ of a symmetric monoidal category $(\C,\otimes,I)$ is an object
$R$ such that there are morphisms
$\eta\colon I\to R\otimes L$ and $\varepsilon\colon L\otimes R\to I$ such
that the equations
\begin{equation}
\label{eq:snakes}
\begin{aligned}
\lambda_L\circ(\varepsilon\otimes\id_L)\circ(\id_L\otimes\eta)\circ\rho_L^{-1}&=\id_L\\
\rho_{R}\circ(\id_{R}\otimes\varepsilon)\circ(\eta\otimes\id_{R})\circ\lambda_{R}^{-1}&=\id_{R}
\end{aligned}
\end{equation}
are satisfied. If $R$ is a right dual to $L$, then $L$ is a \emph{left dual to $R$}.

Denoting the $\eta$ and $\varepsilon$ by

\begin{center}
~\par
\vcbox{\includegraphics{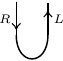}}
\vcbox{ and }
\vcbox{\includegraphics{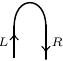}}
\vcbox{ , }
\end{center}
respectively, the string diagram form of the equations \eqref{eq:snakes} then is
\begin{center}
~\par
\vcbox{\includegraphics{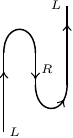}}
\vcbox{$=$}
\vcbox{\includegraphics{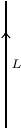}}
\vcbox{\hspace*{3em}}
\vcbox{\includegraphics{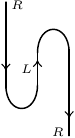}}
\vcbox{$=$}
\vcbox{\includegraphics{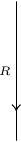}}
\end{center}
This is why we refer to \eqref{eq:snakes} as the \emph{snake identities}.
The morphisms $\eta$ and $\varepsilon$ are called {\em coevaluation} (or \emph{cup}) and
{\em evaluation} (or \emph{cap}), respectively. We say that a quadruple
$(L,R,\eta,\varepsilon)$ satisfying \eqref{eq:snakes} is \emph{a duality situation}.

Let us note that in every duality situation, $\eta$ determines $\varepsilon$ and vice versa.

A symmetric monoidal category $\C$ is {\em compact closed} 
if each of its objects has a chosen duality situation $(A,A^*,\eta_A,\varepsilon_A)$. 
The object-level mapping $A\mapsto A^*$ can be made into a functor $\C\to\C^{op}$: for
a morphism $f\colon A\to B$ we define $f^*\colon B^*\to A^*$ as
\[
f^*:=(\id_{A^*}\otimes\varepsilon_B)\circ(\id_{A^*}\otimes f\otimes\id_{B^*})\circ
    (\eta_A\otimes\id_{B^*}).
\]
For every triple of objects $A,B,C$ in a compact closed category $\C$, we have
$\C(A\otimes B,C)\simeq\C(A,B^*\otimes C)$. Therefore, every compact closed category
is a closed monoidal category.

A {\em dagger category} is a category $\C$ 
equipped with a functor $\dag\colon\C\to\C^{op}$ that is the identity
on objects and satisfies $f^{\dag\dag}=f$ for every morphism $f$ of $\C$. In
fact, the $\dag$ functor can be characterized as a mapping on the class of morphisms
of $\C$ that has the following properties:
\begin{itemize}
\item $(\id_A)^\dag=\id_A$
\item $(f\circ g )^\dag=g^\dag\circ f^\dag$
\item $f^{\dag\dag}=f$
\end{itemize}

We say that a morphism $f$ in a dagger category is \emph{self-adjoint} if and only if
$f^\dagger=f$.

\begin{definition}
\label{def:dagsymmonoidal}
Let $(\C,\otimes,I)$ be a symmetric monoidal category that is at the same time
a dagger category. We say that $\C$ is a {\em dagger symmetric monoidal category}
if the following equalities are satisfied, for all objects $A,B,C,D$ and
morphisms $f\colon A\to B$ and $g\colon C\to D$.
\begin{itemize}
\item $(f\otimes g)^\dag=f^\dag\otimes g^\dag$
\item $\alpha_{A,B,C}^{-1}=\alpha_{A,B,C}^\dag$
\item $\lambda_A^{-1}=\lambda_A^\dag$
\item $\rho_A^{-1}=\rho_A^\dag$
\item $\sigma_{A,B}^{-1}=\sigma_{A,B}^\dag$
\end{itemize}
\end{definition}

\begin{definition}
A dagger symmetric monoidal category $(\C,\otimes,I)$ is a {\em dagger compact category} if
 $\C$ is a compact closed category and the following diagram commutes
\begin{equation}
\label{diag:dagcompact}
\xymatrix{
I
	\ar[r]^-{\varepsilon_A^{\dag}}
	\ar[rd]_{\eta_A}
&
A\otimes A^*
	\ar[d]^{\sigma_{A,A^*}}
\\
~
&
A^*\otimes A
}
\end{equation}

\end{definition}

\begin{example}\cite[Example 3.58]{heunen2019categories}
\begin{itemize}
\item The category of finite dimensional Hilbert spaces $\FHilb$, with linear maps as
morphisms, is dagger compact.
The monoidal structure is the tensor product.
For an object $X$, $X^*$ is the dual space, $\varepsilon_X\colon X\otimes X^*\to\mathbb
C$ is the evaluation of a linear functional and this already determines the $\eta_X$.
For every linear mapping $f\colon X\to Y$, 
the mapping $f^*$ is the usual $Y^*\to X^*$ given by precomposition of linear functionals in $Y^*$ by $f$, and
$f^\dag\colon Y\to X$ is uniquely determined by the rule
$\langle f(x)|y\rangle=\langle x|f^\dag(y)\rangle$, for all $x\in X$ and $y\in Y$.
\item The category of sets and relations $\Rel$ is dagger compact. The monoidal
structure is the product of sets. For an object $X$, $X^*$ is equal to $X$ and 
the relation $\varepsilon_X\colon X\otimes X^*\to I$ is given by the rule 
\[
\varepsilon_X((a,b),*)\iff a=b .
\]
The relation $\eta_X$ is given similarly.
For a relation $f\colon X\to Y$, $f^\dag=f^*$ is just the opposite relation.
\end{itemize}
\end{example}

\subsection{Monoids and comonoids}

Let $(\C,\otimes,I)$ be a monoidal category. A {\em monoid} in $\C$ is
a triple $(S,\nabla,\unit)$, where $S$ is an object of $\C$, $\unit\colon I\to S$ is the
{\em unit} and
$\nabla:S\otimes S\to S$ is the {\em multiplication}, such that the equations
\begin{equation}\label{eq:monoid}
\nabla\circ(e\otimes\id_S)=\lambda_S
\quad
\nabla\circ(\id_S\otimes e)=\rho_S
\quad
\nabla\circ(\nabla\otimes\id_S)=\nabla\circ(\id_S\otimes\nabla)
\end{equation}
are satisfied.
The two left equations in \eqref{eq:monoid} are called
\emph{left and right unit laws}, the right one is the \emph{associative law}.

In string diagrams, the equations can be expressed as

\begin{center}
\vskip 0.5em
\vcbox{\includegraphics{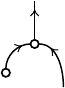}}
\vcbox{ $=$ }
\vcbox{\includegraphics{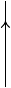}}
\vcbox{ $=$ }
\vcbox{\includegraphics{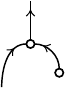}}
\hspace*{3em}
\vcbox{\includegraphics{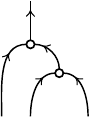}}
\vcbox{ $=$ }
\vcbox{\includegraphics{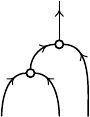}}
\vskip 0.5em
\end{center}

A \emph{comonoid} $(S,\Delta,c)$ is defined in a dual way:
\begin{equation}\label{eq:comonoid}
(c\otimes\id_S)\circ\Delta=\lambda_S
\quad
(\id_S\otimes c)\circ\Delta=\rho_S
\quad
(\id_S\otimes\Delta)\circ\Delta=(\Delta\otimes\id_S)\circ\Delta
\end{equation}
\begin{center}
\vskip 0.5em
\vcbox{\includegraphics{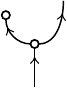}}
\vcbox{ $=$ }
\vcbox{\includegraphics{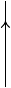}}
\vcbox{ $=$ }
\vcbox{\includegraphics{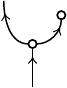}}
\hspace*{3em}
\vcbox{\includegraphics{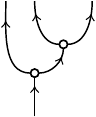}}
\vcbox{ $=$ }
\vcbox{\includegraphics{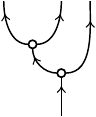}}
\vskip 0.5em
\end{center}
here $c\colon S\to I$ is called the {\em counit} and $\Delta\colon S\to S\otimes S$
is called the {\em comultiplication}.

Let $(S,\nabla_S,e_S)$, $(T,\nabla_T,e_T)$ be monoids in a monoidal category. A
\emph{monoid morphism}
is a morphism $f\colon S\to T$ preserving the multiplication and the unit:
\[
\nabla_T\circ(f\otimes f)=f\circ\nabla_S\qquad f\circ e_S=e_T
\]
\begin{center}
\vskip 0.5em
\vcbox{\includegraphics{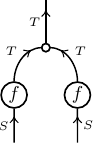}}
\vcbox{ $=$ }
\vcbox{\includegraphics{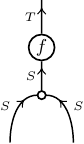}}
\hspace*{3em}
\vcbox{\includegraphics{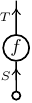}}
\vcbox{ $=$ }
\vcbox{\includegraphics{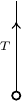}}
\vspace*{0.5em}
\end{center}

An invertible monoid morphism is called a \emph{monoid isomorphism}.
An invertible monoid morphism with the same domain and codomain is called a
\emph{monoid automorphism}.
A morphism that preserves just the multiplication and not necessarily the unit is
called a \emph{semigroup morphism}, or a \emph{nonunital morphism}.

\begin{proposal}
\begin{example}~
\begin{itemize} 
\item A monoid in the category of sets and mappings $\Set$, equipped with direct
products of sets is an (ordinary) monoid.
\item A monoid in the category of sets and relations $\Rel$, equipped with direct
product of sets is called relational monoid. Examples include all hypergroups
\cite{zieschang2023hypergroups}, all small groupoids \cite{heunen2013relative} and
(more generally) all small categories \cite{kenney2011categories}.
\item A monoid in the category of finite dimensional complex vector spaces
$\FinVecC$, equipped with tensor product is a finite dimensional algebra.
\end{itemize}
\end{example}
\end{proposal}

\subsection{Frobenius structures}

A {\em Frobenius structure} in a symmetric monoidal category $(\C,\otimes,I)$ is an object
$S$ equipped with a monoid structure $(S,\nabla,e)$ and a comonoid structure
$(S,\Delta,c)$ such that the equation
\begin{equation}
(\nabla\otimes\id_S)\circ(\id_S\otimes\Delta)
=\Delta\circ\nabla=(\id_S\otimes\nabla)\circ(\Delta\otimes\id_S)
\end{equation}
is satisfied.
In string diagrams, this is expressed as
\begin{center}
\vcbox{\includegraphics{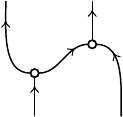}}
\vcbox{ $=$ }
\vcbox{\includegraphics{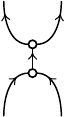}}
\vcbox{ $=$ }
\vcbox{\includegraphics{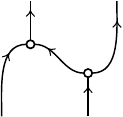}}
\end{center}

\begin{proposal}
\begin{example}
Frobenius structures in $\FinVecC$ are exactly the Frobenius algebras,
see for example \cite{kock2003frobenius}.
\end{example}
\end{proposal}

A Frobenius structure in a dagger symmetric monoidal category is a {\em dagger Frobenius
structure} if $\nabla=\Delta^\dag$ and $e=c^\dag$. Clearly, every dagger Frobenius
structure is completely determined by its monoid (or comonoid) structure.

Let us call $c\colon S\to I$ on a monoid $S$ a 
\emph{nondegenerate form} if there is a
(necessarily unique) coevaluation $\eta\colon I\to S\otimes S$ such that
$(S,S,\eta,c\circ\nabla)$ is a duality situation:
\begin{center}
\vcbox{\includegraphics{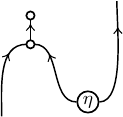}}
\vcbox{ $=$ }
\vcbox{\includegraphics{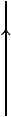}}
\vcbox{ $=$ }
\vcbox{\includegraphics{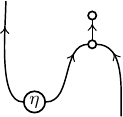}}
\end{center}
\begin{theorem}\cite[Proposition 5.16]{heunen2019categories}. \label{thm:nondeg} 
Let $(S,\nabla,e)$ be a monoid in a symmetric monoidal category. 
There is a bijective correspondence between nondegenerate forms on $S$ and
comonoids $(S,\Delta,c)$ making the monoid-comonoid pair into a Frobenius structure.
\end{theorem}

The bijective correspondence goes as follows. For a given Frobenius structure, the nondegenerate form is simply the counit of the
comonoid part, the self-duality situation is then
$(S,S,\Delta\circ e,c\circ\nabla)$.
On the other hand, given a nondegenerate form $c$ on a monoid $S$, the related
coevaluation $\eta$ determines a comultiplication
\[
\Delta=(\nabla\otimes\id_S)\circ(\id_S\otimes\eta)\circ\rho_S^{-1}
\]
\begin{center}
\includegraphics{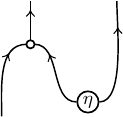}
\end{center}
which gives us a comonoid $(S,\Delta,c)$ such that the Frobenius law for
$\nabla$ and $\Delta$ is satisfied.

It is easy to check that, in a dagger compact category, the existence of the
duality situation $(S,S,\eta,c\circ\nabla)$ is equivalent to the fact that the morphism 
\[
(\id_{S^*}\otimes c)\circ(\id_{S^*}\otimes\nabla)\circ(\eta_S\otimes\id_S)\colon
I\otimes S\to  S^*\otimes I
\]
\begin{center}
\includegraphics{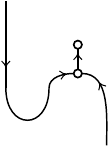}
\end{center}
is an isomorphism.

We will use some additional properties of Frobenius structures 
that are considered in the literature. We say that a Frobenius structure
$S$ is \emph{special} if there is some scalar $t\colon I\to I$ such that
\[
\nabla\circ\Delta=\lambda_S\circ(t\otimes\id_S)\circ\lambda_S^{-1}.
\]
In string diagrams, this means
\begin{center}
\vcbox{\includegraphics{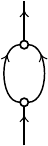}}
\vcbox{ $=$ }
\vcbox{\includegraphics{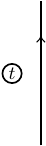}}
\end{center}
Let us remark that in the book \cite{heunen2019categories}, it is required
that $t=1$ for a Frobenius structure to be special.

\begin{example}\cite[Lemma 5.9]{heunen2019categories}\label{ex:pants}
For every object $X$ in a dagger compact category, there is a ``pair of pants'' Frobenius
structure on $X^*\otimes X$, with multiplication, comultiplication, unit and counit
given by the diagrams
\begin{center}
\vcbox{\includegraphics{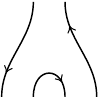}}
\vcbox{ $\quad$ }
\vcbox{\includegraphics{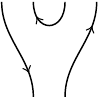}}
\vcbox{ $\quad$ }
\vcbox{\includegraphics{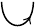}}
\vcbox{ $\quad$ }
\vcbox{\includegraphics{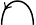}}
\vcbox{ $\quad$ }
\end{center}
\end{example}
\begin{example}\label{ex:groupoids}
\cite{heunen2013relative}
A \emph{groupoid} is a category in which every morphism is invertible.
There is a one-to-one correspondence between small groupoids and special dagger Frobenius
structures in $\Rel$.
For a groupoid $\mathbb G$, the underlying set of the dagger Frobenius structure 
is the set of all morphisms denoted by $\mathbb G_1$, the multiplication is the partial operation
of composition of morphisms, and the unit is the relation $\{*\}\to\mathbb G_1$ that
selects the set of all unit morphisms 
in $\mathbb G_1$.
\end{example}

We say that a Frobenius structure $S$ is \emph{symmetric} if
\[
c\circ\nabla\circ\sigma_{S,S}=c\circ\nabla,
\]
or, in string diagrams,
\begin{center}
\vcbox{\includegraphics{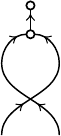}}
\vcbox{ $=$ }
\vcbox{\includegraphics{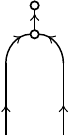}}
\end{center}
\begin{definition}
An $H^*$-algebra is an algebra $A$ over $\mathbb C$ equipped with an anti-linear
involution $\dag\colon A\to A$, that is also a Hilbert
space satisfying
\[
\langle ab|c\rangle = \langle b|a^\dag c\rangle=\langle a|cb^\dag \rangle.
\]
\end{definition}
\begin{example}\label{ex:hStarAlgebras}
\cite{vicary2011acategorical} and \cite[Theorem 5.32]{heunen2019categories}
There is a one-to-one correspondence between finite dimensional $H^*$-algebras and
symmetric dagger Frobenius structures in $\FHilb$.
\end{example}

Every Frobenius structure $S$ in a compact closed category admits an automorphism
$\alpha\colon S\to S$ of its monoid reduct that is completely
determined by the property
\begin{equation}\label{eq:nakayama}
c\circ\nabla\circ(\id_S\otimes\alpha)\circ\sigma_{S,S}=c\circ\nabla
\end{equation}
\begin{center}
\vcbox{\includegraphics{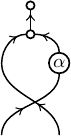}}
\vcbox{ $=$ }
\vcbox{\includegraphics{symmetric_symmetric_nosigma.pdf}}
\end{center}
We say that $\alpha$ is the \emph{Nakayama automorphism}
\cite{nakayama1941frobeniusean} of $S$. Note that
$S$ is symmetric if and only if $\alpha$ is the identity.

We say that a Frobenius structure $S$ is \emph{extra} if $c\circ e=\id_I$.

\subsection{The categories $\MatS$}

A \emph{commutative semiring} is a set $\Sigma$ equipped with a commutative monoid $(\Sigma, 0, +)$ and a commutative monoid $(\Sigma, 1, \cdot)$, such that $0$ is the zero element with respect to multiplication and multiplication distributes over addition. Let $\Sigma$ be a commutative semiring with unit, with $0 \neq 1$.
The category $\MatS$ is the
category whose objects are finite sets and whose morphisms from a set $X$ to a set $Y$ are matrices of elements of
$\Sigma$ with columns indexed by the set $X$ and rows indexed by the set $Y$. For $f\colon X\to Y$, we denote the
element of the matrix $f$ corresponding to a pair $x\in X$, $y\in Y$ by $f_x^y$.

The composition of morphisms
is the usual matrix multiplication: for $f\colon A\to B$ and
$g\colon B\to C$, we have
\[
(g\circ f)_a^c=\sum_{b\in B}f_a^b g_b^c,
\]
and the identity morphism $\id_A\colon A\to A$ is the unit
matrix. In computations, we shall write $\delta_a^{a'}$ for the entries 
of $\id_A$, this is simply the Kronecker delta:
\[
\delta_a^{a'}=
\begin{cases}
1&a=a'\\
0&\text{otherwise.}
\end{cases}
\]

There is a faithful functor $M$ from the category of finite sets $\FinSet$
to $\MatS$ that is identity on objects,
and for each morphism $f\colon A\to B$
\begin{equation}
M(f)_a^b=
\begin{cases}
1&b=f(a)\\
0&\text{otherwise}
\end{cases}
\end{equation}
for every $a\in A$, $b\in B$. In this paper, we identify $M(f)$ with $f$
whenever there is no danger of confusion.

For a morphism $f\colon A\to B$, the \emph{support of $f$} is the subset of $A$ given
by
\[
\supp(f)=\{a\in A:\text{$f_a^b\neq 0$ for some $b\in B$}\}.
\]
The \emph{support endomorphism of} $f$ is the matrix $\based f\colon A\to A$ given by
\begin{equation}\label{eq:supportendo}
{(\based f)}_a^{a'}=
\begin{cases}
1 & a=a'\text{ and $a\in \supp(f)$}\\
0 & \text{otherwise.}
\end{cases}
\end{equation}

Note that $\based f$ is an idempotent endomorphism of $A$ and that $f\circ\based f=f$.

The category $\MatS$ can be equipped with a symmetric monoidal structure
$(\MatS,\otimes,I,\alpha,\lambda,\rho,\sigma)$, where 
\begin{itemize}
\item $I$ is a one-element set (we write $I=\{*\}$ so that $*$ is the unique element of
the unit object $I$),
\item for finite sets $X,Y$, $X\otimes Y=X\times Y$,
\item for morphisms $f\colon A\to B$ and $g\colon C\to D$, $f\otimes g\colon A\otimes C\to B\otimes D$ is
the tensor (or Kronecker) product of matrices:
\[
(f\otimes g)_{ac}^{bd}=f_a^b\cdot g_c^d
\]
\item The $\alpha$, $\lambda$, $\rho$ and $\sigma$ are transferred from the canonical
cartesian monoidal structure on $\FinSet$ via the $M$ functor.
\end{itemize}

Let $\mathbbm 2=\{0,1\}$ denote the two-element totally ordered set, with
$0<1$.  Clearly, the bounded distributive lattice $(\mathbbm 2,0,1,\vee,\wedge)$ is a semiring and it is easy to
see that the category $\MatRel$ is the same thing as the category of finite
sets and relations $\FinRel$, a full subcategory of the category of sets and
relations $\Rel$.

In a very similar way to $\Rel$, we can equip $\MatS$ with a dagger compact
structure, as follows.  For every object $X$, we put $X^*=X$ and  define
$\eta_X\colon I\to X^*\otimes X$, $\varepsilon_X\colon X\otimes X^*\to I$ by
$(\eta_X)_{*}^{ab}=\delta_a^b$ and $(\varepsilon_X)_{ab}^*=\delta_a^b$. For every
morphism $f\colon A\to B$, it turns out that $f^*$ is simply the transpose of the
matrix $f$. We then put $f^\dagger=f^*$.

\subsection{The category $\MatR$}\label{sec:2.5}

Our basic category will be $\MatR$, where $\mathbb R_0^+:=[0,\infty)_{\mathbb R}$ is the semiring of
nonnegative reals equipped with the usual addition and multiplication.  Morphisms in
$\MatR$ are matrices with nonnegative real entries; these will be called nonnegative
matrices. The restriction to nonnegative numbers has remarkable effects. For
example, isomorphisms in $\MatR$ can be characterized in a very simple way, as the
following theorem shows.
\begin{theorem}\cite{ding2014matrix}
A nonnegative matrix has a nonnegative inverse if and only if it is the product
of a diagonal matrix with strictly  positive diagonal entries and a permutation matrix.
\end{theorem}
As a consequence, we have the following.
\begin{corollary}\label{coro:chariso}
In $\MatR$, $f\colon A\to B$ is an isomorphism if and only if
for every $a\in A$ there is a unique $b\in B$ such that $f_a^b>0$ and
for every $b\in B$ there is a unique $a\in A$ such that $f_a^b>0$.
\end{corollary}
Therefore, two objects $X,Y$ of $\MatR$ are isomorphic if and only if
$|X|=|Y|$ and if we want to construct all isomorphisms in the homset
$\MatR(X,Y)$, we may simply take all permutation matrices and then replace the $1$'s in the matrix by 
arbitrary strictly positive real numbers.

Since
$\mathbb R_0^+$ is a partially ordered semiring, $\MatR$ can be enriched over
$\Pos$ by ordering the matrices in the homsets element-wise. But this 2-structure is 
uninteresting: for example, adjunctions are only the isomorphisms, and monads
are only the identity morphisms.

Nevertheless, there is another 2-structure on $\MatR$. Consider the functor
$R\colon\MatR\to\FinRel$ that is identity on objects and 
we have
\[
R(f)_a^b=1\iff f_a^b>0,
\]
for a morphism $f\colon A\to B$ in $\MatR(A,B)$.
We may then preorder the morphisms in $\MatR(A,B)$ by the rule
$f\lesssim g\iff R(f)\leq R(g)$ or, equivalently,
\[
f\lesssim g\iff (\forall a\in A,b\in B: f_a^b>0\implies g_a^b>0).
\]
Note that an isomorphism class with respect to this 2-structure is quite large, for example, all diagonal matrices 
$f\colon A\to A$ that have all of their diagonal entries nonzero are isomorphic to
each other.

By equipping $\MatR$ with such a 2-structure, we upgrade the functor $R\colon\MatR\to\FinRel$ to a 2-functor.
The internal left adjoints in the 2-structure on $\MatR$ are then those morphisms
that are mapped to left adjoints (i.e. mappings) in $\FinRel$ by the functor $R$.
In detail, $f\colon A\to B$ is a left adjoint in $\MatR$ if and only if for every $a\in A$ there
is exactly one $b\in B$ such that $f_a^b>0$. We will call the left adjoints
in $\MatR$ \emph{internal maps} and denote the corresponding subcategory
by $\MapR$. Note that $f$ is an internal map if and only if $R(f)$ is a map.

The functor $R\colon\MatR\to\FinRel$ has very strong properties:
it is identity on objects, full, and strictly monoidal (hence also comonoidal).
Clearly, it also preserves the dagger compact structure. 
Note that \Cref{coro:chariso} can be simply phrased as ``$R$ reflects isomorphisms''.

\subsection{Monoids in $\MatR$}

Let us rewrite the associative law for a monoid $S$ in $\MatR$ in a concise form. For all
$a,b,c,d\in S$, we have
\[
(\nabla\circ(\nabla\otimes\id))_{abc}^d=
\sum_{(p,c')\in S\otimes S}\nabla_{pc'}^d\nabla_{ab}^p\delta_{c}^{c'}=
\sum_{p\in S}\nabla_{pc}^d\nabla_{ab}^p
\]
and
\[
(\nabla\circ(\id\otimes\nabla))_{abc}^d=
\sum_{(a',q)\in S\otimes S}\nabla_{a'q}^d\delta_a^{a'}\nabla_{bc}^{q}=
\sum_{q\in S}\nabla_{aq}^d\nabla_{bc}^q
\]
hence, the associative law means that
\begin{equation}\label{eq:assoc}
\sum_{p\in S}\nabla_{pc}^d\nabla_{ab}^p=
\sum_{q\in S}\nabla_{aq}^d\nabla_{bc}^q,
\end{equation}
for all $(a,b,c)\in S\otimes S\otimes S$ and $d\in S$.

For a monoid $(S,\nabla,\unit)$ in $\MatR$, we denote
$$
E_S=\{p\in S:e_*^p>0\}
$$
and refer to the elements of $E_S$ as \emph{units}. 
The unit laws mean that, for all $a,c\in S$, 
\begin{equation}\label{eq:unitsMatR}
\sum_{p\in E_S}e_*^p\nabla_{pa}^c=\delta_a^c=\sum_{q\in E_S}e_*^q\nabla_{aq}^c.
\end{equation}
The following two lemmas establish basic properties of units in a monoid in $\MatR$.
\begin{lemma}\label{lemma:units}
Let $(S,\nabla,\unit)$ be a monoid in $\MatR$. For every $a\in S$, we have
the following.
\begin{enumerate}
\item[(a)] There exists $p\in E_S$ such that $\nabla_{pa}^c>0$ for some $c\in S$.
\item[(b)]If $p\in E_S$, $\nabla_{pa}^c>0$, then $a=c$.
\item[(a')] There exists $p\in E_S$ such that $\nabla_{ap}^c>0$ for some $c\in S$.
\item[(b')]If $p\in E_S$, $\nabla_{ap}^c>0$, then $a=c$.
\item[(c)]The $p\in E_S$ with $\nabla_{pa}^a>0$ is unique and
\begin{equation}\label{eq:valueoflunit}
\nabla_{pa}^a=\frac{1}{e_*^p}\text{.}
\end{equation}
\item[(c')]The $p\in E_S$ with $\nabla_{ap}^a>0$ is unique and
\begin{equation}\label{eq:valueofrunit}
\nabla_{ap}^a=\frac{1}{e_*^p}\text{.}
\end{equation}
\end{enumerate}
\end{lemma}
\begin{proof}
By \eqref{eq:unitsMatR},
\begin{equation}\label{eq:leftunit}
\sum_{p\in E_S}e_*^p\nabla_{pa}^c=\delta_a^c
\end{equation}
for all $a,c\in S$.
The left-hand side expression here is positive if and only if $\nabla_{pa}^c>0$ for some $p\in E_S$, and 
$\delta_a^c$ is nonzero if and only if $a=c$. This proves simultaneously (a) and (b).
The proof of (a') and (b') is analogous.

To prove (c), suppose that $p,q\in E_S$ are such that $\nabla_{pa}^a>0$,
$\nabla_{qa}^a>0$.
By the associative law, for all $a\in S$,
\begin{equation}\label{eq:assocunit}
\sum_{u\in S}\nabla_{ua}^a\nabla_{pq}^u=\sum_{r\in S}\nabla_{pr}^a\nabla_{qa}^r.
\end{equation}
By (a) and (b), $r=a$ is the unique $r\in S$ for which $\nabla_{qa}^r>0$.
Hence the right-hand side of \eqref{eq:assocunit} equals $\nabla_{pa}^a\nabla_{qa}^a$
and, by our assumptions, is positive. Therefore, the left-hand side of
\eqref{eq:assocunit} is positive as well and, in particular, $\nabla_{pq}^u>0$ for
some $u\in S$. Since
$p\in E_S$, $q=u$ by (b). Since $q\in E_S$, $u=p$ by (b'). Therefore, $p=q$.
By the uniqueness of $p$ and the fact that 
\eqref{eq:leftunit} equals $1$ for $a=c$, we see that $e^p_*\nabla_{pa}^a=1$,
which yields \eqref{eq:valueoflunit}.

The proof of (c') is (up to a symmetry) the same as the proof of (c).
\end{proof}
By \Cref{lemma:units}, for every element $a$ of a monoid $S$ in $\MatR$, there is a unique
$s(a)\in E_S$ such that $\nabla_{s(a)a}^a>0$ and a unique $t(a)\in E_S$ such that
$\nabla_{at(a)}^a>0$. We say that $s(a)$ and $t(a)$ are the \emph{source and target
of $a$}, respectively. Our intuition is that a monoid in $\MatR$ is a
generalized version of a finite category, the latter viewed as a monoid in $\FinRel$
as in \cite{kenney2011categories,heunen2013relative}. We think of the elements of $E_S$ as ``identities/unit
arrows/objects'' and $a$ as an ``arrow/morphism'' from $s(a)$ to $t(a)$. This
analogy is supported by the following lemma.

\begin{lemma}\label{lemma:sourcesTargets}
Let $(S,\nabla,e)$ be a monoid in $\MatR$, let $a,b,c\in S$ be such that
$\nabla_{ab}^c>0$. Then $t(a)=s(b)$, $s(a)=s(c)$ and $t(b)=t(c)$.
\end{lemma}
\begin{proof}
By associativity,
\[
\sum_{p\in S}\nabla_{at(a)}^p\nabla_{pb}^c=
\sum_{q\in S}\nabla_{aq}^c\nabla_{t(a)b}^q
\]
As $t(a)\in E_S$, \Cref{lemma:units} implies that $\nabla_{at(a)}^p>0$ if and only if $p=a$. Hence
the first sum equals $\nabla_{at(a)}^a\nabla_{ab}^c>0$. Therefore,
the second sum is positive so that  $\nabla_{t(a)b}^q>0$ for some
$q\in S$. By \Cref{lemma:units}, we see that $q=b$ and that $t(a)=s(b)$. The proofs of
$s(a)=s(c)$ and $t(b)=t(c)$ are very similar and are thus omitted.
\end{proof}

\begin{proposal}
\subsection{The functor to finite dimensional Hilbert spaces}
\label{sec:Lfunctor}

Recall $\FHilb$, the dagger compact category of finite dimensional
complex Hilbert spaces equipped with the usual tensor product. There is
a familiar looking functor $L\colon\MatR\to\FHilb$ given by the rule
$L(X)={\mathbb C}^X$ on objects, where $\mathbb C^X$ is equipped with the
usual inner product. On morphisms, $L$ takes each matrix in $\MatR$ to a linear
map corresponding to it with respect to the canonical bases. It is easy to see that
the functor $L$ preserves both the compact and the dagger structure.

The monoids in $\FHilb$ are finite dimensional complex algebras, equipped with an
inner product on the underlying vector space. It is a consequence of standard facts
about bases, linear mappings, and the tensor product that the functor $L$ is strictly
monoidal, hence also strictly comonoidal. From this, it follows that the image of a
(co)monoid is always a (co)algebra and it is easy to see that the image of a
Frobenius structure in $\MatR$ is always a Frobenius algebra. Moreover, $L$ maps
monoid morphisms in $\MatR$ to morphisms of algebras in $\FHilb$ and modules to
modules.

In some cases, we will suppress $L$ from our notation and we will treat $f$ and
$L(f)$ as the same thing.
\end{proposal}
\begin{proposal}
\subsection{Connection to classical algebras}

A monoid in $\MatR$ can be equivalently described as a finite dimensional (real or complex)
algebra $A$ with a distinguished basis $B\subseteq A$ such that the unit of $A$ and the
product of any two elements of $B$ always have nonnegative real coordinates in the basis $B$.
Algebras of this type, equipped with additional structure, are studied under various
names; see \cite{blau2009table} for a survey.

\end{proposal}

\subsection{Coherent configurations and association schemes}

The following definition seems to introduce a notational collision  with our use of
$\nabla$, $E_S$, $s$, and $t$ for monoids in $\MatR$. However, in the next section, it
will turn out that those notations are actually consistent.

\begin{definition}\cite{higman1975coherent}
\label{def:cc}
Let $X$ be a finite set, let $S$ be a system of subsets of $X\times X$
such that
\begin{enumerate}
\item[(CC1)] $S$ is a partition of $X\times X$.
\item[(CC2)] If $p\in S$ and $(x,x)\in p$, then $p$ is a subset of the identity relation $\id_X$.
\item[(CC3)]For $a,b,c\in S$  
 and $(x,y)\in c$, the number $\nabla_{ab}^c$ of $z\in X$ such that $(x,z)\in a$ and $(z,y)\in b$ does not depend on the choice of $(x,y)\in c$.
\item[(CC4)]If $c\in S$, then the opposite relation $c^{-1}=\{(y,x)\colon (x,y)\in c\}\in S$.
\end{enumerate}
Then $(X,S)$ is called {\em a coherent configuration on $X$}. The elements $p\in S$
with $p\subseteq\id_X$ are called {\em units}. The set
of all units is denoted by
$E_{S}$. The numbers $\nabla_{ab}^c$ are called \emph{the structure constants} of
the coherent configuration $(X,S)$. The cardinality of the set $X$ is the
\emph{degree} (some authors call it \emph{order}) of $(X,S)$ and the
cardinality of the set $S$ is the \emph{rank} of $(X,S)$. The elements of $X$ are
\emph{vertices}. 
\end{definition}

It is instructive to visualize a coherent configuration $(X,S)$ as a complete digraph
on a set $X$, with edges colored by elements of the set $S$. The elements of
$S$ are therefore sometimes called \emph{colors} and the elements of $E_S\subseteq S$
are called \emph{loop colors} or \emph{unit colors}. The number $\nabla_{ab}^c$ is then the number of $(a,b)$
colored walks from the source vertex $x$ of any $c$-colored edge $(x,y)\in c$ to its target vertex $y$:
\[
\xymatrix{
~
&
\cdot
    \ar[rd]^b
\\
x
    \ar[rr]_c
    \ar[ur]^a
&
~
&
y
}
\]
which we sometimes call the $(a,b)$-triangles over $(x,y)$. 

For every color $a$ of a coherent configuration $(X,S)$, there is a unique unit color
$q\in E_S$ such that $\nabla_{aa^{-1}}^q>0$. 
This $q$ is called \emph{the source of $a$} and is denoted by $s(a)$. 
The source of $a^{-1}$ is called \emph{the target of $a$} and we denote it by $t(a)$.

The number $\nabla_{aa^{-1}}^{s(a)}$ is called
\emph{the valency of} $a$  and is denoted by $\val a$. Note that for every $x$ such that
$(x,y)\in a$ for some $y\in X$,
$\val a$ is the cardinality of the set
\[
\{y\in X:(x,y)\in a\}.
\]

\begin{proposal}
\begin{remark}
In this paper, we are slightly diverging from the prevailing notations and
terminology. The usual
notation for the opposite relation is $c^*$; this collides with the standard
notation for duals in a compact category. The usual notation for valency is $n_a$;
we think that $\val a$ is more readable. We also use the term \emph{unit} for
elements of $E_S$; these are called \emph{fibers} in \cite{hanaki2015thin}; other
authors \cite{higman1975coherent,chen2024lectures} call the sets $\Delta\subseteq X$
with $\id_\Delta\in E_S$ fibers, and the relations $\id_\Delta$ \emph{reflexive basis
relations}. We think that \emph{unit}
fits our more general viewpoint of seeing $S$ as a particular monoid in
$\MatR$ better, see \Cref{lemma:units}.
\end{remark}
\end{proposal}

We say that $(X,S)$ is \emph{homogeneous}, 
or an \emph{association scheme} or simply a \emph{scheme} if $\id_X\in S$, that
means, if there is only one unit color. Let us remark that some authors
\cite{godsil2010association,bailey2004association} assume that every association
scheme is commutative; we do not make this assumption, following, for example,
\cite{zieschang2005theory}.
We say that $(X,S)$ is \emph{thin} (or \emph{regular}) if $\val{a}=1$ for every color $a\in
S$.
We say that $(X,S)$ is \emph{$\frac{1}{2}$-homogeneous} if every unit color (as a subset
of $X\times X$) has the same cardinality. Clearly, if $(X,S)$ is homogeneous or thin,
then $(X,S)$ is $\frac{1}{2}$-homogeneous. 

A coherent configuration $(X,S)$ is $\frac{1}{2}$-homogeneous
if and only if, for every $a\in S$, $\val{a}=\val{a^{-1}}$.
Indeed, for every 
color $a\in S$ the cardinality of $a\subseteq X\times X$ can be expressed as
\begin{equation}\label{eq:valencies}
|a|=|s(a)|\cdot\val{a}
\end{equation}
and it is clear from (CC4) that $|a|=|a^{-1}|$. Hence $|s(a)|=|s(a^{-1})|$ implies 
$\val{a}=\val{a^{-1}}$. For the opposite implication, whenever  $p,q\in E_S$, we may pick
some color $a$ with $s(a)=p$ and $t(a)=s(a^{-1})=q$ and then cancel the valencies from the
equation \eqref{eq:valencies}.

\begin{proposal}
To every coherent configuration $(X,S)$, one can associate a matrix algebra,
as follows.
\begin{definition}
Let $(X,S)$ be a coherent configuration. For every $a\in S$, let
$\sigma_a$ be the square $X\times X$ matrix given by
\[
(\sigma_a)_x^y=
\begin{cases}
1 & (x,y)\in a \\
0 & \text{otherwise}
\end{cases}
\]
The (complex) matrix algebra $\mathcal A(X,S)$ generated by the set of all these $\sigma_a$ is
called \emph{the adjacency algebra of $(X,S)$}.
\end{definition}
The matrices $\sigma_a$, $a\in S$, form a linear basis of $\mathcal A(X,S)$, such that
$\sigma_a^*=\sigma_{a^{-1}}$, $\sum_{p\in E_S}
\sigma_p=I_X$,
and $\nabla_{ab}^c$ are the structure constants of $\mathcal A$ with respect to this
basis:
\[
\sigma_a\sigma_b=\sum_c \nabla_{ab}^c \sigma_c.
\]
Like every matrix algebra, $\mathcal A(X,S)$ is equipped with a
canonical left module structure on the vector space ${\mathbb C}^X$.
\end{proposal}

\begin{example}
Let $X$ be a finite set. Put $S=\{\{(x,y)\}:(x,y)\in X\times X\}$. Then $\mathcal
D_X:=(X,S)$ is always a coherent configuration, called \emph{discrete coherent configuration on $X$}.
\end{example}
\begin{example}
The unique coherent configuration on a singleton $X=\{*\}$ is called the \emph{terminal
coherent configuration} and denoted simply by $1$.
\end{example}

Let $(\Gamma,\cdot,e)$ be a group acting on a finite set $X$ from the right, where
$\odot\colon X\times\Gamma\to X$ denotes the action.
We may extend the action to $X\times X$ in an obvious way:
\[
(x_1,x_2)\odot a:=(x_1\odot a,x_2\odot a).
\]
Clearly, this is an action of $\Gamma$, so it induces a decomposition
of $X\times X$ into orbits.
Each of these orbits is a subset of $X\times X$, that means, a relation
on the set $X$ (these relations are sometimes called \emph{orbitals}, or
\emph{2-orbits}). 
This way, from a group acting on a set
we obtain data in the form $(X,(X\times X)/\Gamma)$, where $(X\times X)/\Gamma$ is
the decomposition of $X\times X$ into 2-orbits. 
This system is always a coherent configuration. The 
coherent configurations arising from an action of a group in this way are called
\emph{Schurian coherent configurations}.
    \begin{figure}
        \centering
        \includegraphics[width=0.3\linewidth]{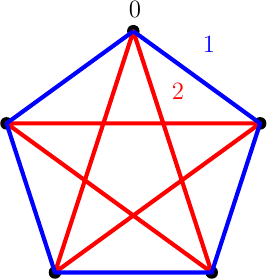}
        \caption{The Schurian scheme $D_5$}
        \label{fig:pentagon}
    \end{figure}
\begin{example}
Consider the dihedral group $D_5$ acting on the $5$ vertices of a regular pentagon in
the usual way. Write $X$ for the set of vertices. This situation
induces a decomposition of $X\times X$ into three classes: the set of all loops $l$ 
(this is the unique unit color, so we obtain an association scheme in this case), the
set of all diagonals $d$ and the set of all sides $f$, see \Cref{fig:pentagon}.
Every color is a
symmetric relation. Some values of the structure constants are
\begin{align*}
\nabla_{dd}^f&=1\\
\nabla_{dd}^d&=0\\
\nabla_{dd}^l&=2
\end{align*}
and the valencies are $\val{l}=1$, $\val{f}=\val{d}=2$.
\end{example}
\begin{example}\label{ex:cayleyScheme}
Every finite group $\Gamma$ acts on itself via multiplication. It is clear that the Schurian coherent configuration given
by this action is a thin association scheme. Moreover, every thin association scheme arises in this way.
\end{example}

\section{Coherent configurations and Frobenius structures}\label{sec:3}

Consider a pair $(X,S)$, where $X$ is a finite set equipped with a system $S$
of subsets such that $S$ is a partition of $X\times X$, that means, the
condition (CC1) of \Cref{def:cc} is satisfied. For each such a pair we may
construct a nonnegative matrix (a $\MatR$ morphism) $\gamma\colon X\otimes S\to X$ given by

\[
\gamma_{xa}^y=
\begin{cases}
1&\text{if $(x,y)\in a$}\\
0&\text{otherwise.}
\end{cases}
\]

Clearly, this matrix $\gamma$ faithfully represents the pair $(X,S)$.

A \emph{pointed magma} in $\MatR$ is a triple $(S,\nabla,\unit)$, 
where $S$ is a finite set, $\nabla\colon S\otimes S\to S$ and $\unit\colon I\to S$.

Note that for a pair $(X,S)$ satisfying (CC3), it is natural to consider the
structure constants $\nabla_{ab}^c$ as elements of a matrix $\nabla\colon
S\otimes S\to S$. Similarly, for $(X,S)$ satisfying (CC2), we may encode the
information about unit colors by a $\MatR$ morphism $\unit\colon I\to S$ given by

\begin{equation}\label{eq:defunit}
\unit_{*}^a
=
\begin{cases}
1 & \text{if $a\subseteq\id_X$}\\
0 & \text{otherwise.}
\end{cases}
\end{equation}

A \emph{module over a pointed magma $(S,\nabla,\unit)$} is a finite set $X$
and a morphism $\gamma\colon X\otimes S\to X$ such that the equations

\begin{equation}
\label{eq:module}
\gamma\circ(\gamma\otimes\id_S)=\gamma\circ(\id_X\otimes\nabla)
\qquad
\gamma\circ(\id_X\otimes e)=\rho_X
\end{equation}

are satisfied. In string diagrams, this is expressed by
\begin{center}
\vskip 0.5em
\vcbox{\includegraphics{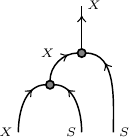}}
\vcbox{ $=$ }
\vcbox{\includegraphics{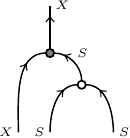}}
\hspace*{3em}
\vcbox{\includegraphics{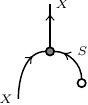}}
\vcbox{ $=$ }
\vcbox{\includegraphics{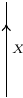}}
\vskip 0.5em
\end{center}

where the gray dot represents the $\gamma$.

\begin{theorem}\label{thm:monAction}
Let $X$ be a finite set, let $S$ be a partition of $X\times X$,
so that $(X,S)$ satisfies (CC1).
Let $\gamma\colon X\otimes S\to X$ be the $\MatR$-morphism
associated with $(X,S)$. 
Then $\gamma$ is a module over some pointed magma $(S,\nabla,\unit)$ if
and only if $(X,S)$ satisfies (CC2) and (CC3) of
\Cref{def:cc}.
Moreover, $(S,\nabla,e)$ is then a monoid.
\end{theorem}
\begin{proof}
Let us prove that $(X,S)$ satisfies (CC3) if and only if
\begin{equation}\label{eq:CC3}
\gamma\circ(\gamma\otimes\id_S)=\gamma\circ(\id_X\otimes\nabla).
\end{equation}
Suppose that we have some magma $\nabla\colon
S\otimes S\to S$ such that \eqref{eq:CC3} is satisfied. This means that,
for all $(x,a,b)\in X\otimes S\otimes S$ and $y\in X$,

\begin{equation}\label{eq:modulesqone}
(\gamma\circ(\id\otimes \nabla))_{xab}^y
    =\sum_{(z,c)\in X\otimes S}\gamma_{zc}^y\delta_x^z\nabla_{ab}^c
    =\sum_{c\in S}\gamma_{xc}^y\nabla_{ab}^c
\end{equation}
is equal to
\begin{equation}\label{eq:modulesqtwo}
(\gamma\circ(\gamma\otimes\id))_{xab}^y
    =\sum_{(w,p)\in X\otimes S}\gamma_{wp}^y\gamma_{xa}^w\delta_b^p
    =\sum_{w\in X}\gamma_{wb}^y\gamma_{xa}^w.
\end{equation}

Consider the rightmost term in \eqref{eq:modulesqtwo} first. Obviously, it just
counts the number of $w\in X$ such that $(x,w)\in a$ and $(w,y)\in b$, that is, the number
of $(a,b)$-triangles over $(x,y)$. Further, in the
rightmost term of \eqref{eq:modulesqone}, there is exactly one $c\in S$ such that
$\gamma_{xc}^y$ is nonzero, namely $c=\<x,y\>$,  and then $\gamma_{xc}^y=1$, so
\eqref{eq:modulesqone} equals $\nabla_{ab}^c$, irrespective of $x$ and $y$.  Hence
\eqref{eq:CC3} is equivalent to (CC3), and $\nabla$ is necessarily given by the structure
constants.

Let us prove that the equality 
\begin{equation}\label{eq:CC2}
\gamma\circ(\id_X\otimes e)=\rho_X
\end{equation}
for some $\unit\colon I\to S$ is equivalent to the fact that $(X,S)$
satisfies (CC2). This equality  means that, for all $u\in X$
and $w\in X$

\begin{equation}\label{eq:moduletrone}
\delta_u^w=(\gamma\circ(\id_X\otimes\unit))_{u*}^w=
    \sum_{(v,p)\in X\otimes S}\gamma_{vp}^w\delta_u^v\unit_{*}^p=
    \sum_{p\in S}\gamma_{up}^w\unit_{*}^p.
\end{equation}
Note that, for all $u\in X$ and $w\in X$, the very definition
of $\gamma$ implies that 
\begin{equation}\label{eq:isnabla}
\sum_{p\in S}\gamma_{up}^w\unit_{*}^p=
\unit_*^c
\end{equation}
where $(u,w)\in c$. Hence equation \eqref{eq:CC2} is equivalent to
\begin{equation}\label{eq:CC2eq}
e_*^c=\delta_u^w,\qquad \forall (u,w)\in c.
\end{equation}
If this is true for some $e\colon I\to S$, then it is clear that $e_*^c\in\{0,1\}$ for all $c$,
moreover, $e_*^c=1$ if and only if $(u,u)\in c$ for some $u$, in which case we must have
$u=w$ for all $(u,w)\in c$. This means that (CC2) holds and $e$ must be given by
\eqref{eq:defunit}.

On the other hand, suppose that (CC2) is satisfied. We may then define
$\unit$ as in \eqref{eq:defunit}, which clearly satisfies \eqref{eq:CC2eq}.

Let us prove that $\nabla$ is associative; that is for all
$(a,b,c)\in S\otimes S\otimes S$, and $d\in S$, the equality \eqref{eq:assoc} is satisfied.
Fix any $(x,y)\in d$. Then $\nabla_{pc}^d$ is the number of $(p,c)$-triangles over
$(x,y)$. For each of those triangles, $\nabla_{ab}^p$ is the number of $(a,b)$-triangles
along the $p$ edge. Thus, the number $\nabla_{pc}^d\cdot\nabla_{ab}^p$ is the number of squares
of the form
$$
\xymatrix{
\cdot
	\ar[r]^b
&
\cdot
	\ar[d]^c
\\
x
	\ar[u]^a
	\ar[ru]^p
	\ar[r]_d
&
y
}
$$
The sum on the left hand side of \eqref{eq:assoc} then runs over all possible colors $p$ of the
diagonal. Therefore, the left hand side of \eqref{eq:assoc} is equal to the number of all squares of the form
\begin{equation}
\label{eq:squareassoc}
\xymatrix{
\cdot
	\ar[r]^b
&
\cdot
	\ar[d]^c
\\
x
	\ar[u]^a
	\ar[r]_d
&
y
}
\end{equation}
Note that this number does not depend on the particular choice of $(x,y)\in d$.
For $(x,y)\in d$, the number $\nabla_{aq}^d\cdot\nabla_{bc}^q$ is the number of squares of the form
$$
\xymatrix{
\cdot
	\ar[r]^b
	\ar[rd]^q
&
\cdot
	\ar[d]^c
\\
x
	\ar[u]^a
	\ar[r]_d
&
y
}
$$
Summing over all possible colors $q$ of the diagonal gives us the number of all squares of the
form \eqref{eq:squareassoc} and we have proved the associativity of $\nabla$.

Let us prove that $\unit$ is a right unit with respect to $\nabla$. By
\eqref{eq:unitsMatR}, this means that
\begin{equation}
\label{eq:runit}
\sum_{q\in E_S}\nabla_{aq}^b =\delta_a^b,\qquad \forall a,b\in S.
\end{equation} 
So let $a\in S$ and $(x,y)\in a$. Then for $q\in E_S$, $\nabla_{aq}^a>0$ if and only if
$(y,y)\in q$. Since there is a unique unit color with this property, and $(x,y,y)$ is the
unique $(a,q)$-colored walk from $x$ to $y$, the sum on the left hand side of
\eqref{eq:runit} is equal to 1 for $a=b$. 
If $\nabla_{aq}^b>0$ for $b\in S$, then $(y,y)\in q$ and $(x,y)\in b$. Since $S$ is a partition of
$X\times X$, $a\cap b\neq\emptyset$ implies that $a=b$. 

The proof that $\unit$ is a left unit is analogous.
\end{proof}

We have proved that every pair of finite sets satisfying (CC1)--(CC3) can be represented
as a module over a monoid in $\MatR$. We say that $(S,\nabla,e)$ is the 
\emph{associated monoid of the coherent configuration $(X,S)$} and that $\gamma$ is
\emph{the module of $(X,S)$}.

A natural question arises: what does (CC4) mean in this context?
This question is answered by the following theorem.

\begin{theorem}\label{thm:frobIsC4}
Let $(X,S)$ be a pair of finite sets satisfying (CC1)--(CC3). Let $f\colon  S\to I$ be such that
$f^*_a>0$ iff $a\subseteq \id_X$. 
Then $f$ is a nondegenerate form on the associated monoid 
$(S,\nabla,\unit)$ if and only if $(X,S)$ is a coherent configuration.
\end{theorem}
\begin{proof}
Suppose that $(X,S)$ is a coherent configuration. We need to prove that
there is some $\eta\colon I\to S\otimes S$ such that
$(S,S,\eta,f\circ\nabla)$ is a duality situation, that is, that the identities in
\eqref{eq:snakes} hold.  The first identity has the form
\begin{equation}\label{eq:leftfrobsnake}
\lambda_S\circ((f\circ\nabla)\otimes\id_S)\circ(\id_S\otimes\eta)\circ\rho_S^{-1}=\id_S.
\end{equation}
On the left-hand side, we have, for all $a,b\in S$,
\begin{equation}\label{eq:lsnakeelement}
(\lambda\circ(f\otimes\id)\circ(\nabla\otimes\id)\circ(\id\otimes\eta)\circ\rho^{-1})_a^b=
\sum_{p,q\in S}f_q^*\nabla_{ap}^q\eta_*^{pb}.
\end{equation}
Observe that $f_q^*\nabla_{ap}^q\eta_*^{pb}>0$ implies that $q\in E_S$ and
$\nabla^q_{ap}>0$. As $q$ is a unit color, this means that $p=a^{-1}$, $q=s(a)$, and
$\nabla^q_{ap}=\nabla^{s(a)}_{aa^{-1}}=\|a\|$. Hence the right-hand side of
\eqref{eq:lsnakeelement}  reads
\[
f_{s(a)}^*\val a\eta_*^{a^{-1}b}.
\]
It is now clear that this is equal to $\delta_a^b$, so that the identity
\eqref{eq:leftfrobsnake} is satisfied,  if 
$\eta\colon I\to S\otimes S$ is given by the rule
\begin{equation}\label{eq:etafrob}
\eta_*^{pb}=
\begin{cases}
\frac{1}{f_{s(b)}^*\val b} & \text{if $p=b^{-1}$}\\
0 & \text{otherwise}.
\end{cases}
\end{equation}
We now check the second identity of \eqref{eq:snakes}. Similarly as above, we need to
prove that for all $a,b\in S$,
\[
\delta_a^b=(\rho_S\circ(\id_S\otimes (f\circ\nabla))\circ(\eta\otimes
\id_S)\circ\lambda_S^{-1})_a^b=\sum_{p,q}f_p^*\nabla_{qa}^p\eta_*^{bq}.
\]
This time, the right-hand side is equal to
\[
f^*_{t(a)}\nabla^{t(a)}_{a^{-1}a}\eta_*^{ba^{-1}}=f^*_{t(a)}\val{a^{-1}}\eta_*^{ba^{-1}}=\delta_a^b,
\]
note that $s(a^{-1})=t(a)$.

Let us now prove the converse. Suppose that $f$ is a nondegenerate form. By the remarks
below \Cref{thm:nondeg}, this means that
\[
M:=(\id_{S^*}\otimes f)\circ(\id_{S^*}\otimes\nabla)\circ(\eta_S\otimes\id_S)\colon
I\otimes S\to S^*\otimes I
\]
is an isomorphism in $\MatR$. The matrix elements
of this isomorphism are
\[
M_a^b=\sum_{p\in S}\nabla_{ba}^pf_p^*=\sum_{p\in E_S}\nabla_{ba}^p
f_p^*=\nabla_{ba}^{t(a)}f^*_{t(a)},
\]
since the expressions under the sum are nonzero only if $p\in E_S$ and $\nabla^p_{ba}>0$,
which implies that $p=t(a)$. By \Cref{coro:chariso}, we obtain that for every $a\in S$ there is a unique $b\in S$ such
that $\nabla^{t(a)}_{ba}$ is nonzero.  Let $(x,y)\in a$.
As $t(a)$ is a unit color, $(y,y)\in t(a)$ and, by the uniqueness of $b$,
$(y,x)\in b$. Therefore $b$ is an inverse relation to $a$, and we see that the set of
relations $S$ is closed with respect to taking inverses.
\end{proof}

A natural choice of a nondegenerate form in the above theorem is $f=\unit^\dag$; this choice will be used in our paper.
Our main motivation for choosing $e^\dag$ is \Cref{thm:connections} below.

Let $(X,S)$ be a coherent configuration. As $e^\dag$ is a nondegenerate form on $(S,\nabla,e)$, 
the bijective correspondence in \Cref{thm:nondeg}
gives us a Frobenius structure comultiplication  
\[
\Delta=(\id_S\otimes\nabla)\circ(\eta\otimes\id_S)\circ\lambda_S^{-1},
\] 
where $\eta$ is given by \eqref{eq:etafrob} (with $f=\unit^\dag$). Unwinding this definition of $\Delta$,
for all $a,b,c\in S$
\begin{equation}\label{eq:comul}
\Delta_a^{bc}=\sum_{p\in
S}\nabla_{pa}^c\eta_*^{bp}=\frac{1}{\val{b^{-1}}}\nabla_{b^{-1}a}^c
\end{equation}
By \cite[Proposition 5.16]{heunen2019categories}, the same $\Delta$ is also given by
the equation
\[
\Delta=(\nabla\otimes\id_S)\circ(\id_S\otimes\eta)\circ\rho_S^{-1}
\]
which in our situation gives us
\begin{equation}\label{eq:comulAlt}
\Delta_a^{bc}=\frac{1}{\val{c}}\nabla_{ac^{-1}}^b
\end{equation}
so there are two ways of expressing the $\Delta$ in terms of $\nabla$ and the
valencies.

Let us remark that, alternatively, it is possible to define $\Delta$ by
\eqref{eq:comul} or \eqref{eq:comulAlt} and then prove directly by a
combinatorial argument that $(S,\Delta,e^\dag)$ is a comonoid and that
$\Delta,\nabla$ satisfy the Frobenius laws.

The following theorem shows that some properties of a coherent configuration
are reflected in the properties of the associated Frobenius structure.

\begin{theorem}\label{thm:connections}
For every coherent configuration $(X,S)$, the associated Frobenius structure
is
\begin{enumerate}[(a)]
\item special if and only if $(X,S)$ is thin,
\item symmetric if and only if $(X,S)$ is $\frac{1}{2}$-homogeneous,
\item extra if and only if $(X,S)$ is an association scheme.
\end{enumerate}
\end{theorem}
\begin{proof}
\begin{enumerate}[(a)]
\item
Suppose that $S$ is special, meaning that
for all $a,d\in S$,
\begin{multline}\label{eq:eye}
(\nabla\circ\Delta)_a^d=
\sum_{(b,c)\in S\otimes S}\nabla_{bc}^d\Delta_a^{bc}=\\
\sum_{(b,c)\in S\otimes S}\nabla_{bc}^d\frac{1}{\val{b^{-1}}}\nabla_{b^{-1}a}^c=
\sum_{b\in S}\frac{1}{\val{b^{-1}}}\sum_{c\in S}\nabla_{bc}^d\nabla_{b^{-1}a}^c
\end{multline}
is equal to $t\cdot\delta_a^d$, for some $t\in\mathbb R_0^+$.
We want to prove that this implies that $(X,S)$ is thin. Assume the contrary,
and denote a color with valency greater than one by $b^{-1}$. This means that
there is a vertex that is a source of at least two $b^{-1}$-colored edges, so there
must be some edges and vertices as in the following picture
\[
\xymatrix{
\cdot
    \ar[r]^{b^{-1}}
    \ar[d]_{b^{-1}}
&
\cdot
    \ar@(ur,ul)_a
\\
\cdot
    \ar[ur]_{d}
}
\]
Note that $d$ is a nonunit color and $a$ is a unit color, which means that $a\neq
d$ and that the picture can be drawn as
\[
\xymatrix{
\cdot
    \ar[r]^{b^{-1}}
&
\cdot
    \ar@(ur,ul)_a
\\
\cdot
    \ar[ur]_{d}
    \ar[u]^{b}
}
\]
The existence of such an arrangement of vertices and edges means that
\[
\nabla_{bb^{-1}}^d\nabla_{b^{-1}a}^{b^{-1}}>0,\] 
and this means (put $c:=b^{-1}$)
that the inner sum (and thus the whole sum) in the rightmost term of \eqref{eq:eye} is positive.
Therefore, $(\nabla\circ\Delta)_a^d>0$ for some $a\neq d$, so
$\nabla\circ\Delta$ is not diagonal.

Suppose that $(X,S)$ is thin. Let us prove that all off-diagonal entries
of $\nabla\circ\Delta$ are zero. Assuming the contrary, let $a,d\in S$ be such that
$a\neq d$ and $(\nabla\circ\Delta)_a^d>0$. Then there must be some  $b,c\in S$ for which 
$\nabla_{bc}^d\nabla_{b^{-1}a}^c>0$.
This means that 
for every $(x,y)\in d$ there exist $z,w\in X$ such that we have edges and colors
as in the following picture:
\[
\xymatrix{
z
    \ar[r]^{b^{-1}}
    \ar[rd]^c
&
w
    \ar[d]^a
\\
x
    \ar[r]_d
    \ar[u]^b
&
y
}
\]
The assumption $a\neq d$ means that $x\neq w$. Therefore, $(z,w),(z,x)$ are two
distinct edges colored by $b^{-1}$ with the same source vertex. 
Hence, $\val{b^{-1}}>1$ and $(X,S)$ is not thin.
We have proved that $\nabla\circ\Delta$ is a diagonal matrix. We now claim that every diagonal entry is equal to the
cardinality of $X$. As $(X,S)$ is thin, we have 
\[
 (\nabla\circ\Delta)_a^a=\sum_{b\in S}\sum_{c\in S}\nabla_{bc}^a\nabla_{b^{-1}a}^c.
\]
For every $b\in S$, the inner sum counts the number of squares
\begin{equation}\label{diag:collapsingSquare}
\xymatrix{
z
    \ar[r]^{b^{-1}}
&
w
    \ar[d]^a
\\
x
    \ar[r]_a
    \ar[u]^b
&
y
}
\end{equation}
over an arbitrary fixed $(x,y)\in a$. If $w\neq x$ in one of these squares, 
then $(y,x),(y,w)\in a^{-1}$ are
distinct, hence $\val{a^{-1}}\ge 2$ which contradicts the assumption that $(X,S)$ is
thin. Thus $w=x$ and we see that every square of the type \eqref{diag:collapsingSquare} collapses to
\[
\xymatrix{
z
    \ar@<-.75ex>[r]_-{b^{-1}}
&
x
    \ar[l]_-b
    \ar[r]_-a
&y
}
\]
Summing the number of such collapsed squares over all $b\in S$ actually means to
count the number of all $z\in X$, because each $z\in X$ gives us exactly one such
collapsed square. Therefore, $(\nabla\circ\Delta)_a^a=|X|$ for every $a\in S$ and we have proved that
the Frobenius structure $S$ is special.
\item
For every $a,b\in S$
\begin{equation}\label{eq:symmetrysigma}
(e^\dag\circ\nabla\circ\sigma_{S,S})_{ab}^*=
\sum_{p\in{E_S}}\nabla_{ba}^p=
\begin{cases}
\val{a^{-1}}&b=a^{-1}\\
0&\text{otherwise}
\end{cases}
\end{equation}
and 
\begin{equation}\label{eq:symmetrynosigma}
(e^\dag\circ\nabla)_{ab}^*=
\sum_{p\in{E_S}}\nabla_{ab}^p=
\begin{cases}
\val{a}&b=a^{-1}\\
0&\text{otherwise.}
\end{cases}
\end{equation}
Clearly, the equality of \eqref{eq:symmetrysigma} and \eqref{eq:symmetrynosigma}
is equivalent to $\val{a}=\val{a^{-1}}$, for all $a\in S$.
\item
It suffices to observe that $e^\dag\circ e$ is a $1\times 1$ matrix containing the number
\[
\sum_{p\in S}(e^\dag)_p^*e_*^p=\sum_{p\in E_S}1,
\]
which is equal to the cardinality of $E_S$.
\end{enumerate}
\end{proof}

\begin{theorem}
For every coherent configuration $(X,S)$, the Nakayama
automorphism $\alpha\colon S\to S$ 
of the associated Frobenius structure is a diagonal matrix given
by the rule
\[
\alpha_a^q=
\begin{cases}
\frac{\val{a}}{\val{a^{-1}}}&q=a\\
0&\text{otherwise.}
\end{cases}
\]
\end{theorem}
\begin{proof}
It suffices to prove that the morphism $\alpha$ given by the above rule satisfies \eqref{eq:nakayama}. On the left-hand side
of \eqref{eq:nakayama} we obtain, for all $(a,b)\in S\otimes S$,
\begin{multline*}
(e^\dag\circ\nabla\circ(\id_S\otimes\alpha)\circ\sigma_{S,S})_{ab}^*=\\
\sum_{\substack{q\in S\\r\in E_S}}\alpha_a^q\nabla_{bq}^r=
\sum_{r\in E_S}\frac{\val{a}}{\val{a^{-1}}}\nabla_{ba}^r=
\begin{cases}
\frac{\val{a}}{\val{a^{-1}}}\val{a^{-1}}&b=a^{-1}\\
0&\text{otherwise}
\end{cases}
\end{multline*}
and this is equal to \eqref{eq:symmetrynosigma}, which is the same as the right-hand
side of \eqref{eq:nakayama}.

\end{proof}
\begin{proposal}
The following corollary should be considered as well known. A combinatorial proof
follows: it can be easily
seen from the formula \eqref{eq:valencies} and the observation $t(a)=s(a^{-1})$
that
\[
|a| = |s(a)|\cdot\val{a} = |t(a)|\cdot\val{a^{-1}},
\]
hence
\[
\frac{\val{a}}{\val{a^{-1}}}=\frac{|t(a)|}{|s(a)|}.
\]
Then $\nabla_{ab}^c>0$ implies $t(a)=s(b)$ and the rest is trivial.

\end{proposal}
\begin{corollary}\label{coro:nakayama}
Let $(X,S)$ be a coherent configuration. For every $a,b,c\in S$ such that
$\nabla_{ab}^c>0$,
\[
\frac{\val{a}}{\val{a^{-1}}}\cdot
\frac{\val{b}}{\val{b^{-1}}}=
\frac{\val{c}}{\val{c^{-1}}}.
\]
\end{corollary}
\begin{proof}
The Nakayama automorphism $\alpha\colon S\to S$ preserves $\nabla$:
\begin{center}
\vskip 0.5em
\vcbox{\includegraphics{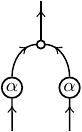}}
\vcbox{ $=$ }
\vcbox{\includegraphics{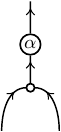}}
\end{center}
meaning that, for all $(a,b)\in S\otimes S, c\in S$
\[
\sum_{(q,r)\in S\otimes S}\nabla_{qr}^c\alpha_a^q\alpha_b^r=
\sum_{p\in S}\nabla_{ab}^p\alpha_p^{c}.
\]
Both  sums are reduced to a single term
\[
\nabla_{ab}^c\cdot\frac{\val{a}}{\val{a^{-1}}}.
\frac{\val{b}}{\val{b^{-1}}}=
\nabla_{ab}^c\cdot\frac{\val{c}}{\val{c^{-1}}}
\]
and we may cancel $\nabla_{ab}^c>0$ from both sides.

\end{proof}

\begin{proposal}
We have proved that every coherent configuration gives rise to a Frobenius structure
in $\MatR$. Hence by \Cref{sec:Lfunctor}, for every coherent configuration $(X,S)$,
$L(S)$ is a monoid in $\FHilb$, that means, a finite dimensional algebra. The module 
$\gamma\colon X\otimes S\to X$ in $\MatR$ gives us a module 
$L(\gamma)\colon L(X)\otimes L(S)\to L(X)$ in $\FHilb$,  which can be canonically
identified with an injective homomorphism of algebras $\gamma_*\colon L(S)\to L(X)^*\otimes L(X)$, where $L(X)^*\otimes
L(X)$ is isomorphic to the algebra of all
$X\times X$-indexed complex matrices. The image of $\gamma_*$ is then exactly the
adjacency algebra $\mathcal A(X,S)$. In particular, this means that $L(S)$ is, as an algebra,
isomorphic to $\mathcal A(X,S)$.
Equipping $\mathcal A(X,S)$ with
the usual Hilbert--Schmidt inner product of matrices gives us an $H^*$-algebra.

The algebra $\mathcal A(X,S)$ is equipped with matrix trace, which
on the $\MatR$ side corresponds to the nondegenerate form
\[
c_p^*=
\begin{cases}
|p|&p\in E_S\\
0&\text{otherwise}.
\end{cases}
\]
that gives rise to a Frobenius structure on $S$.

This Frobenius structure is special if and only
if $(X,S)$ is thin (the proof of the (a) part of \Cref{thm:connections} works with some small
alterations). On the other hand, the Frobenius structure arising from the matrix trace 
is always symmetric and the extra property is satisfied if and only if $|X|=1$.
Hence the Frobenius structure we are using in the
present paper could be reasonably considered as ``more sensitive'' than the one
arising from the matrix trace, because it distinguishes more properties of the
original data $(X,S)$.
\end{proposal}

\section{Morphisms and admissible morphisms}

In this section we will use $(X,S,\gamma)$ to denote the fact that
$\gamma$ is the module associated with the coherent configuration $(X,S)$. We do not 
distinguish between mappings in $\FinSet$ and their matrix representations
in $\MatR$.

The following definition is a straightforward extension of the definition of a morphism of association
schemes from \cite{zieschang1996}.
\begin{definition}\label{def:morphism}

A {\em  morphism} of coherent configurations $(X,S)\to (X',S')$ is
a pair of mappings $\phi_0\colon X\to X'$ and $\phi_1\colon S\to S'$ such that
for all $x,y\in X$ and $a\in S$, $(x,y)\in a$ implies $(\phi_0(x),\phi_0(y))\in \phi_1(a)$.

\end{definition}

We write $(\phi_0,\phi_1)\colon (X,S)\to (X',S')$ for a morphism of
coherent configurations.

\begin{proposition}\label{prop:laxmonoidmorphism}
Let $(\phi_0,\phi_1)\colon (X,S)\to (X',S')$ be a morphism of coherent
configurations.
\begin{enumerate}[(a)]
\item For every $p\in E_S$, $\phi_1(p)\in E_{S'}$.
\item For every $a\in S$, $\phi_1(s(a))=s(\phi_1(a))$.
\item For every $a\in S$, $\phi_1(t(a))=t(\phi_1(a))$.
\item For all $a,b,c\in S$, if $\nabla_{ab}^c>0$, then 
$\nabla_{\phi_1(a)\phi_1(b)}^{\phi_1(c)}>0$.
\item For all $a\in S$, $\phi_1(a^{-1})={\phi_1(a)}^{-1}$.
\end{enumerate}
\end{proposition}
\begin{proof}~
\begin{enumerate}[(a)]
\item If $p\in E_S$, then $(x,x)\in p$ for some $x\in X$. If $(x,x)\in p$, then $(\phi_0(x),\phi_0(x))\in\phi_1(p)$, so $\phi_1(p)\in E_{S'}$.
\item If $(x,y)\in a$, then $(x,x)\in s(a)$ and
$(\phi_0(x),\phi_0(x))\in\phi_1(s(a))$. On the other hand, 
$(\phi_0(x),\phi_0(y))\in\phi_1(a)$, so $(\phi_0(x),\phi_0(x))\in s(\phi_1(a))$.
\end{enumerate}
The remaining statements (c)--(e) can be proved in a very similar way.
\end{proof}

Let us link the notion of morphism with the 2-structure of $\MatR$. Items (a)
and (d) of \Cref{prop:laxmonoidmorphism} mean that $\phi_1$ is a \emph{lax morphism}
of monoids in $\MatR$, that means that the diagrams
\[
\begin{diagram}
I
    \ar@{}[rrd]|(.35){{\gtrsim}}
    \ar[r]^{e_S}
    \ar[rd]_{e_{S'}}
&
S
    \ar[d]^{\phi_1}
\\
~
&
S'
&
~
\end{diagram}
\quad
\begin{diagram}
S\otimes S
    \gscellrd    
    \ar[r]^{\nabla}
    \ar[d]_{\phi_1\otimes\phi_1}
&
S
    \ar[d]^{\phi_1}
\\
S'\otimes S'
    \ar[r]_{\nabla}
&
S'
\end{diagram}
\]
commute.
The pair of mappings $(\phi_0,\phi_1)$ then is a lax module morphism, as the following
proposition shows.

\begin{proposition}\label{prop:charmorphism}

Let $(X,S,\gamma)$, $(X',S',\gamma')$ be coherent configurations. A pair
of mappings $\phi_0\colon X\to X'$ and $\phi_1\colon S\to S'$
is a morphism if and only if the diagram
\begin{equation}\label{diag:morphism}
\begin{diagram}
X\otimes S
    \ar[r]^\gamma
    \ar[d]_{\phi_0\otimes \phi_1}
    \gscellrd
&
X
    \ar[d]^{\phi_0}
\\
X'\otimes S'
    \ar[r]_{\gamma'}
&
X'
\end{diagram}
\end{equation}
lax-commutes.
\end{proposition}
\begin{proof}

We see that
\[
(\phi_0\circ\gamma)_{xa}^{y'}=
\sum_{y\in X}(\phi_0)_y^{y'}\gamma_{xa}^y
\]
is nonzero if and only if there is $y\in X$ such that $(x,y)\in a$ and $\phi_0(y)=y'$.
On the other hand,
\[
(\gamma'\circ(\phi_0\otimes \phi_1))_{xa}^{y'}=
\sum_{(w,p)\in X'\otimes S'}(\gamma')_{wp}^{y'}(\phi_0)_x^w(\phi_1)_a^p
\]
is nonzero if and only if $(\phi_0(x),y')\in \phi_1(a)$.

$(\Leftarrow)$: Suppose that \eqref{diag:morphism} holds. Assume $(x,y)\in a$,
we need to prove that $(\phi_0(x),\phi_0(y))\in \phi_1(a)$. Putting $y'=\phi_0(y)$
means that $(\phi_0)_{y}^{y'}\gamma_{xa}^y>0$, hence
$(\phi_0\circ\gamma)_{xa}^{y'}>0$. By the diagram, this means that
$(\gamma'\circ(\phi_0\otimes \phi_1))_{xa}^{y'}>0$, so
$(\phi_0(x),y')=(\phi_0(x),\phi_0(y))\in \phi_1(a)$.

$(\Rightarrow)$: Suppose that $(\phi_0,\phi_1)$ is a morphism.
Assume that $(\phi_0\circ\gamma)_{xa}^{y'}>0$, this means that
there is $y\in X$ such that $\phi_0(y)=y'$ and $(x,y)\in a$. As $(x,y)\in a$,
$(\phi_0(x),\phi_0(y))\in \phi_1(a)$. Putting $w=\phi_0(x)$ and $p=\phi_1(a)$ we obtain
$(\gamma')_{wp}^{y'}=1$, so $(\gamma')_{wp}^{y'}(\phi_0)_x^w(\phi_1)_a^p>0$ and thus
$(\gamma'\circ(\phi_0\otimes \phi_1))_{xa}^{y'}>0$.

\end{proof}

In \cite{french2013functors} C. French introduced and studied a subclass of morphisms of association schemes called \emph{admissible morphisms}.
When we attempted to extend this notion to coherent configurations, 
it turned out that there are (at least) two appropriate generalizations of the
original notion: \Cref{def:admorphism} and \Cref{def:stradmorphism}. When restricted
to association schemes, both definitions coincide (\Cref{coro:defsCoincide}).
Let us consider the more general definition first.

\begin{definition}\label{def:admorphism}
An {\em admissible morphism} of coherent configurations $(X,S)\to (X',S')$ is
a morphism $(\phi_0,\phi_1)\colon (X,S)\to(X',S')$ such
that for any $x\in X$, $y'\in X'$, $a\in S$ such that $(x,a)\in\supp(\gamma)$ and
$(\phi_0(x),y')\in \phi_1(a)$,
there exists $y\in X$ such that $\phi_0(y)=y'$ and $(x,y)\in a$.
\end{definition}

The notion of a support endomorphism \eqref{eq:supportendo} allows us to express
\Cref{def:admorphism} in the form of a commutative diagram, as the following
proposition shows.

\begin{proposition}\label{prop:admorphism}
A morphism $(\phi_0,\phi_1)\colon (X,S,\gamma)\to(X',S',\gamma')$ of coherent configurations
is admissible if and only if the diagram
\begin{equation}\label{diag:admorphism}
\begin{diagram}
X\otimes S
    \ar[r]^-\gamma
    \ar[d]_{\based{\gamma}}
    \simeqcellrdd
&
X
    \ar[dd]^{\phi_0}
\\
X\otimes S
    \ar[d]_{\phi_0\otimes \phi_1}
    &
    ~
\\
X'\otimes S'
    \ar[r]_-{\gamma'}
&
X'
\end{diagram}
\end{equation}
lax-commutes.
\end{proposition}
\begin{proof}
By \Cref{prop:charmorphism} and the fact that $\gamma\circ\based\gamma=\gamma$,
we see that, for any morphism $(\phi_0,\phi_1)$,
\[
\begin{diagram}
X\otimes S
    \ar[r]^-\gamma
    \ar[d]_{\based{\gamma}}
&
X
    \ar[dd]^{\phi_0}
\\
X\otimes S
    \ar[d]_{\phi_0\otimes \phi_1}
    \ar[ru]_{\gamma}
    \gscellrd
    &
    ~
\\
X'\otimes S'
    \ar[r]_-{\gamma'}
&
X'
\end{diagram}
\]
Therefore, it suffices to prove that
the condition in \Cref{def:admorphism}
is equivalent to the lax commutativity of the diagram
\begin{equation}\label{diag:admap}
\begin{diagram}
X\otimes S
    \ar[r]^-\gamma
    \ar[d]_{\based{\gamma}}
    \lscellrdd
&
X
    \ar[dd]^{\phi_0}
\\
X\otimes S
    \ar[d]_{\phi_0\otimes \phi_1}
    &
    ~
\\
X'\otimes S'
    \ar[r]_-{\gamma'}
&
X'
\end{diagram}
\end{equation}
The value
\[
(\gamma'\circ(\phi_0\otimes \phi_1)\circ\based\gamma)_{xa}^{y'}=
\sum_{(w,p)\in X'\otimes S'}
{\gamma'}_{wp}^{y'}(\phi_0)_x^w(\phi_1)_a^p
\]
is nonzero if and only if $(x,a)\in\supp(\gamma)$ and $(\phi_0(x),y')\in \phi_1(a)$.
The value
\[
(\phi_0\circ\gamma)_{xa}^{y'}=
\sum_{y\in X}(\phi_0)_y^{y'}\gamma_{xa}^y
\]
is nonzero if and only if there is $y\in X$ such that $(x,y)\in a$ and $\phi_0(y)=y'$.
The rest is trivial.
\end{proof}

Clearly, for every coherent configuration $(X,S)$, $(\id_X,\id_S)\colon (X,S)\to
(X,S)$ is an admissible morphism. Moreover, admissible morphisms are closed with
respect to composition; before proving this fact, we need a lemma.
\begin{lemma}\label{lemma:paste}
Let $(\phi_0,\phi_1)\colon(X,S,\gamma)\to(X',S',\gamma')$ be a morphism of coherent
configurations. Then
\[
\begin{diagram}
X\otimes S
    \ar[r]^-{\based\gamma}
    \ar[dd]_{\based\gamma}
&
X\otimes S
    \ar[d]^{\phi_0\otimes\phi_1}
\\
~
&
X'\otimes S'
    \ar[d]^-{\based{\gamma'}}
\\
X\otimes S
    \ar[r]_{\phi_0\otimes\phi_1}
&
X'\otimes S'
\end{diagram}
\]
commutes.
\end{lemma}
\begin{proof}
Both composites in the diagram are $0,1$-matrices. Let us prove that 
\begin{equation}\label{eq:ineq}
(\based{\gamma'}\circ(\phi_0\otimes\phi_1)\circ\based\gamma)_{xa}^{x'a'}\geq
((\phi_0\otimes\phi_1)\circ\based\gamma)_{xa}^{x'a'}.
\end{equation}
We see that $((\phi_0\otimes\phi_1)\circ\based\gamma)_{xa}^{x'a'}=1$ if and
only if 
$(x,a)\in\supp(\gamma)$, $x'=\phi_0(x)$ and $a'=\phi_1(a)$. Since
$(x,a)\in\supp(\gamma)$, there is $y\in X$ such that $(x,y)\in a$ and
since $(\phi_0,\phi_1)$ is a morphism, $(\phi_0(x),\phi_0(y))\in\phi_1(a)$.
Therefore, $(x',a')=(\phi_0(x),\phi_1(a))\in\supp(\gamma')$ and the left-hand
side of \eqref{eq:ineq} is equal to $1$.

The opposite inequality
is trivially true, because $\based{\gamma'}$ is a diagonal matrix with $0,1$ on the
diagonal.
\end{proof}
\begin{proposition}
If
\begin{align*}
(\phi_0,\phi_1)\colon(X,S,\gamma)&\to(X',S',\gamma')\text{ and}\\
(\phi'_0,\phi'_1)\colon(X',S',\gamma')&\to(X'',S'',\gamma'')
\end{align*}
are admissible morphisms of
coherent configurations, then 
\[
(\phi'_0\circ\phi_0,\phi'_1\circ\phi_1)\colon(X,S,\gamma)\to(X'',S'',\gamma'')
\]
is an admissible
morphism.
\end{proposition}
\begin{proof}
\[
\begin{diagram}
X\otimes S
    \ar[rr]^-\gamma
    \ar[rd]^-{{\based{\gamma}}}
    \ar[d]_-{\based{\gamma}}
&
~
&
X
    \ar[dd]^{\phi_0}
\\
X\otimes S
    \ar[d]_-{\phi_0\otimes\phi_1}
&
X\otimes S
    \ar[d]^-{\phi_0\otimes\phi_1}
    \simeqcellr
&
~
\\
X'\otimes S'
    \ar[d]_-{\phi'_0\otimes\phi'_1}
    \simeqcellrrd
&
X'\otimes S'
    \ar[l]_-{\based{\gamma'}}
    \ar[r]^-{\gamma'}
&
X'
    \ar[d]^{\phi'_0}
\\
X''\otimes S''
    \ar[rr]_-{\gamma''}
&
~
&
X''
\end{diagram}
\]
The inner
cells in the diagram are just the two admissible morphisms $(\phi_0,\phi_1)$ and $(\phi_0',\phi_1')$ that we are composing, together with the third cell coming from \Cref{lemma:paste}. The outer border
of the diagram 
is just the admissible morphism $(\phi'_0\circ\phi_0,\phi'_1\circ\phi_1)\colon(X,S)\to(X'',S'')$.
\end{proof}
Suppose that $(\phi_0,\phi_1)\colon (X,S)\to (X',S')$ is an
admissible morphism. Can we find a matrix $\widehat{\phi_1}$ that
will make the diagram \eqref{diag:admorphism} commute precisely? The
answer to this question is positive, as the following proposition shows. 
\begin{proposition}\label{prop:phionehat}(cf. \cite[Lemma 3.12, Corollary 3.13]{french2013functors})
Let $(\phi_0,\phi_1)\colon(X,S,\gamma)\to(X',S',\gamma')$ be an admissible morphism of
coherent configurations.
\begin{enumerate}[(a)]
\item For all $z,y\in X$, with $(z,y)\in b\in S$, we have $\phi_0(z)=\phi_0(y)$ if and only if
$\phi_1(b)\in E_{S'}$.
\item For every $a\in S$,
$x\in X$, $y'\in X'$ such that 
$(x,a)\in\supp(\gamma)$ and
$(\phi_0(x),y')\in\phi_1(a)$, the number of
elements in the set
\[
\{z\in X:(x,z)\in a\text{ and }\phi_0(z)=y'\}
\]
does not depend on the choice of $x,y'$, only on $a$ and $(\phi_0,\phi_1)$.
\item Denote the number from (b) by $\valphi{a}$. Let $\widehat{\phi_1}$ be a matrix
given by
\[
(\widehat{\phi_1})_a^b=
\begin{cases}
\valphi{a}&\phi_1(a)=b,\\
0&\text{otherwise}.
\end{cases}
\]
The diagram
\begin{equation}\label{diag:modulemorph}
\begin{diagram}
X\otimes S
    \ar[r]^-\gamma
    \ar[d]_{\based{\gamma}}
&
X
    \ar[dd]^{\phi_0}
\\
X\otimes S
    \ar[d]_{\phi_0\otimes\widehat{\phi_1}}
    &
    ~
\\
X'\otimes S'
    \ar[r]_-{\gamma'}
&
X'
\end{diagram}
\end{equation}
commutes.
\item For every $a\in S$,
\[
\valphi{a}=\frac{\val a}{\val{\phi_1(a)}}.
\]
\end{enumerate}
\end{proposition}
\begin{proof}~\par
\begin{enumerate}[(a)]
\item Suppose that \( \phi_0(z) = \phi_0(y) \). Since \( (z,y) \in b \), it follows that  
\( (\phi_0(z), \phi_0(y)) = (\phi_0(z), \phi_0(z)) \in \phi_1(b) \), so \( \phi_1(b) \) is a unit color.  
Conversely, suppose that \( \phi_1(b) \in E_{S'} \). As $(\phi_0(z),\phi_0(y))$ is colored by a
unit color $\phi_1(b)$, we see that $\phi_0(z)=\phi_0(y)$.
\item As $(\phi_0,\phi_1)$ is admissible, there exists at least one 
$y\in X$ such that $\phi_0(y)=y'$ and $(x,y)\in a$. Fix one such $y$
in what follows, and consider the set 
\[
W=\{z\in X:(x,z)\in a\text{ and }\phi_0(z)=y'=\phi_0(y)\}.
\]
By (a), $z\in W$ if and only if $(z,y)$ is colored by a color $b\in S$ such that 
$\phi_1(b)\in E_{S'}$. Hence
$W$ is a disjoint union of sets
\[
W=\mathop{\dot\bigcup}\limits_{\phi_1(b)\in E_{S'}}^{} \{z\in X:(x,z)\in a\text{ and }(z,y)\in b\}
\]
and the cardinality of each of those sets is $\nabla_{ab}^a$. Hence the number of
elements in the set $W$ is equal to
\[
\sum_{\phi_1(b)\in E_{S'}}\nabla_{ab}^a
\]
and this clearly depends only on $a$ and $\phi_1$.
\item
The value
\[
(\phi_0\circ\gamma)_{xa}^{y'}=
\sum_{y\in X}(\phi_0)_y^{y'}\gamma_{xa}^y
\]
counts the number of elements in the set
\[
\{y\in X: (x,y)\in a\text{ and }\phi_0(y)=y'\}
\]
which is equal to $\valphi{a}$. On the other hand,
\[
(\gamma'\circ(\phi_0\otimes\widehat{\phi_1})\circ\based\gamma)_{xa}^{y'}=
\begin{cases}
\sum_{(w,p)\in X'\otimes S'}(\gamma')_{wp}^{y'}(\phi_0)_x^w(\widehat{\phi_1})_a^p &
(x,a)\in\supp(\gamma)\\
0 & \text{otherwise}
\end{cases}
\]
is nonzero if and only if $(\phi_0(x),y')\in\phi_1(a)$, and in this case, it is equal
to $\valphi{a}$.
\item
Whenever $(x,a)\in \supp(\gamma)$, the cardinality of the set
\[
Y=\{y\in X:(x,y)\in a\}
\]
is equal to $\val{a}$. Since $(\phi_0,\phi_1)$ is a morphism, \Cref{prop:charmorphism}
implies that $(\phi_0(x),\phi_1(a))\in\supp(\gamma')$, so the cardinality of the
set
\[
Y'=\{y'\in X':(\phi_0(x),y')\in\phi_1(a)\}
\]
is equal to $\val{\phi_1(a)}$. For every $y'\in Y'$, the number of
elements in the set
\[
W_{y'}=\{z\in X:(x,z)\in a\text{ and }\phi_0(z)=y'\}
\]
is equal to $\valphi{a}$, and the set $Y$ is a disjoint union of all those
sets:
\[
Y=\dot\bigcup_{y'\in Y'}W_{y'}.
\]
Therefore, $\val{a}=|Y|=|Y'|\cdot \valphi{a}=\val{\phi_1(a)}\cdot \valphi{a}$.
\end{enumerate}
\end{proof}

Since (as we will prove in the next section) admissible morphisms of coherent
configurations generalize admissible
morphisms of association schemes, all the examples from \cite{french2013functors}
apply.

\begin{example}\label{ex:terminalMorph}
It is easy to see that the terminal object of the category of coherent configurations
equipped with morphisms in the sense of \Cref{def:morphism} is
$1=(\{*\},\{\id_{\{*\}}\})$, the
coherent configuration with both degree and rank equal to $1$. For every coherent
configuration $(X,S)$, the unique morphism $!\colon (X,S)\to 1$ is admissible.
\end{example}

\begin{example}\label{ex:tensorProduct}
A \emph{tensor product} of coherent configurations $(X,S)$ and $(X',S')$ is the
coherent configuration $(X,S)\otimes(X',S')=(X\otimes X',S\otimes S')$, with
relations in $S\otimes S'$ given by the rule $((x,x'),(y,y'))\in(a,a')\in S\otimes S'$ if and only if
$(x,y)\in a$ and $(x',y')\in a'$. By \cite[Theorem 3.2.9]{chen2024lectures},
the tensor product of coherent configurations is a coherent configuration. Moreover,
the pair of mappings $\pi_0\colon X\otimes X'\to X'$, $\pi_1\colon S\otimes S'\to
S'$ given by the projection $\pi_0(x,x')=x'$, $\pi_1(a,a')=a'$ are admissible
morphisms. Indeed, suppose that $((x,x'),(a,a'))\in\supp(\gamma)$ and  
$(x',y')=(\pi_0(x,x'),y')\in\pi_1(a,a')=a'$. The first assumption implies that
there is some $(w,w')\in X\otimes X'$ such that $((x,x'),(w,w'))\in (a,a')$,
which implies that $(x,w)\in a$. Therefore, $((x,x'),(w,y'))\in (a,a')$ and
$\pi_0(w,y')=y'$.
\end{example}

\section{Strongly admissible morphisms}
Recall the ``lifting'' problem from \Cref{def:admorphism}. If an admissible
morphism admits solutions to a broader class of lifting problems -- namely those obtained
by dropping the assumption $(x,a)\in\supp(\gamma)$ from \Cref{def:admorphism} -- we call it a \emph{strongly admissible morphism}.
The following definition is a straightforward extension of \cite[Definition 3.1]{french2013functors} from association schemes to coherent configurations.
\begin{definition}\label{def:stradmorphism}
A {\em strongly admissible morphism} of coherent configurations $(X,S)\to (X',S')$ is
a morphism $(\phi_0,\phi_1)\colon (X,S)\to(X',S')$ such
that for any $x\in X$, $y'\in X'$, and $a\in S$ such that 
$(\phi_0(x),y')\in \phi_1(a)$,
there exists $y\in X$ with $\phi_0(y)=y'$ and $(x,y)\in a$.
\end{definition}
Clearly, every strongly admissible morphism is admissible and one can easily adjust the
proof of \Cref{prop:admorphism} to prove that a morphism of coherent configurations
$(\phi_0,\phi_1)$ is strongly admissible if and only if the diagram
\begin{equation}\label{diag:strongadmorphism}
\begin{diagram}
X\otimes S
    \ar[r]^{\gamma}
    \ar[d]_{\phi_0\otimes\phi_1}
    \simeqcellrd
&
X
    \ar[d]^{\phi_0}
\\
X'\otimes S'
    \ar[r]_{\gamma'}
&
X'
\end{diagram}
\end{equation}
commutes. One can also see this directly: omitting the assumption
$(x,a)\in\supp(\gamma)$ in \Cref{def:admorphism} corresponds exactly to replacing the $\based{\gamma}$
in the diagram \eqref{diag:admorphism} by $\id_{X\otimes S}$.

\begin{lemma}\label{lemma:charstradmissible}
Let $(\phi_0,\phi_1)\colon(X,S,\gamma)\rightarrow (X',S',\gamma')$ be an admissible morphism. 
The following are equivalent
\begin{enumerate}[(a)]
\item $(\phi_0,\phi_1)$ is strongly admissible.
\item $\phi_1$ is injective on $E_S$.
\item For all $x\in X$, $a\in S$ and $y'\in X'$: if $(\phi_0(x),y')\in\phi_1(a)$,
then $(x,a)\in\supp(\gamma)$.
\item The diagram
\begin{equation}\label{diag:charstrong}
\begin{diagram}
X\otimes S
    \ar[r]^{\based\gamma}
    \ar[dd]_{\phi_0\otimes\phi_1}
&
X\otimes S
    \ar[d]_{\phi_0\otimes\phi_1}
\\
&
X'\otimes S'
    \ar[d]^{\gamma'}
\\
X'\otimes S'
    \ar[r]_{\gamma'}
&
X'
\end{diagram}
\end{equation}
commutes. 
\end{enumerate}
\end{lemma}
\begin{proof}
(a)$\implies$(b): Suppose that $p,q\in E_S$ and that $\phi_1(p)=\phi_1(q)$. Let $x\in
X$ be such that $(x,x)\in p$. As $(\phi_0,\phi_1)$ is a morphism,
$(\phi_0(x),\phi_0(x))\in\phi_1(p)=\phi_1(q)$. Putting $y'=\phi_0(x)$ and 
$a=q$ in \Cref{def:stradmorphism} yields the existence of $y\in X$ such that
$\phi_0(y)=\phi_0(x)$ and $(x,y)\in q$. As $q\in E_S$, $x=y$ and thus $(x,x)\in q$,
so $p=q$.

(b)$\implies$(c):  
Let $(x,x)\in p\in E_S$. Then $(\phi_0(x),\phi_0(x))\in\phi_1(p)$.  
If $(\phi_0(x),y')\in\phi_1(a)$, then $s(\phi_1(a))=\phi_1(p)$.  
By \Cref{prop:laxmonoidmorphism}, we have $s(\phi_1(a))=\phi_1(s(a))$, and by the injectivity
of $\phi_1$ on $E_S$, it follows that $s(a)=p$, which means that $(x,x)\in s(a)$.  
Therefore, $(x,a)\in\supp(\gamma)$.

(c)$\implies$(d):
As $\based{\gamma}$ is a $0,1$-diagonal matrix,
\begin{equation}
\gamma'\circ(\phi_0\otimes\phi_1)
\geq
\gamma'\circ(\phi_0\otimes\phi_1)\circ\based{\gamma}.
\end{equation}
To prove the reverse inequality, note that
for $(x,a)\in X\otimes S$ and $y'\in X'$, 
\begin{equation}
(\gamma'\circ(\phi_0\otimes\phi_1))_{xa}^{y'}=
\begin{cases}\label{eq:gammaprime}
1 & (\phi_0(x),y')\in\phi_1(a)\\
0 & \text {otherwise.}
\end{cases}
\end{equation}
By (c), $(\phi_0(x),y')\in\phi_1(a)$ implies that $(x,a)\in\supp(\gamma)$, which is
equivalent to $(\based\gamma)_{xa}^{xa}=1$ and the rest is trivial.

(d)$\implies$(a):
Pasting of \eqref{diag:charstrong} and \eqref{diag:admorphism} yields the diagram
\eqref{diag:strongadmorphism}.
\end{proof}
\begin{corollary}
Let $(\phi_0,\phi_1)\colon(X,S,\gamma)\to(X',S',\gamma')$ be a strongly admissible morphism
of coherent configurations. Let $\widehat{\phi_1}$ be as in \Cref{prop:phionehat}. Then  the diagram
\begin{equation}\label{diag:modulestrongmmorph}
\begin{diagram}
X\otimes S
    \ar[r]^{\gamma}
    \ar[d]_{\phi_0\otimes\widehat{\phi_1}}
&
X
    \ar[d]^{\phi_0}
\\
X'\otimes S'
    \ar[r]_{\gamma'}
&
X'
\end{diagram}
\end{equation}
commutes.
\end{corollary}
\begin{proposal}
\begin{proof}
Let $D\colon S\to S$ be the diagonal matrix with $D_a^a=\valphi{a}$, so that
$\widehat{\phi_1}=\phi_1\circ D$. Since both $\id_X\otimes D$ and $\based\gamma$ are
diagonal matrices on $X\otimes S$, they commute. Consider the diagram
\[
\begin{diagram}[C=3.5pc]
X\otimes S
    \ar[rrr]^{\gamma}
    \ar[rd]^{\based\gamma}
    \ar[dd]^{\id_X\otimes D}
    \ar@/_4pc/[ddd]_{\phi_0\otimes\widehat{\phi_1}}
& & &
X
    \ar[ddd]^{\phi_0}
\\
&
X\otimes S
    \ar[d]_{\id_X\otimes D}
    \ar[rd]^{\phi_0\otimes\widehat{\phi_1}}
&
    \ar@{}[r]|{\eqref{diag:modulemorph}}
&
\\
X\otimes S
    \ar[r]^{\based\gamma}
    \ar[d]^{\phi_0\otimes\phi_1}
&
X\otimes S
    \ar[r]^{\phi_0\otimes\phi_1}
    \ar@{}[d]|{\eqref{diag:charstrong}}
&
X'\otimes S'
    \ar[rd]^{\gamma'}
&
\\
X'\otimes S'
    \ar[rrr]_{\gamma'}
& & &
X'
\end{diagram}
\]
The outer border of the diagram is then \eqref{diag:modulestrongmmorph}.
\end{proof}
\end{proposal}
\begin{corollary}\label{coro:defsCoincide}
Every admissible morphism from an association scheme $(X,S)$ is strongly admissible.
\end{corollary}
\begin{proof}
If $(X,S)$ is an association scheme, then $|E_S|=1$, the rest follows by (b) of
\Cref{lemma:charstradmissible}.
\end{proof}
\begin{example}\label{ex:tau}
Let $(X,S)$ be a coherent configuration. Consider the discrete coherent configuration
on the set of all unit colors, $\mathcal D_{E_S}$ with vertices
$E_S$ and colors given by all singleton relations on $E_S$ (which we identify with
$E_S\times E_S$).
Consider a pair of maps
$(\tau_{X,S})_0\colon X\to E_S$ and 
$(\tau_{X,S})_1\colon S\to E_S\times E_S$, 
where $(\tau_{X,S})_0(x)$ is the loop color on the vertex $x$
and $(\tau_{X,S})_1(a)=(s(a),t(a))$. Then $((\tau_{X,S})_0,(\tau_{X,S})_1)$ is a
strongly admissible morphism.
\end{example}
\begin{example}
Following \Cref{ex:terminalMorph}, consider the admissible morphism 
$!_{(X,S)}$ to the terminal object. By \Cref{lemma:charstradmissible} (b), $!_{(X,S)}$ is strongly
admissible if and only if $(X,S)$ is an association scheme.
\end{example}
\begin{example}
By \Cref{lemma:charstradmissible} (b), the projection morphism
$(\pi_0,\pi_1)\colon(X,S)\otimes(X',S')\rightarrow (X',S')$ in \Cref{ex:tensorProduct}
is strongly admissible if and only if $(X,S)$ is an association scheme.
\end{example}
The identity~\eqref{eq:semHom} from the following lemma was proved for association schemes in~\cite[Lemma 6.3]{french2013functors}. In the case of coherent configurations, one needs to be more careful: equality holds for strongly admissible ones, while for general admissible morphisms, it holds only in a certain lax form. 
\begin{lemma}\label{lemma:semHom}
    Let $(\phi_0,\phi_1)\colon (X,S)\rightarrow (X',S')$ be a strongly admissible morphism of coherent configurations, $a,b\in S$, and $c'\in S'$. 
Then $\widehat{\phi_1}$ is a semigroup morphism $(S,\nabla)\rightarrow
(S',\nabla)$, meaning that for all $a,b\in S$ and $c'\in S'$,
    \begin{equation}\label{eq:semHom}
    \sum_{\phi_1(c)=c'}\nabla_{ab}^c\frac{\val{c}}{\val{\phi_1(c)}}=\nabla^{c'}_{\phi_1(a),\phi_1(b)}\frac{\val{a}}{\val{\phi_1(a)}}\frac{\val{b}}{\val{\phi_1(b)}}.
\end{equation}
In the case of a general admissible morphism, either the equality in~\eqref{eq:semHom} holds or the left-hand side equals zero. 
\end{lemma}
\begin{proof}
First, note that if $a,b,c$ are such that $\nabla_{ab}^c>0$ then we must have
$\nabla_{\phi_1(a)\phi_1(b)}^{\phi_1(c)}>0$ for any morphism, so that if the right-hand side is zero, then
the left-hand side is zero as well. 

Consider the case $t(a)\not=s(b)$. Then by \Cref{lemma:sourcesTargets}, $\nabla_{ab}^c=0$ for all $c\in S$,
hence the left-hand side is zero. If
$(\phi_0,\phi_1)$ is strongly admissible, $t(\phi_1(a))\not=s(\phi_1(b))$ due to \Cref{lemma:charstradmissible}, and the
right-hand side vanishes as well. Moreover, if $s(\phi_1(a))\not=s(c')$ or
$t(\phi_1(b))\not=t(c')$, then the right-hand side of~\eqref{eq:semHom} vanishes,
so the left-hand side must vanish as well.

For the rest of the proof we assume $t(a)=s(b)$, $s(\phi_1(a))=s(c')$, and
$t(\phi_1(b))=t(c')$. In particular, there is some vertex $x'\in\mathrm{Im}(\phi_0)$
of color $s(c')=s(\phi_1(a))$. We will show that in this case the equality
in~\eqref{eq:semHom} holds also for admissible $(\phi_0,\phi_1)$.  

    Let us fix any $(x',z')\in c'$ and $x\in X$ such that $\phi_0(x)=x'$ and $x$ has color $s(a)$ (we can find such vertices by the aforementioned assumptions). We will count in two ways the number of $(y,z)\in X^2$, satisfying $(x,y)\in a$, $(y,z)\in b$, and $\phi_0(z)=z'$. 
First, there are $\nabla^{c'}_{\phi_1(a),\phi_1(b)}$ ways to choose possible $y'=\phi_0(y)$. Since $(\phi_0,\phi_1)$ is admissible, for any chosen $y'$, we have $\valphi{a}$ possibilities to find $y$, and then  $\valphi{b}$ possibilities to find $z$. This gives us the right-hand side.

For the other direction we may proceed as follows. 
We first find $z$ and discuss the color $c$ of $(x,z)$. Depending on color $c$,
we can choose $z$ in $\valphi{c}$ ways, due to $(\phi_0,\phi_1)$ being admissible.
After choosing $z$, we have $\nabla_{ab}^c$ possibilities to find a suitable
$y$. This counting gives us the left-hand side. Note that the left-hand side sum 
can be restricted to those $c$, for which $s(c)$ is the color of $x$, since the
corresponding $\nabla$ entry is zero otherwise.
\end{proof}
\begin{proposal}
\begin{corollary}\label{coro:LPhiIsSemHom}
Let $(\phi_0,\phi_1)\colon(X,S)\to(X',S')$ be a strongly admissible morphism of
coherent configurations. Then $L(\widehat{\phi_1})\colon L(S)\to L(S')$ is a non-unital
morphism of algebras.
\end{corollary}
\begin{proof}
$L$ is a strictly monoidal functor, hence it preserves semigroup
morphisms. By \Cref{lemma:semHom}, $\widehat{\phi_1}$ is a semigroup homomorphism.
\end{proof}
\begin{theorem}
Let $(\phi_0,\phi_1)\colon(X,S)\to(X',S')$ be a strongly admissible morphism of
coherent configurations. 
Then $\widehat{\phi_1}\colon S\to S'$ is a morphism of monoids
if and only if $\phi_1$ maps $E_S$ onto $E_{S'}$ bijectively.
\end{theorem}
\begin{proof}
By \Cref{lemma:semHom}, $\widehat{\phi_1}$ is always a semigroup morphism. By
\Cref{lemma:charstradmissible}, $\phi_1$ is always injective on $E_S$. Hence it remains to
prove that $\widehat{\phi_1}\circ e_{S}=e_{S'}$ if and only if $\phi_1$ maps $E_S$
onto $E_{S'}$, but this is trivial.
\end{proof}
\begin{corollary}
For every admissible morphism $(\phi_0,\phi_1)\colon(X,S)\to(X',S')$ between association
schemes, $\widehat{\phi_1}$ is a morphism of monoids.
\end{corollary}
\begin{proof}
$|E_S|=|E_{S'}|=1$.
\end{proof}
\begin{corollary}\label{coro:LPhiIsAlgHom}
For every admissible morphism $(\phi_0,\phi_1)\colon(X,S)\to(X',S')$ between association
schemes, $L(\widehat{\phi_1})\colon L(S)\to L(S')$ is a morphism of algebras.
\end{corollary}
\begin{proof}
By the previous corollary and strict monoidality of $L$.
\end{proof}
\end{proposal}

\begin{proposal}
    \begin{remark}
    Let $(X,S)$ be a coherent configuration and $\mathcal{A}(X,S)$ its adjacency algebra.
    By \cite[Corollary~6.4]{french2013functors},
every admissible morphism
$(\phi_0,\phi_1)\colon (X,S)\to (X',S')$
of association schemes induces a unital algebra homomorphism
between their adjacency algebras, given by
\begin{equation}
    \sigma_a \longmapsto
    \frac{\val{a}}{\val{\phi_1(a)}}\,
    \sigma_{\phi_1(a)}.
\end{equation}
By \cite[Corollary~6.5]{french2013functors},
this construction defines a functor from the category of
association schemes and admissible morphisms to the category
of algebras and unital algebra homomorphisms;
this is (essentially) \Cref{coro:LPhiIsAlgHom}, if we identify each $L(S)$ with $\mathcal A(X,S)$.
\Cref{coro:LPhiIsSemHom} then tells us that $\mathcal A$ is
a functor from coherent configurations and strongly 
admissible morphisms to the category of algebras and non-unital homomorphisms of
algebras.
\end{remark}
\end{proposal}

\section{Dagger Frobenius structures}

Examples \ref{ex:groupoids} and \ref{ex:hStarAlgebras} show that 
dagger Frobenius structures are important and interesting in the categories
$\FHilb$ and $\Rel$. This leads to a natural question: are the Frobenius
structures that we have associated to coherent configurations \emph{dagger} Frobenius?
It is very easy to see that, in general, the answer to this question is negative. Indeed, the
matrix $\nabla$ consists only of natural numbers, whereas $\Delta$ may contain
nonnatural rational numbers, whenever there is a triple of colors
$a,b,c$ such that $\val c$ is not a divisor of $\nabla_{ac^{-1}}^b$. Clearly,
whenever this happens, $\Delta\neq\nabla^\dag$.
However, as we will prove in this section, there is another Frobenius structure
on the set of all colors, that happens to be  dagger Frobenius.
The ``daggerization'' can be understood as a ``normalization'' of
structure constants via valencies,
converting the Frobenius structure to a dagger Frobenius structure. 
As we will demonstrate in the following sections, the dagger variant has more
convenient properties, especially in connection with admissible morphisms.

Let $(X,S)$ be a coherent configuration. For each $a,b,c\in S$ we denote 
\begin{equation}\label{eq:daggerNabla}
    \dnabla{a}{b}{c}=\ddnabla{a}{b}{c}.
\end{equation}

We will prove in this section that $(S,\Dnabla,e)$ is a dagger Frobenius
structure. We will refer to it as \emph{dagger Frobenius structure associated
with $(X,S)$}.

\begin{lemma}\label{lem:rotations}
    Let $(X,S)$ be a coherent configuration and $a,b,c\in S$. Then
    \begin{align}
        \dnabla{a}{b}{c}&=\dnabla{b^{-1}}{a^{-1}}{c^{-1}},\label{eq:dagOp}\\
        \dnabla{a}{b}{c} &= \sqrt{\frac{\val{b^{-1}}}{\val{b}}}\dnabla{b}{c^{-1}}{a^{-1}},\label{eq:dagRotR}\\
        \dnabla{a}{b}{c} &= \sqrt{\frac{\val{a}}{\val{a^{-1}}}}\dnabla{c^{-1}}{a}{b^{-1}}.\label{eq:dagRotL}
    \end{align}
    In the case $(X,S)$ is $\frac{1}{2}$-homogeneous, we have $\dnabla{a}{b}{c}=\dnabla{b}{c^{-1}}{a^{-1}}=\dnabla{c^{-1}}{a}{b^{-1}}$.
\end{lemma}
\begin{proof}
To prove \eqref{eq:dagOp}, observe that 
$\nabla_{ab}^{c}=\nabla_{b^{-1}a^{-1}}^{c^{-1}}$. Note that 
$\dnabla{a}{b}{c}$ is nonzero if and only if $\nabla_{ab}^c$ is nonzero, so we may 
assume that $\nabla_{ab}^c$ is nonzero.
Then, by \Cref{coro:nakayama},
\[
\frac{\val{a}}{\val{a^{-1}}}\cdot
\frac{\val{b}}{\val{b^{-1}}}=
\frac{\val{c}}{\val{c^{-1}}},
\]
which implies
\begin{equation}\label{eq:tripRot}
    \frac{\val{c}}{\val{a}\val{b}}=\frac{\val{c^{-1}}}{\val{a^{-1}}\val{b^{-1}}},
\end{equation}
and this clearly implies~\eqref{eq:dagOp}.
The proof of \eqref{eq:dagRotR} is as follows:
    \begin{align*}
        \dnabla{a}{b}{c}&=\ddnabla{a}{b}{c}=\frac{\val{c}}{\sqrt{\val{a}\val{b}\val{c}}}\nabla_{a b}^c=\frac{\val{a}}{\sqrt{\val{a}\val{b}\val{c}}}\nabla_{c b^{-1}}^a\\
        &=\sqrt{\frac{\val{b^{-1}}}{\val{b}}}\sqrt{\frac{\val{a}}{\val{c}\val{b^{-1}}}}\nabla^{a}_{cb^{-1}}=\sqrt{\frac{\val{b^{-1}}}{\val{b}}}\dnabla{c}{b^{-1}}{a}=\sqrt{\frac{\val{b^{-1}}}{\val{b}}}\dnabla{b}{c^{-1}}{a^{-1}}.
    \end{align*}
The third equality follows by the double counting principle: the terms $\val c\nabla_{ab}^{c}$ and $\val
a\nabla_{cb^{-1}}^a$ are either both zero, or they count the number of triangles
\[
\xymatrix{
~
&
\cdot
    \ar[rd]^b
\\
x
    \ar[rr]_c
    \ar[ur]^a
&
~
&
\cdot
}
\]
for any vertex $x$, for which at least one such  triangle exists.
    
Equation~\eqref{eq:dagRotL} follows from~\eqref{eq:dagRotR} by first
changing the variables $a\mapsto b^{-1}$, $b\mapsto a^{-1}$, $c\mapsto c^{-1}$ and
then applying \eqref{eq:dagOp} to both sides.
\end{proof}
The ``normalization'' of the structure constants by \eqref{eq:daggerNabla} preserves
the monoid laws, as the following lemma shows.
\begin{lemma}\label{lem:daggerMonoid}
For a coherent configuration $(X,S)$, the triple $(S,\Dnabla,e)$ is a monoid in
$\MatR$.
\end{lemma}
\begin{proof} Recall that by \Cref{thm:monAction} $(S,\nabla,e)$ is a monoid.
    First, we prove the associativity law. For each quadruple of colors $a,b,c,d\in S$, we have:
    \begin{align*}
        \sum_f\dnabla{a}{b}{f}\dnabla{f}{c}{d}&=
        \sqrt{\frac{\val{d}}{\val{a}\cdot\val{b}\cdot\val{c}}}\sum_{f\in S}\nabla_{ab}^f\nabla_{fc}^d\\
        &=\sqrt{\frac{\val{d}}{\val{a}\cdot\val{b}\cdot\val{c}}}\sum_{f\in S}\nabla_{af}^d\nabla_{bc}^f=\sum_{f\in S}\dnabla{a}{f}{d}\dnabla{b}{c}{f}.
    \end{align*}
    For $a,b\in S$, we verify the unit rule $\Dnabla\circ (e\otimes \id)=\lambda$:
    \begin{equation*}
        \sum_{u\in E_S}\dnabla{u}{a}{b}=\sum_{u\in E_S}\sqrt{\frac{\val{b}}{\val{u}\cdot\val{a}}}\nabla_{ua}^b=\sqrt{\frac{\val{b}}{\val{a}}}\sum_{u\in E_S}\nabla_{ua}^b=\sqrt{\frac{\val{b}}{\val{a}}}\delta_a^b=\delta_a^b.
    \end{equation*}
    We can show that $\Dnabla\circ (\id\otimes e)=\rho$ analogously.
\end{proof}

\begin{theorem}\label{thm:daggerFrob}
    Let $(X,S)$ be a coherent configuration. Then $(S,\Dnabla,e)$ is a dagger Frobenius structure in $\MatR$.
\end{theorem}
\begin{proof}
    It is enough to verify the Frobenius law $$(\Dnabla\otimes\id)\circ(\id\otimes\Dnabla^\dagger)= (\id\otimes \Dnabla)\circ(\Dnabla^\dagger\otimes \id).$$
    For each $a, b, c, d\in S$, using formulas from \Cref{lem:rotations}, we derive:
    \begin{align*}
        \sum_{f\in S} \dnabla{f}{d}{b}\dnabla{a}{f}{c}&=\sum_{f\in S}\sqrt{\frac{\val{f}}{\val{f^{-1}}}}\dnabla{b^{-1}}{f}{d^{-1}}\sqrt{\frac{\val{f^{-1}}}{\val{f}}}\dnabla{f}{c^{-1}}{a^{-1}}=\sum_{f\in S}\dnabla{b^{-1}}{f}{d^{-1}}\dnabla{f}{c^{-1}}{a^{-1}}\\
        &=\sum_{f\in S}\dnabla{f^{-1}}{b}{d}\dnabla{c}{f^{-1}}{a}=\sum_{f\in S}\dnabla{f}{b}{d}\dnabla{c}{f}{a}.
    \end{align*}
\end{proof}
\begin{corollary}\label{coro:thinIsDaggerSpecial}
For a thin coherent configuration, $(S,\nabla,e)$ is a special dagger Frobenius structure
in $\MatR$.
\end{corollary}
\begin{proof}
If $(X,S)$ is thin, then $\Dnabla=\nabla$. By \Cref{thm:connections}, the Frobenius structure is special.
\end{proof}
By the main result of \cite{heunen2013relative}, see also \cite[Theorem
5.41]{heunen2019categories}, special dagger Frobenius structures in
$\Rel$ are in a one-to-one correspondence with groupoids. In this correspondence, the
underlying set of the Frobenius structure is the set of all morphisms of a groupoid,
the multiplication $\nabla$ is essentially the composition $\circ$ and the unit
relation $e$ selects the identity morphisms.

In \cite{hanaki2015thin}, the authors proved that every thin coherent configuration
$(X,S)$ can be presented as a groupoid with the set of objects $E_S$ and the set of
morphisms $S$. This result is also reproved in \cite{kuribayashi2015association}.
\Cref{coro:thinIsDaggerSpecial} provides an alternative proof of this fact.
\begin{corollary}\label{coro:thinCCIsGroupoid}
If $(X,S)$ is thin, then the monoid $(R(S),R(\nabla),R(e))$ in $\Rel$ is a groupoid.
\end{corollary}

Another consequence of the existence of dagger Frobenius structure is the following theorem.
\begin{theorem}\label{coro:hstar}
For every $\frac{1}{2}$-homogeneous coherent configuration $(X,S)$,
the algebra $(L(S),L(\Dnabla),L(e))$ in $\FHilb$ is an $H^*$-algebra.
\end{theorem}
\begin{proof}
By \Cref{thm:connections}, $(X,S)$ is $\frac{1}{2}$-homogeneous if and only if
the associated Frobenius structure is symmetric. Let us prove that
the associated \emph{dagger} Frobenius structure is symmetric, as well.
Indeed, similarly as in the proof of \Cref{thm:connections}, we obtain
after some simple algebraic manipulations,
\[
(e^\dag\circ\Dnabla\circ\sigma_{S,S})_{ab}^*=
\sum_{p\in{E_S}}\Dnabla_{ba}^p=
\sum_{p\in{E_S}}\sqrt{\frac{1}{\val b\val a}}\nabla_{ba}^p=
\begin{cases}
\sqrt{\frac{\val b}{\val a}}&b=a^{-1}\\
0&\text{otherwise}
\end{cases}
\]
and
\[
(e^\dag\circ\Dnabla)_{ab}^*=
\sum_{p\in{E_S}}\Dnabla_{ab}^p=
\sum_{p\in{E_S}}\sqrt{\frac{1}{\val a\val b}}\nabla_{ab}^p=
\begin{cases}
\sqrt{\frac{\val a}{\val b}}&b=a^{-1}\\
0&\text{otherwise}
\end{cases}
\]
and we see that the symmetry of the dagger Frobenius structure associated
with $(X,S)$ is equivalent to the $\frac{1}{2}$-homogeneity of $(X,S)$. Therefore,
$L(S)$ is a dagger symmetric Frobenius structure in $\FHilb$ and those 
are exactly $H^*$-algebras by \cite[Theorem 5.32]{heunen2019categories}.
\end{proof}

The $H^*$-algebra in the above theorem is actually obtained from the adjacency
algebra $\mathcal A(X,S)$. Introducing the inner product $(A,B)\mapsto \mathrm{Tr}\,[A^*B]$ in $\mathcal A(X,S)$ endows
it with an $H^*$-algebra structure (see \cite[Sec. 5.4]{heunen2019categories}). Moreover,
one can show that 
$\{\bar \sigma_a:= |a|^{-1/2}\sigma_a\}$ is an orthonormal basis of $\mathcal A(X,S)$. The structure
constants with respect to this new basis have the form:
\[
\sqrt{\frac{|c|}{|a||b|}} \nabla_{ab}^c=\sqrt{\frac{|s(c)|\val c}{|s(a)|\val a|s(b)|\val
b}}\nabla_{ab}^c=k \Dnabla_{ab}^c,
\]
where $k=|p|^{-1/2}$ for $p\in E_S$ (remember that $(X,S)$ is assumed to be
$\frac12$-homogeneous, so that all unit colors have the same cardinality). This shows
that, up to multiplication by a constant, the symmetric dagger Frobenius structure in $\FHilb$
$(L(S),L(\Dnabla), L(e))$ coincides with the structure of the $H^*$-algebra $\mathcal A(X,S)$.

\section{Admissible morphisms and the dagger Frobenius structures}

A strongly admissible morphism induces a mapping of Frobenius structures which
preserves multiplication (\Cref{lemma:semHom}). In this section, we observe that
if we pass to the dagger variant, the comultiplication structure is to some extent
preserved as well. Note that one should not seek a  notion of morphisms which
preserve all the operations of Frobenius structures, since these are necessarily
isomorphisms as proved in~\cite[Lemma 5.11]{heunen2019categories}.

\begin{lemma}\label{lemma:dagSemHom}  
A strongly admissible morphism $(\phi_0,\phi_1)\colon (X,S)\to(X',S')$ 
induces a semigroup morphism $\pmb\phi$ of associated dagger Frobenius structures,
given by the rule
\begin{equation}\label{eq:daggerAdmissibleEntries}
    \pmb\phi_a^{b'}=\begin{cases}
        \sqrt{\valphi{a}} & \text{if\ } \phi_1(a)=b',\\
        0 & \text{otherwise}.
    \end{cases}
\end{equation}
\end{lemma}
\begin{proof}
Multiplying~\eqref{eq:semHom} by $\frac{\sqrt{\|c'\|}}{\sqrt{\|a\|}\sqrt{\|b\|}}$ results in
\begin{multline}
    \sum_{{\phi_1}(c)=c'}\frac{\sqrt{\val{c}}}{\sqrt{\val{a}}\sqrt{\val{b}}}\nabla_{ab}^c\frac{\sqrt{\val{c}}}{\sqrt{\val{{\phi_1}(c)}}}\\
    =\frac{\sqrt{\val{c'}}}{\sqrt{\val{{\phi_1}(a)}}\sqrt{\val{{\phi_1}(b)}}}\nabla^{c'}_{{\phi_1}(a){\phi_1}(b)}\frac{\sqrt{\val{a}}}{\sqrt{\val{{\phi_1}(a)}}}\frac{\sqrt{\val{b}}}{\sqrt{\val{{\phi_1}(b)}}}.
\end{multline}
This gives us
\begin{equation}\label{eq:dMonHom}
\sum_{\phi_1(c)=c'}\dnabla{a}{b}{c}\sqrt{\valphi{c}}=\dnabla{\phi_1(a)}{\phi_1(b)}{c'}\sqrt{\valphi{a}}\sqrt{\valphi{b}}.
\end{equation}
\end{proof}

In the remainder of the paper, whenever $(\phi_0,\phi_1)$ is a strongly admissible
morphism of coherent configurations, $\pmb\phi$ will denote the semigroup
morphism of associated dagger Frobenius structures, as in \Cref{lemma:dagSemHom}. 

In the case of association schemes, the morphism $\pmb\phi$ from the previous lemma is even a monoid morphism. This may happen also for coherent configurations, as the next corollary shows.
\begin{corollary}\label{cor:daggerMonHom}
Let $(\phi_0,\phi_1)\colon(X,S)\rightarrow (X',S')$ be a strongly admissible
morphism. Then $\pmb\phi\colon (S,\Dnabla,e)\rightarrow (S',\Dnabla',e')$ is a
morphism of monoids if and only if  $E_{S'}\subseteq \mathrm{Im}(\phi_1)$.
\end{corollary}
\begin{proof}
By \Cref{lemma:dagSemHom}, we already know that $\pmb\phi$ preserves multiplication. Therefore, it is enough to focus on the preservation of the unit. 
First, assume $\pmb\phi$ is a monoid morphism. For each $u'\in E_{S'}$, we have
$1={(e')}^{u'}_*=(\pmb\phi\circ e)^{u'}_*$. This could happen only if $u'$ is in the image of $\phi_1$.

    Next, consider the other implication. By the assumptions, the mapping $\phi_1$
    restricts to a bijection between $E_S$ and $E_{S'}$. Moreover, for each $u\in
    E_S$, $\valphi{u}=1$. It follows that \[
    (\pmb\phi\circ e)_*^{b'}=\sum_{u\in
    E_S}\pmb\phi^{b'}_u=\sum_{u\in E_S}\delta^{b'}_{\phi_1(u)}=\sum_{u'\in
    E_{S'}}\delta^{b'}_{u'}={(e')}_*^{b'}.\]
\end{proof}

Let $(\phi_0,\phi_1)\colon(X,S)\to(X',S')$ be a strongly admissible morphism of
coherent configurations. For every $p'\in E_{S'}$ and $(x',x')\in p'$,
the number $|\phi_0^{-1}(x')|$ does not depend on the choice of $x'$. Indeed,
it is easy to see that
\begin{equation}\label{eq:kp}
|\phi_0^{-1}(x')|=\sum_{\substack{a\in S\\ \phi_1(a)=p'}}\valphi{a}.
\end{equation}
In what follows, we denote the quantity in \eqref{eq:kp} by $k_{p'}$.
\begin{definition}\label{def:ddager}
    For a strongly admissible morphism $(\phi_0,\phi_1)\colon(X,S)\to(X',S')$ 
    the matrix
    $\phi^\ddagger\colon S'\to S$ is given by the rule
    \begin{equation}\label{eq:daggerHom}
    (\phi^\ddagger)_{b'}^a=
        \frac{1}{k_{t(b')}}\pmb\phi_a^{b'},\qquad a\in S,\ b'\in S'.
\end{equation}
\end{definition}

It immediately follows from~\eqref{eq:daggerHom} that we have the following equations
\begin{equation}\label{eq:diagonalK}
    \pmb\phi^\dagger=\phi^\ddagger\circ k_{\phi} \quad\text{and}\quad\phi^\ddagger=\pmb\phi^\dagger\circ k^{-1}_{\phi}
\end{equation}
where $k_\phi\colon S'\to S'$ is a diagonal matrix having $k_{t(b')}$'s on the diagonal. 

Note that if $(\phi_0,\phi_1)$ is an admissible morphism between schemes, then $k_\phi$ is a constant diagonal matrix, hence
$(\phi^\ddagger)_{b'}^a=\frac{1}{k}\cdot\pmb\phi^{b'}_a$ for some integer $k\geq 1$.
\begin{lemma}\label{lemma:semigroupHom}
Let $(\phi_0,\phi_1)\colon (X,S)\rightarrow (X',S')$ be a strongly admissible
morphism of coherent configurations, let $a',b'\in\mathrm{Im}(\phi_1)$, and let $c\in S$.
Then
\begin{equation}\label{eq:cosemHom}
    \sum_{p,q\in S}\dnabla{p}{q}{c}\cdot(\phi^\ddagger)^p_{a'}\cdot
(\phi^\ddagger)_{b'}^q= (\phi^\ddagger)_{\phi_1(c)}^{c}\cdot\dnabla{a'}{b'}{\phi_1(c)}.
\end{equation}
\end{lemma}
\begin{proof}
By definition of $\phi^\ddagger$, we may restrict  the summation on the left-hand side of
\eqref{eq:cosemHom}  to $p,q\in S$ such that $\phi_1(p)=a'$ and $\phi_1(q)=b'$. It
then follows from \eqref{eq:dMonHom} that  if $\Dnabla_{a'b'}^{\phi_1(c)}=0$, then
both sides of \eqref{eq:cosemHom} vanish. For the
rest of the proof we assume that $\Dnabla_{a'b'}^{\phi_1(c)}\neq0$.

    We will prove by double-counting the following equation
    \begin{equation}\label{eq:help1}
\sum_{\substack{\phi_1(p)=a'\\\phi_1(q)=b'}}\nabla^c_{p q}=k_{t(a')}\cdot \nabla^{\phi_1(c)}_{a' b'}.
    \end{equation}
    For a fixed edge $(x,y)\in c$, the left-hand side equals the number of $z\in X$,
such that $(\phi_0(x),\phi_0(z))\in a'$ and $(\phi_0(z),\phi_0(y))\in b'$. Let us
count the same number in a different way. First, there are
$\nabla^{\phi_1(c)}_{a'b'}$ choices of an element $z'\in X'$ such that
$(\phi_0(x),z')\in a'$ and $(z',\phi_0(y))\in b'$.  Due to strong admissibility and the assumption
$a'\in\mathrm{Im}(\phi_1)$, each chosen $z'$ belongs to $\mathrm{Im}(\phi_0)$ and
there are $k_{t(a')}$ elements in $\phi_0^{-1}(z')$. This yields the right-hand side.

    Let us finally prove the equation~\eqref{eq:cosemHom}:
    \begin{multline*}
        \sum_{p,q\in S}\dnabla{p}{q}{c}\cdot(\phi^\ddagger)^p_{a'}\cdot
(\phi^\ddagger)_{b'}^q=\sum_{\substack{\phi_1(p)=a'\\ \phi_1(q)=b'}}\ddnabla{p}{q}{c}\sqrt{\frac{\val{p}}{\val{a'}}}\sqrt{\frac{\val{q}}{\val{b'}}}\frac{1}{k_{t(a')}k_{t(b')}}\\
        =\sqrt{\frac{\val{c}}{\val{a'}\val{b'}}}
\frac{1}{k_{t(a')}k_{t(b')}}\sum_{\substack{\phi_1(p)=a' \\ \phi_1(q)=b'}}\nabla_{pq}^{c}\\
        \overset{\eqref{eq:help1}}{=}\sqrt{\frac{\val{c}}{\val{a'}\val{b'}}}\nabla^{\phi_1(c)}_{a'b'}\frac{1}{k_{t(b')}}
        =\sqrt{\frac{\val{\phi_1(c)}}{\val{a'}\val{b'}}}\nabla^{\phi_1(c)}_{a'b'}\sqrt{\frac{\val{c}}{\val{\phi_1(c)}}}\frac{1}{k_{t(b')}}=(\phi^\ddagger)_{\phi_1(c)}^{c}\cdot\dnabla{a'}{b'}{\phi_1(c)}.
    \end{multline*}
    Note that in the last step we use that $t(b')=t(\phi_1(c))$, which follows from
the assumption $\dnabla{a'}{b'}{\phi_1(c)}>0$ and \Cref{lemma:sourcesTargets}.
\end{proof}
\begin{corollary}\label{cor:semigroupHom}
    Let $(\phi_0,\phi_1)\colon (X,S)\rightarrow (X',S')$ be a strongly
admissible morphism with surjective $\phi_1$. 
Then $\phi^\ddagger\colon (S',\Dnabla)\to(S,\Dnabla)$ is a morphism of semigroups and $k_\phi^{-1}\circ\pmb\phi\colon
(S,\Dnabla^\dagger)\rightarrow(S',\Dnabla^\dagger)$ is a morphism of cosemigroups.
\end{corollary}
\begin{proof}
As $\phi_1$ is surjective, \Cref{lemma:semigroupHom} clearly implies 
that $\phi^\ddagger$ is a morphism of semigroups.

    Note that $k^{-1}_\phi\circ\pmb\phi=(\phi^\ddagger)^\dagger$. Hence the statement that 
$(\phi^\ddagger)^\dagger$ preserves comultiplication $\Dnabla^\dagger$ is equivalent
to the statement $\phi^\ddagger$ preserves multiplication $\Dnabla$.
\end{proof}
\begin{lemma}\label{lem:ddagger}
Let $(\phi_0,\phi_1)\colon (X,S)\rightarrow (X',S')$ be a strongly admissible morphism. Then $\Dnabla\circ (\id\otimes
k_\phi)=k_\phi\circ \Dnabla$.
\end{lemma}
\begin{proof}
    For each $a',b',c'\in S'$, we have $\dnabla{a'}{b'}{c'}\cdot
k_{t(b')}=\dnabla{a'}{b'}{c'}\cdot k_{t(c')}$, since $t(b')=t(c')$ whenever the
$\nabla$ entry is nonzero.
\end{proof}
\begin{corollary}\label{cor:ddagger}
    Let $(\phi_0,\phi_1)\colon (X,S)\rightarrow (X',S')$ be a strongly admissible
morphism. If $\phi_1$ is surjective,
then $\Dnabla\circ (\pmb\phi^\dagger\otimes\pmb\phi^\dagger)=
\pmb\phi^\dagger\circ\Dnabla\circ(k_\phi\otimes\id)$.
\end{corollary}
\begin{proof}
We will prove the equality using string diagrams as follows. In the computation, the
first step is by~\eqref{eq:diagonalK}, the second step follows from
\Cref{cor:semigroupHom}, and the last step is by \Cref{lem:ddagger} and~\eqref{eq:diagonalK}:
\begin{center}
\vcbox{\includegraphics{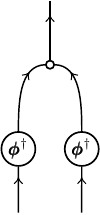}}
\vcbox{$=$\hspace*{0.7em}}
\vcbox{\includegraphics{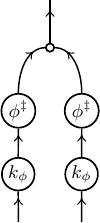}}
\vcbox{\hspace*{0.7em}$=$}
\vcbox{\includegraphics{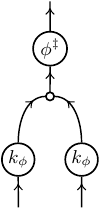}}
\vcbox{$=$}
\vcbox{\includegraphics{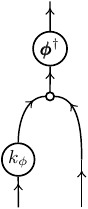}}
\end{center}
\end{proof}
We can express the image and the kernel of a strongly admissible morphism in terms of $\phi^\ddagger$.
\begin{lemma}\label{lem:projections}
Let $(\phi_0,\phi_1)\colon (X,S)\rightarrow (X',S')$ be a strongly admissible morphism of coherent configurations. Then
both compositions $\pmb\phi\circ\phi^\ddagger$ and $\phi^\ddagger\circ\pmb\phi$ are idempotent endomorphisms in $\MatR$, given by the following formulas. For each $a',b'\in S'$, we
have
    \begin{equation}\label{eq:im}
        (\pmb\phi\circ\phi^\ddagger)_{a'}^{b'}=\begin{cases}
        1 & \text{if\ } a'=b'\in\mathrm{Im}(\phi_1),\\
        0 & \text{otherwise},
    \end{cases}
    \end{equation}
and for each $a,b\in S$, we have
    \begin{equation}\label{eq:ker}
        (\phi^\ddagger\circ\pmb\phi)_a^b=\begin{cases}
        \frac{1}{k_{t(\phi_1(a))}}\sqrt{\valphi{a}\valphi{b}} & \text{if\ } \phi_1(a)=\phi_1(b),\\
        0 & \text{otherwise}.
    \end{cases}
    \end{equation}
\end{lemma}
\begin{proof}
    We compute \begin{align*}
        (\pmb\phi\circ\phi^\ddagger)_{a'}^{b'}&=\sum_{p\in
        S,a'=\phi_1(p)=b'}\sqrt{\valphi{p}}\sqrt{\valphi{p}}\frac{1}{k_{t(b')}}
        =\delta_{a'}^{b'}\frac{1}{k_{t(a')}}\sum_{p\in S,\phi_1(p)=a'}\valphi{p}\\
        &=\begin{cases}
        \frac{1}{k_{t(a')}}k_{t(a')}=1 & \text{if\ } a'=b'\in\mathrm{Im}(\phi_1),\\
        0 & \text{otherwise}.
    \end{cases}
    \end{align*}
The equation~\eqref{eq:ker} is immediate by definition. Morphism $\pmb\phi\circ\phi^\ddagger$ is obviously idempotent,
and from $\pmb\phi\circ\phi^\ddagger\circ\pmb\phi=\pmb\phi$ it easily follows that
$\phi^\ddagger\circ\pmb\phi$ is idempotent as well.
\end{proof}
\begin{theorem}\label{thm:projOnCoimage}
    Let $(\phi_0,\phi_1)\colon (X,S)\rightarrow (X',S')$ be a strongly admissible morphism. Then
    $\phi^\ddagger\circ\pmb\phi$ is a semigroup endomorphism of $(S,\Dnabla)$ and a
    cosemigroup endomorphism of $(S,\Dnabla^\dagger)$.
\end{theorem}
\begin{proof}
    Since $\phi^\ddagger\circ\pmb\phi$ is  self-adjoint by \eqref{eq:diagonalK}, it is
    enough to prove that it is a semigroup morphism. 
The fact that $\phi^\ddagger\circ\pmb\phi$ is a semigroup endomorphism follows from \Cref{lemma:semigroupHom} by a straightforward
computation. As it is self-adjoint by \eqref{eq:diagonalK}, we get the second statement by applying the dagger.
\end{proof}

\begin{proposal}
\begin{remark}
    Recall the ``adjacency algebra'' functor $\mathcal{A}$ from the category of
    coherent configurations and strongly admissible morphisms
    to the category of algebras with non-unital morphisms.
    On the algebra side, daggerization corresponds to the
    change of basis
    \[
        u_a=\frac{1}{\sqrt{\val{a}}}\sigma_a,\quad a\in S.
    \]
    For a strongly admissible morphism
    $\phi=(\phi_0,\phi_1)\colon (X,S)\to (X',S')$,
    the coefficients of $\mathcal{A}(\phi)$ with respect
    to these bases are precisely $\pmb{\phi}_a^{b'}$.

    Equip each adjacency algebra with the inner product
    determined by
    \begin{equation}\label{eq:scalarProdForDagger}
        \langle u_a,u_b\rangle
        =|t(a)|\delta_{a}^{b}.
    \end{equation}

    Suppose that $\phi_1$ is surjective.
    Since $\phi_1$ is injective on unit colors,
    we have
    \[
        |t(a)|=k_{t(b')}|t(b')|
        \qquad\text{whenever }b'=\phi_1(a).
    \]
    The adjoint of $\mathcal{A}(\phi)$ with respect to
    these inner products therefore has coefficients
    \begin{equation}
        \bigl(\mathcal{A}(\phi)^\dag\bigr)_{b'}^a
        =\frac{|t(b')|}{|t(a)|}\pmb{\phi}_a^{b'}
        =\frac{1}{k_{t(b')}}\pmb{\phi}_a^{b'}
        =(\phi^\ddagger)_{b'}^a.
    \end{equation}
    Here we use that $\pmb{\phi}_a^{b'}=0$
    unless $b'=\phi_1(a)$.

    By \Cref{cor:semigroupHom},
    the adjoint $\mathcal{A}(\phi)^\dag$ is therefore
    also an algebra homomorphism, not necessarily
    preserving the unit. Moreover, it follows from \Cref{lem:projections} that $\mathcal{A}(\phi)$ is a partial isometry (with respect to the scalar product~\eqref{eq:scalarProdForDagger}).

\end{remark}
\end{proposal}
\section{The valency vector}
Let $(X,S)$ be a coherent configuration, and let $\mathcal{F}=(S,\Dnabla,e)$ be the
corresponding dagger Frobenius structure. Even though the dagger-structure constants
have better properties in computations (in particular when dealing with
admissible morphisms), their combinatorial meaning is not clear. It is 
natural to ask  whether we can reconstruct the original structure constants
from $\Dnabla$. In this section, we give an affirmative answer. Obviously,
knowing $\Dnabla$ and valencies of colors is sufficient to express $\nabla$. 
We will show how to extract the information
about valencies from $\Dnabla$ in the form of the \emph{valency vector} $\mathbf v\colon I\to S$ given by
the rule
\[
    {\mathbf v}_*^a=\sqrt{\val a}
\]
for all $a\in S$.
For each pair of ``composable'' colors $a,b$ ($t(a)=s(b)$), the corresponding valencies fit in the equation
\begin{equation}\label{eq:cop}
    \val{a}\val{b}=\sum_c \nabla^{c}_{ab}\val{c}. 
\end{equation}
This follows from an easy calculation, where the right-hand side is obtained by counting the number of two consecutive edges of color $a$ and $b$ starting at a given vertex. 

In the dagger Frobenius structure associated to a coherent configuration, ~\eqref{eq:cop} changes to
\begin{multline}\label{eq:copp}
    \sqrt{\val{a}\val{b}}=\frac{1}{\sqrt{\val{a}\val{b}}}\val{a}\val{b}=\frac{1}{\sqrt{\val{a}\val{b}}}\sum_c\nabla^c_{a b}\val{c}\\
    =\sum_c\sqrt{\frac{\val{c}}{\val{a}\val{b}}}\nabla^c_{a b}\sqrt{\val{c}}=\sum_c\Dnabla^c_{a b}\sqrt{\val{c}}.
\end{multline}
In the case of an association scheme, equation~\eqref{eq:copp} holds for each pair of colors $a$ and $b$, hence the
valency vector $\bf v$ is a \emph{copyable state}, \cite[Definition 4.23]{heunen2019categories}.

We can express valency vectors and gain insight into their properties
using the results from the previous section about strongly
admissible morphisms. 
Consider for some coherent configuration $(X,S)$ the admissible morphism
$!\colon(X,S)\rightarrow 1$. Following pattern~\eqref{eq:daggerAdmissibleEntries} for
the induced semigroup morphism $\pmb !\colon S\to I$, we derive
\begin{equation}
    \pmb !_a^*=\sqrt{\valphi{a}}=\sqrt{\val{a}}.
\end{equation}
This way, we can identify the valency vector $\mathbf{v}$ with the matrix $\pmb !^\dagger$.
If $(X,S)$ is a scheme, then $!$ is a strongly admissible surjective morphism, and we
can apply results from the previous section. For a general coherent configuration
$(X,S)$, we utilize the splitting of $!=!_S$ through the strongly admissible
$(\tau_0,\tau_1)\colon(X,S)\rightarrow \mathcal{D}_{E_S}$ described in \Cref{ex:tau}. Consequently, we can express the valency vector as
\begin{equation}\label{eq:splitOf!}
    \mathbf{v}=(\pmb !_S)^\dagger=(\pmb !_\mathcal{D}\circ\tau)^\dagger=\pmb\tau^\dagger\circ {!_\mathcal{D}}^\dagger.
\end{equation}
\begin{lemma}\label{lemma:valency}
    The valency vector of a coherent configuration $(X,S)$ is an eigenvector of
$\Dnabla\circ\Dnabla^\dagger$\footnote
{The operator $\nabla\circ\Delta$ is called the \emph{handle operator} in
\cite{kock2003frobenius}. In classical algebra, the handle operator is always a
multiplication by a central element. This element is called \emph{Casimir element}
or \emph{Euler element} in the literature.}
 with an eigenvalue $n=|X|$.
\end{lemma}
\begin{proof}
For the sake of simplicity, denote $\mathcal{D}=\mathcal{D}_{E_S}$, and let
$(\tau_0,\tau_1)\colon (X,S)\rightarrow \mathcal{D}$ be a strongly
admissible surjective morphism from \Cref{ex:tau}.  Applying
equation~\eqref{eq:splitOf!}, \Cref{cor:semigroupHom}, and
\Cref{cor:ddagger} yields the first three equalities in the following
string diagram computation:
\begin{equation}\label{eq:valencyIsEigen}
\vcbox{\includegraphics{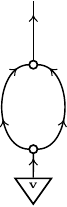}}
\hspace{.8em}\vcbox{$=$}\hspace{.8em}
\vcbox{\includegraphics{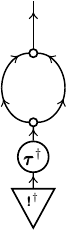}}
\hspace{.8em}
\vcbox{$=$}
\hspace{.8em}
\vcbox{\includegraphics{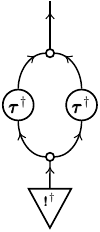}}
\hspace{.8em}
\vcbox{$=$}
\hspace{.8em}
\vcbox{\includegraphics{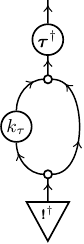}}
\hspace{.8em}
\vcbox{$=$}
\hspace{.8em}
\vcbox{\includegraphics{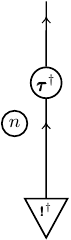}}
\hspace{.8em}
\vcbox{$=$}
\hspace{.8em}
\vcbox{\includegraphics{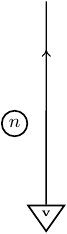}}
\end{equation}
Let us 
prove the fourth equality. For this, we first describe the entries of the diagonal matrix 
$k_\tau$. For a color $b=\{(v,u)\}$ of $\mathcal{D}$, 
where $(v,u)\in (E_S)^2$, the corresponding matrix entry is 
\begin{equation}\label{eq:kOfTau}
{(k_\tau)}^b_b=k_{t(b)}=|\tau_0^{-1}(u)|=|u|.    
\end{equation}
 Next, note that in $\mathcal{D}$, each color has valency equal to $1$. Hence,
 $\mathcal{D}$ is thin. Moreover, $\Dnabla^\mathcal{D}=\nabla^\mathcal{D}$ has all entries
 in $\{0,1\}$. One can prove by the same way as in \Cref{thm:connections}(a) that 
$$\Dnabla^{\mathcal{D}}\circ(k_\tau\otimes \id)\circ(\Dnabla^{\mathcal{D}})^\dagger=\nabla^{\mathcal{D}}\circ(k_\tau\otimes \id)\circ(\nabla^{\mathcal{D}})^\dagger$$
is a diagonal matrix. Computing the entries, we get
    \begin{equation}\label{eq:pipes}
        (\Dnabla^{\mathcal{D}}\circ(k_\tau\otimes
	\id)\circ(\Dnabla^{\mathcal{D}})^\dagger)_a^a=\sum_{b,c}{(k_\tau)}_b^b((\nabla^\mathcal{D})_{b
	c}^{a})^2=\sum_{u\in E_S}|u|=n,
    \end{equation}
    where we  used~\eqref{eq:kOfTau} in the second step.
\end{proof}

Our next aim is to show that we can identify the valency vector among the eigenvectors of $\Dnabla\circ\Dnabla^\dagger$. The key tool is the following:
\begin{theorem}[Perron--Frobenius]\label{thm:PF}\cite{meyer2023matrix}
    Let $A$ be a square matrix with strictly positive entries. Then $A$ has, up to scalar multiple, a unique positive eigenvector $\mathbf v$. Moreover, $\mathbf v$ corresponds to the largest eigenvalue.
\end{theorem}
We aim to apply the Perron--Frobenius theorem to $O=\Dnabla\circ\Dnabla^\dagger$ to identify
the valency vector. Note that if $O$ has all entries strictly positive, the valency vector
is identified as the positive eigenvector of the eigenvalue
$n$. In general, however,  we need to deal with the fact that $O$ might have some zero entries.

\begin{theorem}\label{thm:determineStructureConstants}
    Let $(S,\Dnabla,e)$ be a dagger Frobenius structure induced by a coherent configuration. Then the data $(S,\Dnabla,e)$ uniquely determines the structure constants $\nabla$.
\end{theorem}
\begin{proof}
Obviously, it is enough to determine the valency vector $\mathbf{v}$.
 Denote $O=\Dnabla\circ\Dnabla^\dagger$ and define a relation $\sim$ on the set of colors $S$ by the prescription
\begin{equation}
    a\sim b \Longleftrightarrow O_a^b\not= 0,
\end{equation}
in other words, $a\sim b$ if and only if there are some $p,q\in S$ such that
$\nabla^a_{pq}\nabla^b_{pq}>0$. 
It is easy to see that $\sim$ is reflexive and symmetric and there is some
$k\in\mathbb N$ such that $(O^k)_a^b\not=0$ if and only if $a$ and
$b$ belong to the same equivalence class of the transitive closure of $\sim$. Using the
corresponding decomposition of $S$, it is easy
to see that there is some permutation matrix $\Pi$ such that $\Pi O\Pi^\dagger$ is
block-diagonal, with blocks $O_1,\dots, O_m$, and we have 
\[
\Pi O^k\Pi^\dagger = (\Pi O\Pi^\dagger)^k=\mathrm{Diag}(O^k_1,\dots, O^k_m).
\]
Moreover,  for each $i=1,\ldots,m$, the block  $O_i^k$ has all entries positive.
As each 
$O_i$ is symmetric, the eigenvectors of $O_i$ are the same as the eigenvectors of
$O_i^k$. Consequently, by the Perron--Frobenius theorem, each $O_i$ has, up to a scalar
multiple, a unique positive eigenvector. By \Cref{lemma:valency}, $\Pi\mathbf{v}$ is an
eigenvector of $\Pi O\Pi^\dagger$ with eigenvalue $n=|X|$; hence, it decomposes as a direct sum
$\Pi\mathbf{v}=\mathbf{v}_1\oplus \cdots\oplus \mathbf{v}_m$, where each $\mathbf{v}_i$ is
an eigenvector of $O_i$ with eigenvalue $n$. Since each $\mathbf{v}_i$ consists of some
entries of the valency vector $\mathbf{v}$, it is strictly positive and hence  it is
given uniquely, up to a scalar multiple.  

Given any block $O_i$ such that the corresponding equivalence class contains at least one unit color, we can identify the
 part $\mathbf{v}_i$ of the valency vector as the
 unique positive eigenvector of $O_i$ having $1$'s on unit color coordinates. Hence, each
 color in the equivalence class of a unit color has determined valency.

Let now $a$ be any color, and let $P=\{b\in S : \Dnabla_{ab}^a>0\}$. 
Note that for $b\in
P$, $\nabla^b_{a^{-1}a}\nabla^{t(a)}_{a^{-1}a}>0$, so that $b\sim t(a)$ and hence $b$ has
determined valency by the previous paragraph. The proof is completed by the following
computation.
\begin{equation}
        \val{a}=\sum_{b\in S}\nabla^{a}_{a b}=\sum_{b\in P}\nabla^{a}_{a b}=\sum_{b\in P}\sqrt{\frac{\val{a}\val{b}}{\val{a}\val{b}}}\nabla^{a}_{a b}=\sum_{b\in P}\sqrt{\val{b}}\dnabla{a}{b}{a}.
    \end{equation}

\end{proof}

\section{The spectrum of $O_{(X,S)}$}

In the previous section, we have shown that the valency vector is one of the eigenvectors of
$\Dnabla\circ\Dnabla^\dag$. In this section, we will examine the spectrum of this matrix and then apply
the results to show how the spectrum behaves with respect to the surjective
admissible morphisms of association schemes. For some simple
classes of schemes, we will be able to characterize the spectrum.

From now on, for a dagger Frobenius structure associated to a
coherent configuration $(X,S)$, we denote the matrix $\Dnabla\circ\Dnabla^\dagger$ by
$O_{(X,S)}$, or simply by $O$, and its spectrum by $\Spec(O)$.
Note that $O$, being a symmetric matrix, is diagonalizable over $\mathbb R$.

In order to study the spectral data of $O$, in particular the eigenvectors, we also
need negative reals. Hence, it is convenient to pass from the category $\MatR$ to a
larger category  $\FHilb$. Following the (monoidal, dagger-preserving)
functor $L\colon\MatR\hookrightarrow\FHilb$ described in \Cref{sec:Lfunctor}, each 
coherent configuration also induces a dagger Frobenius structure in $\FHilb$. For a moment, we will work with a general dagger Frobenius structure $(V,\mu,\eta)$ in $\FHilb$, where $\mu$ denotes the multiplication and $\eta$ denotes the unit. 
\begin{lemma}\label{lem:specDec}
    Let $(V,\mu,\eta)$ be a dagger Frobenius structure in $\FHilb$ and $\mathbf{u}$, $\mathbf{v}$ be two eigenvectors of $\mu\circ\mu^\dagger$ with eigenvalues $\lambda_u$ and $\lambda_v$, respectively. If $\lambda_u=\lambda_v$, then $\mu\circ (\mathbf{u}\otimes\mathbf{v})$ is an eigenvector of $\mu\circ\mu^\dagger$ corresponding to $\lambda_u=\lambda_v$ as well. If $\lambda_u\not=\lambda_v$, then $\mu\circ (\mathbf{u}\otimes\mathbf{v})$ is the zero vector.
\end{lemma}
\begin{proof}
The statement is a consequence of the following equalities: 
\begin{equation}\label{eq:multiplicationPreservesEigenvectors}
    \lambda_u \cdot \mu\circ (\mathbf{u}\otimes\mathbf{v})=\mu\circ\mu^\dag\circ\mu\circ(\mathbf{u}\otimes\mathbf{v})=\lambda_v \cdot \mu\circ (\mathbf{u}\otimes\mathbf{v}).
\end{equation}
We will prove the first equality in~\eqref{eq:multiplicationPreservesEigenvectors} using string diagrams as follows:
\begin{center}
\vcbox{\includegraphics{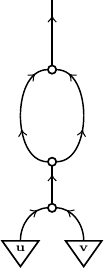}}
\vcbox{$=$}
\vcbox{\includegraphics{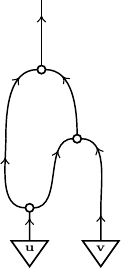}}
\vcbox{$=$}
\vcbox{\includegraphics{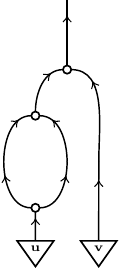}}
\vcbox{$=$}
\vcbox{\includegraphics{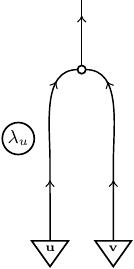}}
\end{center}
The first step follows from the Frobenius law, the second one applies
associativity, and the last follows from the assumption that $\mathbf{u}$ is an
eigenvector of $\mu\circ \mu^\dagger$.
\end{proof}

\begin{remark} Note that in the above string diagram equalities we did not use the fact
that $\mathbf{v}$ is an eigenvector, the same computation goes through for any vector in $V$.
Using the other Frobenius law similarly, we obtain that for any eigenvector
$\mathbf{u}$
with eigenvalue $\lambda_u$, both $\mu\circ(\mathbf{x}\otimes \mathbf{u})$ and
$\mu\circ(\mathbf{u}\otimes \mathbf{x})$ are eigenvectors with the same eigenvalue, for
any vector $\mathbf{x}$. It
follows that the eigenspaces of $\mu\circ\mu^\dagger$ are two-sided ideals in the algebra
given  by the monoid structure $(V,\mu,e)$.

\end{remark}

In $\FHilb$, each dagger Frobenius structure is a direct sum of special dagger Frobenius structures. 

\begin{proposition}\label{prop:decomposition}
    Let $\mathcal{F}=(V,\mu,\eta)$ be a dagger Frobenius structure in $\FHilb$ and 
    \begin{equation}\label{eq:eigenSpaceDec}
        V=\bigoplus_{\lambda\in\Spec(\mu\circ\mu^\dagger)} V_\lambda
    \end{equation} 
    be a spectral decomposition of $V$ with respect to $\mu\circ\mu^\dagger$. Then
    $\mathcal{F}$ admits a direct sum decomposition
\begin{equation}
    \mathcal{F}=\bigoplus_{\lambda\in\Spec(\mu\circ\mu^\dagger)} \mathcal{F}_\lambda.
\end{equation}
Here $\mathcal{F}_\lambda=(V_\lambda,\mu_\lambda,\eta_\lambda)$ is a dagger Frobenius structure on $V_\lambda$ satisfying $\mu_\lambda\circ\mu^\dagger_\lambda=\lambda\cdot\id$.
\end{proposition}
\begin{proof}
    From \Cref{lem:specDec} we obtain a decomposition
$\mu=\bigoplus_{\lambda\in\Spec}\mu_\lambda$, where
$\Spec=\Spec(\mu\circ\mu^\dagger)$. Applying dagger, we obtain an analogous direct
decomposition of the comultiplication $\mu^\dagger$. Similarly, consider the
decomposition $\eta=\bigoplus_{\lambda\in\Spec}\eta_\lambda$ with respect
to~\eqref{eq:eigenSpaceDec}. It is easy to deduce from \Cref{lem:specDec} that for each $\lambda\in\Spec$, the summand $\eta_\lambda$ is a unit for $\mu_\lambda$.  Finally, since $\mu_\lambda$ and $\mu^\dagger_\lambda$ are essentially (co)restrictions of $\mu$ and $\mu^\dagger$, we get that associativity, coassociativity, and the Frobenius law for $\mathcal{F}_\lambda$ follow from the corresponding rules for $\mathcal{F}$.
\end{proof}
\begin{proposition}
Let $(S,\Dnabla,e)$ be the dagger Frobenius structure associated to a  coherent
configuration on $n$ vertices. Then the spectrum satisfies $\Spec(\Dnabla\circ \Dnabla^\dagger)\subseteq (0,n]$.
\end{proposition}
\begin{proof}
   Clearly, $O=\Dnabla\circ \Dnabla^\dagger$ is a positive semi-definite matrix. By
the proof of \Cref{thm:determineStructureConstants}, we deduce that $O$ has $n$ as its maximal eigenvalue. It remains to prove $0$ is not an eigenvalue. 

Write $(V,\mu,\eta)$ for $(L(S),L(\Dnabla),L(e))$, so that we obtain a dagger Frobenius
structure in $\FHilb$. Let us assume $V_0$ is an eigenspace corresponding to an eigenvalue $\lambda=0$. 
   According to \Cref{prop:decomposition}, we have a dagger Frobenius structure $\mathcal{F}_0=(V_0,\mu_0,\eta_0)$ 
   in $\FHilb$, such that $\mu_0\circ\mu^\dagger_0$ is the zero operator. Therefore,
   $\mu_0$ is a zero operator as well, which forces the Frobenius algebra $\mathcal{F}_0$
   to be  trivial, with $V_0=\{0\}$, but then $V_0$ is not an eigenspace.
\end{proof}
Recall the self-adjoint idempotent ${\phi}^\ddagger\circ{\pmb \phi}$ from \Cref{lem:projections}. In $\FHilb$, it corresponds to an orthogonal
projection to $\mathrm{Im}(L({\pmb \phi}^\dagger))$. It turns out that this projection
commutes with the orthogonal projections to eigenspaces established in
 \Cref{prop:decomposition}.
\begin{proposition}
    Let $(\phi_0,\phi_1)\colon(X,S)\rightarrow(X',S')$ be a strongly admissible morphism
    of coherent configurations and let  $\phi^\ddagger\circ{\pmb \phi}\colon S\rightarrow
    S$ be the idempotent from \Cref{lem:projections}. Let $V=L(S)$ and let
    $\pi:=L(\phi^\ddagger\circ {\pmb \phi})\colon V\to V$. Let  $V=\oplus_{\lambda} V_\lambda$ be the decomposition~\eqref{eq:eigenSpaceDec}. Then $\pi$ admits decomposition $\pi=\oplus_\lambda\pi_\lambda$, where $\pi_\lambda$ is an idempotent endomorphism on $V_\lambda$.  
\end{proposition}

\begin{proof}
    It is enough to prove that $\pi$ preserves each eigenspace $V_\lambda$. It easily
    follows from \Cref{thm:projOnCoimage} and properties of the functor $L$ that
    $\pi$ is a semigroup and cosemigroup endomorphism for the corresponding structures on
    $V$. Consequently, for each eigenvector $\mathbf{v}\in V_\lambda$ we have:
\begin{equation}
\vcbox{\includegraphics{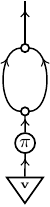}}
\hspace{.8em}\vcbox{$=$}\hspace{.8em}
\vcbox{\includegraphics{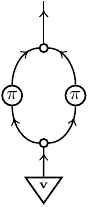}}
\hspace{.8em}\vcbox{$=$}\hspace{.8em}
\vcbox{\includegraphics{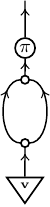}}
\hspace{.8em}\vcbox{$=$}\hspace{.8em}
\vcbox{\includegraphics{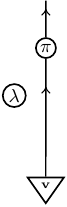}}
\end{equation}
\end{proof}
Revealing the formula of a trace of $O$, we get a sum over three variables. The following lemma gives us a formula with only two variables. The formula has a nice form under an assumption of $\frac{1}{2}$-homogeneity. However, note that in the proof we use this assumption only at the very last step.
\begin{lemma}
If $(X,S)$ is a $\frac{1}{2}$-homogeneous coherent configuration, then
    \begin{equation}\label{eq:trace}
    \mathrm{Tr}(\Dnabla\circ\Dnabla^\dagger)=\sum_{d\in S}\left(\sum_{a\in S}\Dnabla^d_{a a^{-1}}\right)^2.
\end{equation}
\end{lemma}
\begin{proof}
We compute
\begin{align*}
    \sum_{a,b,c}\dnabla{a}{b}{c}\dnabla{a}{b}{c}&\overset{\eqref{eq:dagRotR}}{=} \sum_{a,b,c}\sqrt{\frac{\val{b^{-1}}}{\val{b}}}\dnabla{b}{c^{-1}}{a^{-1}}\dnabla{a}{b}{c}\overset{\eqref{eq:dagOp}}{=}\sum_{a,b,c}\sqrt{\frac{\val{b^{-1}}}{\val{b}}}\dnabla{c}{b^{-1}}{a}\dnabla{a}{b}{c}\\
    &\overset{\text{(assoc.)}}{=}\sum_{b,c,d}\sqrt{\frac{\val{b^{-1}}}{\val{b}}}\dnabla{c}{d}{c}\dnabla{b^{-1}}{b}{d}\\
    &\overset{\eqref{eq:dagRotL}}{=}
    \sum_{b,c,d}\sqrt{\frac{\val{b^{-1}}}{\val{b}}}\sqrt{\frac{\val{c}}{\val{c^{-1}}}}\dnabla{c^{-1}}{c}{d^{-1}}\dnabla{b^{-1}}{b}{d}\\
    &\overset{\eqref{eq:dagOp}}{=}
    \sum_{b,c,d}\sqrt{\frac{\val{b^{-1}}}{\val{b}}}\sqrt{\frac{\val{c}}{\val{c^{-1}}}}\dnabla{c^{-1}}{c}{d}\dnabla{b^{-1}}{b}{d}\\
    &=\sum_{d}\bigg(\sum_a \sqrt{\frac{\val{a^{-1}}}{\val{a}}} \dnabla{a^{-1}}{a}{d}\bigg)\bigg(\sum_a \sqrt{\frac{\val{a}}{\val{a^{-1}}}} \dnabla{a^{-1}}{a}{d}\bigg).
\end{align*}
    Since $(X,S)$ is $\frac{1}{2}$-homogeneous, $\val{a}=\val{a^{-1}}$ and the last expression equals~\eqref{eq:trace}. 
\end{proof}
\begin{example}
Let $G$ be a finite group with $n$ elements and consider the corresponding thin
scheme $S(G)=(X,S)$ as in \Cref{ex:cayleyScheme}. In
$S(G)$, both vertices and colors can be addressed with the elements of $G$.
Then $\Dnabla\circ\Dnabla^\dagger$ is a diagonal $n\times n$ matrix in which
all diagonal elements are equal to $n$. One can deduce that by the following
computation
    \begin{equation}
        (\Dnabla\circ\Dnabla^\dagger)_a^b=(\nabla\circ\nabla^\dagger)_a^b=\sum_{c,d\in S, a=c\cdot d=b}\nabla^a_{cd}\nabla^b_{cd}=n\cdot\delta_a^b.
    \end{equation}
\end{example}

Consider a surjective morphism of finite groups $f\colon G\rightarrow H$.
According to~\cite{french2013functors}, $f$ induces an admissible morphism of schemes
$S(f)\colon S(G)\rightarrow S(H)$. For groups, it is clear that $|H|$ is a divisor of
$|G|$. This fact directly generalizes to $\phi_1$-surjective admissible
morphisms $(\phi_0,\phi_1)\colon (X,S)\to (X',S')$ of general association schemes.
Indeed, such a morphism is strongly admissible by \Cref{coro:defsCoincide}, so by
\eqref{eq:kp}, $|\phi_0^{-1}(x')|$ is the same number for every vertex $x'$ of the
codomain scheme; hence $|X'|$ is a divisor of $|X|$,
so that $\frac{|X|}{|X'|}$ is an integer. By the results of the previous
section, this exactly means that the largest eigenvalue of $O_{(X',S')}$ divides the
largest eigenvalue of $O_{(X,S)}$. The following theorem shows that, more generally,
$\frac{|X|}{|X'|}\cdot\lambda$ is an eigenvalue of $O_{(X,S)}$ for every eigenvalue
$\lambda$ of $O_{(X',S')}$.

\begin{theorem}\label{thm:genLagrange}
Let $(\phi_0,\phi_1)\colon (X,S)\rightarrow (X',S')$ be an admissible
morphism of association schemes with $\phi_1$ surjective, and let \((\lambda_1,\ldots,\lambda_m)\),
$m=|S'|$, be the vector of the eigenvalues of $O_{(X',S')}$ (counting
multiplicities). Put $k=|X|/|X'|$ (an integer by the discussion above). Then
$(k\cdot\lambda_1,\ldots,k\cdot\lambda_m)$ is a subvector of the vector of
the eigenvalues of $O_{(X,S)}$.
\end{theorem}
\begin{proof}
Since we have schemes, the diagonal matrix $k_\phi$ (see
formulas~\eqref{eq:diagonalK}) has all its diagonal entries equal to $k$.
We claim that if $\mathbf{v}$ is an eigenvector of $(X',S')$
corresponding to eigenvalue
$\lambda$, then $\pmb\phi^\dagger\circ\mathbf{v}$ is an eigenvector of $(X,S)$, corresponding to eigenvalue
$k\cdot\lambda$.  This is enough, since $\pmb\phi^\dagger$ has a left inverse
$\frac{1}{k}\pmb\phi$ by \Cref{lem:projections}, hence it has a trivial kernel. Let us prove the claim with the following string diagram
calculation:
\begin{center}
\vcbox{\includegraphics{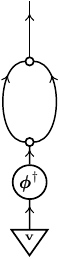}}
\vcbox{~$=$~}
\vcbox{\includegraphics{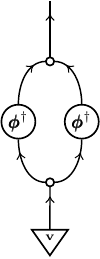}}
\vcbox{~$=$~}
\vcbox{\includegraphics{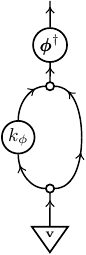}}
\vcbox{~$=$~}
\vcbox{\includegraphics{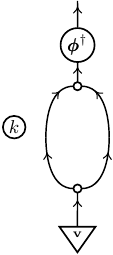}}
\vcbox{~$=$~}
\vcbox{\includegraphics{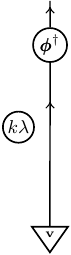}}
\end{center}
The first equality follows by (the dagger variant of) \Cref{cor:daggerMonHom}, the second equality follows from
\Cref{cor:ddagger}, and the third one is due to $k_\phi$ being a constant
diagonal matrix. 
\end{proof}

\Cref{thm:genLagrange} gives a necessary condition for the existence of a
$\phi_1$-surjective admissible morphism between schemes, which can be checked by
comparing the $O$-spectra. We will demonstrate this in several examples.

\begin{example}\label{ex:completeGraph}
    Consider a trivial scheme $\mathcal{T}_n$, $n\geq 2$, which has one unit color and one nonunit color of valency $n-1$. A straightforward computation leads to
    \begin{equation}
        O= \Dnabla\circ\Dnabla^\dagger =
        \begin{bmatrix}
    2 & \frac{n-2}{\sqrt{n-1}} \\
    \frac{n-2}{\sqrt{n-1}} & \frac{(n-1)^2+1}{n-1}
    \end{bmatrix}.
    \end{equation}
The spectrum of $O$ equals $\Spec(O)=\{n,\frac{n}{n-1}\}$. Eigenvalue $n$
corresponds to a valency eigenvector $\mathbf{v}=(1,\sqrt{n-1})^T$ and
eigenvalue $\frac{n}{n-1}$ corresponds to an eigenvector
$\mathbf{u}=(\sqrt{n-1},-1)^T$. 
The multiplication and comultiplication act on the eigenvectors as follows
    \begin{align*}
        \Dnabla^\dagger\circ\mathbf{v}= \mathbf{v}\otimes\mathbf{v} &\quad\text{and}\quad \Dnabla\circ(\mathbf{v}\otimes\mathbf{v})= n\cdot \mathbf{v},\\
         \Dnabla^\dagger\circ\mathbf{u}= \frac{1}{\sqrt{n-1}}\cdot\mathbf{u}\otimes\mathbf{u}
         &\quad\text{and}\quad\Dnabla\circ(\mathbf{u}\otimes\mathbf{u})= \frac{n}{\sqrt{n-1}}\cdot \mathbf{u},\\
         \Dnabla \circ(\mathbf{v}\otimes\mathbf{u})&=0=\Dnabla\circ(\mathbf{u}\otimes\mathbf{v}).
    \end{align*}

Consider now two schemes $\mathcal T_n$ and $\mathcal T_{kn}$, where $n\geq 2$ and
$k\geq 1$ are integers.
\Cref{thm:genLagrange} implies that if $k\neq 1$, there is no $\phi_1$-surjective
admissible morphism from $\mathcal T_{kn}$ to $\mathcal T_n$. Indeed,
we see that
\[
\Spec(O_{\mathcal T_n})=\{n,\frac{n}{n-1}\}\qquad
\Spec(O_{\mathcal T_{kn}})=\{kn,\frac{kn}{kn-1}\}.
\]
For $n\geq 3$, $k\cdot\frac{n}{n-1}\neq kn$ and
\[
k\cdot\frac{n}{n-1}=\frac{kn}{kn-1}
\]
if and only if $k=1$. For $n=2$, the eigenvalue $2$ of $O_{\mathcal T_2}$ has
multiplicity $2$, while for $k\geq 2$ the eigenvalue $2k$ of $O_{\mathcal T_{2k}}$ has
multiplicity $1$, because $\frac{2k}{2k-1}\neq 2k$.

\end{example}
\begin{example}\label{ex:wiki}
Consider a pair $(\phi_0,\phi_1)\colon (X,S)\rightarrow \mathcal{T}_3$ and
$(\psi_0,\psi_1)\colon (X,S)\rightarrow \mathcal{T}_2$ of admissible morphisms of
association schemes given by \Cref{fig:exWiki}. Let us label the colors of scheme $(X,S)$ by integers $0,1,2,3$. 
\begin{figure}
    \centering
    \includegraphics[width=0.95\linewidth]{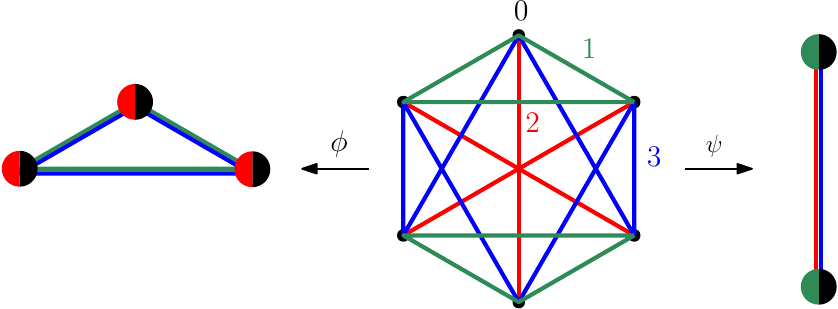}
    \caption{The association scheme $\mathcal T_3\otimes\mathcal T_2$}
    \label{fig:exWiki}
\end{figure}
For this scheme, we have
\begin{equation}
    O = \begin{bmatrix}
    4 & \sqrt{2} & 0 & 0 \\
    \sqrt{2} & 5 & 0 & 0 \\
    0 & 0 & 4 & \sqrt{2} \\
    0 & 0 & \sqrt{2} & 5
    \end{bmatrix}.
\end{equation}
The spectrum has two elements, $\Spec(O)=\{3,6\}$. The eigenvalue $6$ equals the number of vertices, and the corresponding eigenspace is generated by two vectors
$\mathbf{v}_1=(1,\sqrt{2},1,\sqrt{2})^T$ and $\mathbf{v}_2=(1,\sqrt{2},-1,-\sqrt{2})^T$. The eigenspace corresponding to $3$ is generated by $\mathbf{u}_1=(\sqrt{2},-1,\sqrt{2},-1)^T$ and $\mathbf{u}_2=(\sqrt{2},-1,-\sqrt{2},1)^T$.

The matrices $\pmb\phi$ and $\pmb\psi$ corresponding to the two admissible morphisms in concern are 
\begin{equation}
    \pmb\phi = \begin{bmatrix} 
    1 & 0 & 1 & 0 \\ 
    0 & 1 & 0 & 1
    \end{bmatrix}, \quad
    \pmb\psi = \begin{bmatrix} 
    1 & \sqrt{2} & 0 & 0 \\ 
    0 & 0 & 1 & \sqrt{2} 
    \end{bmatrix}
    \label{eq:matrices_2x4}
\end{equation}
Denote by $\mathbf{v}'=(1,1)^T$, $\mathbf{u}'=(1,-1)^T$ and
$\mathbf{v}''=(1,\sqrt{2})^T$, $\mathbf{u}''=(\sqrt{2},-1)^T$ the eigenvectors from
\Cref{ex:completeGraph}, for $\mathcal{T}_2$ and $\mathcal{T}_3$.
Note that $\mathbf{v}_1=\pmb\phi^\dagger\mathbf{v}''=\pmb\psi^\dagger\mathbf{v}'$ is the valency vector. Moreover, $\mathbf{v}_2=(1,\sqrt{2},-1,-\sqrt{2})^T=\pmb\psi^\dagger\mathbf{u}'$ and $\mathbf{u}_1=(\sqrt{2},-1,\sqrt{2},-1)^T=\pmb\phi^\dagger\mathbf{u}''$.

Moreover, observe that $k_\phi$ is a diagonal matrix having $2$'s on its
diagonal and the spectrum of $\mathcal{T}_3$ has two eigenvalues $3$ and
$\frac{3}{2}$. In accordance with \Cref{thm:genLagrange}, $2\cdot 3$ and
$2\cdot \frac{3}{2}$ belong to the spectrum of $O_{(X,S)}$. Similarly, $k_\psi$ has
$3$'s on the diagonal and $\mathcal{T}_2$ has in spectrum twice $2$. Indeed,
the value $3\cdot 2$ appears twice in the spectrum of $O_{(X,S)}$. 
\end{example}
\begin{lemma}
Let $(X,S)$ and $(Y,T)$ be coherent configurations with spectra $\Spec_S$
and $\Spec_T$. Then the spectrum of the tensor product configuration $(X,S)\otimes(Y,T)$
is the elementwise product: $\{\lambda\cdot\rho\mid
\lambda\in\Spec_S,\rho\in\Spec_T\}$.
\end{lemma}
\begin{proof}
This is an easy consequence of the fact that $\Dnabla_{(X,S)\otimes(Y,T)}\simeq\Dnabla_{(X,S)}\otimes\Dnabla_{(Y,T)}$.
\end{proof}
In \Cref{ex:wiki}, the association scheme $(X,S)$ is isomorphic to
$\mathcal{T}_3\otimes \mathcal{T}_2$. Therefore, we get the spectral data for $S$
by multiplying those for $\mathcal{T}_3$ and $\mathcal{T}_2$.

In the following example, there is an obvious symmetry, which reflects the presence of an automorphism switching two colors. 
\begin{example}\label{ex:D5spectrum}
Let $S$ be the Schurian scheme $D_5$ defined by \Cref{fig:pentagon}. We have
\begin{equation}
    O:=(\Dnabla\circ\Dnabla^\dagger)_a^b = \begin{bmatrix}
    3 & \frac{1}{\sqrt{2}} & \frac{1}{\sqrt{2}} \\
    \frac{1}{\sqrt{2}} & \frac{7}{2} & 1 \\
    \frac{1}{\sqrt{2}} & 1  & \frac{7}{2}
    \end{bmatrix}.
\end{equation}
    The spectrum equals $\Spec=\{5,\frac{5}{2}\}$. Note that $O$ is invariant under  the symmetry $\Pi:
    \mathbb R^3\to \mathbb R^3$ given as $\Pi(x,y,z)=(x,z,y)$,  which
    implies that the eigenspaces are invariant under $\Pi$. 
The eigenvalue $5$ corresponds to the valency vector $\mathbf{v}=(1,\sqrt{2},\sqrt{2})$ which
clearly satisfies $\Pi\mathbf{v}=\mathbf{v}$. The eigenvalue $\frac{5}{2}$ has eigenspace of dimension $2$ generated by 
the vectors $\mathbf{u}_1$ and $\mathbf{u}_2=\Pi\mathbf{u}_1$, where we may choose 
    \begin{equation}
        \mathbf{u}_1=\bigl(\frac{2}{5},\frac{-\sqrt{2}-\sqrt{10}}{10},\frac{\sqrt{10}-\sqrt{2}}{10}\bigr).
    \end{equation}
    Note that then 
    \begin{equation}
        \langle \mathbf{u}_i,\mathbf{u}_j\rangle =\frac{2}{5}\delta_i^j,\quad \mu(\mathbf{u}_i\otimes\mathbf{u}_j)=\delta_i^j \mathbf{u}_i.
    \end{equation}
\end{example}
\begin{example}
    Denote by $S$ the association scheme from \Cref{fig:petersen} induced by the Petersen graph.
    \begin{figure}
        \centering
        \includegraphics[width=0.5\linewidth]{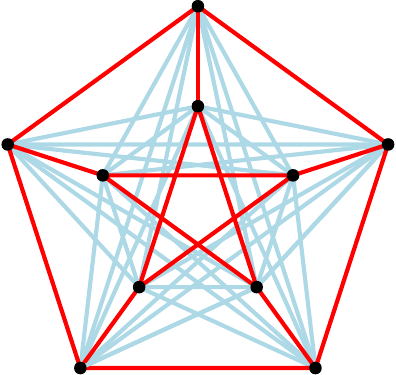}
        \caption{The association scheme induced by the Petersen graph}
        \label{fig:petersen}
    \end{figure}
    Even though the matrix
    \[
            \Dnabla\circ\Dnabla^\dagger =
            \begin{bmatrix}
                3 & \frac{2}{3}\sqrt{3} & \frac{5}{6}\sqrt{6}\,\\[4pt]
                \frac{2}{3}\sqrt{3} & \frac{14}{3} & \frac{7}{3}\sqrt{2} \,\\[4pt]
                \frac{5}{6}\sqrt{6} & \frac{7}{3}\sqrt{2} & \frac{41}{6}\,
            \end{bmatrix}
            \]
            is rather complicated, the spectrum consists of only rational numbers
$\Spec=\{10,\frac{5}{2},2\}$.

By \Cref{thm:genLagrange}, there is no $\phi_1$-surjective admissible morphism from
$S$ onto the scheme $D_5$ from \Cref{ex:D5spectrum}, because $2\cdot\frac{5}{2}=5$ is not
in the spectrum of the Petersen graph.
\end{example}
\begin{proposal}
\begin{example}\label{ex:hoffman_singleton}
    Let $S$ be the association scheme induced by the Hoffman--Singleton graph
    \cite{hoffman1960moore}, the unique strongly regular graph with parameters
    $(50,7,0,1)$. Its colors are the distance relations of the graph, with
    valencies $1$, $7$ and $42$. We have
    \[
            \Dnabla\circ\Dnabla^\dagger =
            \begin{bmatrix}
                3 & \frac{6}{7}\sqrt{7} & \frac{41}{42}\sqrt{42}\,\\[4pt]
                \frac{6}{7}\sqrt{7} & \frac{62}{7} & \frac{47}{7}\sqrt{6} \,\\[4pt]
                \frac{41}{42}\sqrt{42} & \frac{47}{7}\sqrt{6} & \frac{1777}{42}\,
            \end{bmatrix}
    \]
    and the spectrum is $\Spec=\{50,\frac{25}{14},\frac{50}{21}\}$, each
    eigenvalue having multiplicity $1$.

    By \Cref{thm:genLagrange}, there is no $\phi_1$-surjective admissible morphism
    from $S$ onto the scheme induced by the Petersen graph, because
    $5\cdot\frac{5}{2}=\frac{25}{2}$ is not in $\Spec$. Similarly, there is none
    onto the scheme $D_5$ from \Cref{ex:D5spectrum}, because $10\cdot\frac{5}{2}=25$
    is not in $\Spec$.
\end{example}
\end{proposal}
\begin{example}
The smallest non-Schurian coherent configuration was found in \cite{klin2017non}. It has rank 11 and degree
14. This example is not $\frac{1}{2}$-homogeneous, there are two unit colors with cardinalities
$6$ and $8$.
The spectrum of $\theO$ is
\begin{center}
\begin{tabular}{|r|r|}
\hline
eigenvalue & multiplicity\\
\hline
14& 4\\
 3& 1\\
 $\frac{14}{3}$& 4\\
 4& 2\\
\hline
\end{tabular}
\end{center}
\end{example}

\begin{example}\label{ex:half_homo}
Let  $(X,S)$ be $\frac12$-homogeneous. As we observed in
\Cref{coro:hstar} and the
paragraph below it, our dagger Frobenius structure is symmetric
and coincides with an $H^*$-algebra structure of the adjacency algebra $\mathcal A$. In
particular, the matrix $O_{(X,S)}$ becomes the matrix of the operator 
\[
    m\cdot(\nabla_{\mathcal A}\circ\nabla_{\mathcal A}^\dag)
\]
in the normalized basis $\{\bar \sigma_a,\ a\in S\}$, where $m$ is the cardinality of
unit colors and $(\mathcal A,\nabla_{\mathcal
A},e_{\mathcal A})$ is the corresponding symmetric dagger Frobenius structure on $\mathcal A$.
We now show how the spectrum of $O_{(X,S)}$ is related to the structure of $\mathcal A$. 

Since $\mathcal A$ is a $*$-subalgebra of $M_n(\mathbb C)$, we have a direct sum
decomposition of the form $\mathbb C^n\simeq \bigoplus_i \mathbb C^{d_i}\otimes \mathbb C^{k_i}$, such that
(up to a unitary conjugation)
\[
\mathcal A= \bigoplus_i M_{d_i}(\mathbb C)\otimes I_{k_i}.
\]
The inner product induced in each of the factors is  
\[(A,B)\mapsto \mathrm{Tr}\,
[(A^*B)\otimes I_{k_i}]=k_i\mathrm{Tr}\,[A^*B],\qquad A,B\in M_{d_i}(\mathbb C).
\] The $H^*$-algebra $\mathcal A$ is thus
isomorphic to $\bigoplus_i B(\mathbb C^{d_i}, k_i)$, in the notation of
\cite[Section 5.4]{heunen2019categories}. By \cite[Prop.~5.33]{heunen2019categories}, the symmetric
dagger Frobenius structure in the factors is given by the scaled pair of pants structure on $(\mathbb C^{d_i})^*\otimes
\mathbb C^{d_i}$, with the multiplication and the unit given by
\begin{center}
\vcbox{\includegraphics{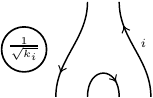}}
\vcbox{\hspace{3em}}
\vcbox{\includegraphics{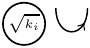}}
\end{center}
and the Frobenius structure in $\mathcal A$ is
then their direct sum. In particular,
\begin{center}
\vcbox{$\nabla_{\mathcal A}\circ\nabla_{\mathcal A}^\dag=\bigoplus_i$}
\vcbox{\includegraphics{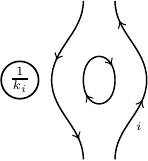}}
\vcbox{$=\bigoplus_i\frac{d_i}{k_i} I_{d_i^2}$}
\end{center}
This shows that $\Spec(O)=\{ m\frac{d_i}{k_i}\}_i$. In particular, the spectrum contains only rational numbers.
\end{example}

\begin{example}\label{ex:permut} 
Let us describe the spectrum of $O$ for Schurian coherent configurations arising from
an action of $\mathbb Z$, that means, from a single permutation.
Let $X$ be a finite set and let $\Gamma$ be a permutation
on $X$. Let 
$(x,y)\sim (x',y')$ if $x'=\Gamma^k(x)$, $y'=\Gamma^k(y)$ for some $k\in \mathbb Z$ and
let 
$S$ be the set of equivalence classes of $\sim$. Then $(X,S)$ is a coherent configuration on $X$. 
The permutation $\Gamma$ can be decomposed
into cycles $C_1,\dots, C_m$, with lengths $n_l=|C_l|$, $\sum_l n_l=n:=|X|$. Two vertices
in $X$ have the same unit color iff they are in the same cycle, so that $(X,S)$ is not
$\frac12$-homogeneous if $n_l\ne n_{l'}$ for some $l$, $l'$. 

Computing  the form of the matrices $\nabla$, $\Dnabla$, $O$ and the spectrum of $O$ in this case is not difficult, but somewhat lengthy. We will
only report the general description of  $\Spec(O)$.  Let $N:=\LCM(n_1,\dots,n_m)$ and let
$\mathcal L$ be the sublattice generated by $n_1,\dots, n_m$ in the divisibility lattice
$\mathcal L_D(N)$. Any join irreducible element in $\mathcal L$ has the form $n_T:=\wedge_{l\in T} n_l$ for
some $T\subseteq [m]:=\{1,\dots,m\}$, with the lattice operations in $\mathcal L_D(N)$. 
Let 
\[\mathcal T:=\{\emptyset \neq T\subseteq [m],\ n_T \text{ is join irreducible in }
\mathcal L \text{ or } T=[m]\}.
\]
We then have 
\[
\Spec(O)=\{\lambda_T:=\sum_{l\in T} n_l,\ T\in \mathcal T\}.
\]
Moreover, the multiplicity of each  eigenvalue $\lambda\in \Spec(O)$ is obtained as
\[
\mu(\lambda)=\sum_{T\in \mathcal T, \lambda=\lambda_T} \chi_T|T|^2,\qquad
\chi_T=\begin{dcases} (n_T-n_T\wedge (\vee_{l\not\in T} n_l)) &\text{if } T\neq
[m]\\
n_{[m]} & \text{otherwise}.
\end{dcases}
\]
For example, assume that all $n_l$ are pairwise coprime. Then $\mathcal
T$ consists of all singletons and the whole set $[m]$. The eigenvalues are 
\[
\lambda=\begin{dcases} n_l,\ l=1,\dots,m & \text{with multiplicity } n_l-1\\
 n=\sum_l n_l & \text{with multiplicity } m^2.
 \end{dcases}
 \]
For the other extreme example $n_1=\dots=n_m=k$, the sublattice $\mathcal L$ consists of a
single point and hence has no join-irreducible elements. In this case, $O$ has a single
eigenvalue $n$, with multiplicity $km^2=|S|$. 
\end{example}

\begin{proposal}
A natural question to ask at this point is whether there is a 
pair of coherent configurations 
$(X,S)$, $(Y,T)$ with $|S|=|T|$ such that the valency vectors for them are the same up to a permutation, 
but the spectra of $O_{(X,S)}$ and $O_{(Y,T)}$ are different. In the submitted
manuscript of this paper, we formulated this question as an open problem. We are very
thankful to the anonymous reviewer who has pointed us to the following pair of association schemes,
which are denoted by {\tt as08n5} and {\tt as08n6} on the web page
\cite{hanaki2019classification}. So the answer to our question is ``yes''.
\begin{example}
The association scheme {\tt as08n5} has degree 8, rank 4 and coloring
\[
\begin{array}{c|cccccccc}
 & 1 & 2 & 3 & 4 & 5 & 6 & 7 & 8 \\
\hline
1 & 0 & 1 & 2 & 2 & 2 & 3 & 3 & 3 \\
2 & 1 & 0 & 3 & 3 & 3 & 2 & 2 & 2 \\
3 & 2 & 3 & 0 & 2 & 2 & 1 & 3 & 3 \\
4 & 2 & 3 & 2 & 0 & 2 & 3 & 1 & 3 \\
5 & 2 & 3 & 2 & 2 & 0 & 3 & 3 & 1 \\
6 & 3 & 2 & 1 & 3 & 3 & 0 & 2 & 2 \\
7 & 3 & 2 & 3 & 1 & 3 & 2 & 0 & 2 \\
8 & 3 & 2 & 3 & 3 & 1 & 2 & 2 & 0 \\
\end{array}
\]
The association scheme {\tt as08n6} has the same degree and rank and its coloring is
\[
\begin{array}{c|cccccccc}
 & 1 & 2 & 3 & 4 & 5 & 6 & 7 & 8 \\
\hline
1 & 0 & 1 & 2 & 2 & 2 & 3 & 3 & 3 \\
2 & 1 & 0 & 3 & 3 & 3 & 2 & 2 & 2 \\
3 & 3 & 2 & 0 & 2 & 3 & 1 & 2 & 3 \\
4 & 3 & 2 & 3 & 0 & 2 & 2 & 3 & 1 \\
5 & 3 & 2 & 2 & 3 & 0 & 3 & 1 & 2 \\
6 & 2 & 3 & 1 & 3 & 2 & 0 & 3 & 2 \\
7 & 2 & 3 & 3 & 2 & 1 & 2 & 0 & 3 \\
8 & 2 & 3 & 2 & 1 & 3 & 3 & 2 & 0 \\
\end{array}
\]
In both cases, the valencies are $(1,1,3,3)$. On the other hand, their $O$-spectra are
\begin{center}
\begin{tabular}{|r|r|}
\hline
\multicolumn{2}{|c|}{{\tt as08n5}}\\
\hline
eigenvalue & multiplicity\\
\hline
8& 2\\
$\frac{8}{3}$& 2\\
\hline
\end{tabular}
$\quad$
\begin{tabular}{|r|r|}
\hline
\multicolumn{2}{|c|}{{\tt as08n6}}\\
\hline
eigenvalue & multiplicity\\
\hline
8& 1\\
$\frac{8}{3}$& 1\\
4& 2\\
\hline
\end{tabular}
\end{center}
so they are distinct.
\end{example}
\end{proposal}

\section{Future research, open problems}

In future papers, we intend to consider a variety of topics and themes; for example, the following.
\begin{itemize}
\item We would like to describe the relationships between various types of morphisms we introduced, from the categorical perspective.
\item We would like to understand our results in the context of classical algebra, in
particular of the various types of algebras with a fixed basis surveyed in
\cite{blau2009table}.
\item A coherent configuration is a module. A module over a monoid in a compact category is the same thing as the
representation of that monoid. There is a theory of representations of Frobenius algebras (see for example
\cite{skowronski2017frobenius}), that could be applicable in the theory of coherent
configurations.
\item In \cite{kuribayashi2015association}, Kuribayashi and Matsuo introduced
\emph{association (quasi-)schemoids}, a generalization of coherent configurations. We
would like to extend some of our results to this setting.
\end{itemize}

\begin{problem}
Is it true that for every coherent configuration $(X,S)$, the spectrum of $O_{(X,S)}$ contains
only rational numbers? By \Cref{ex:half_homo}, this is true for
all $\frac{1}{2}$-homogeneous coherent configurations, hence also for all association
schemes.
\end{problem}
\begin{problem}
Is it true that for every coherent configuration $(X,S)$, the spectrum of $O_{(X,S)}$ is
lower bounded by $\frac{|X|}{|X|-1}$?
\end{problem}

\section*{Acknowledgements}
\begin{proposal}
We are grateful to the anonymous referee for a careful reading of the manuscript and
for many valuable comments. In particular, the referee's remarks on the history of
morphisms of association schemes, on coherent and table algebras, and on the
terminology used in the literature on coherent configurations helped us to place our
results in a broader context. The referee also suggested studying the effect of
admissible morphisms on adjacency algebras, and pointed us to the pair of association
schemes that resolved one of the open problems of the original version of this paper.
\end{proposal}


\begin{thebibliography}{10}

\bibitem{abramsky2009categorical}
Samson Abramsky and Bob Coecke.
\newblock {C}ategorical {Q}uantum {M}echanics.
\newblock In Kurt Engesser, Dov~M. Gabbay, and Daniel Lehmann, editors, {\em
  Handbook of Quantum Logic and Quantum Structures}, pages 261--32. Elsevier,
  Amsterdam, 2009.

\bibitem{bailey2004association}
Rosemary~A Bailey.
\newblock {\em Association schemes: Designed experiments, algebra and
  combinatorics}, volume~84.
\newblock Cambridge University Press, 2004.

\bibitem{blau2009table}
Harvey~I. Blau.
\newblock Table algebras.
\newblock {\em European Journal of Combinatorics}, 30:1426--1455, 2009.

\bibitem{bose1939partially}
Raj~Chandra Bose and K~Raghavan Nair.
\newblock Partially balanced incomplete block designs.
\newblock {\em Sankhy{\=a}: The Indian Journal of Statistics}, pages 337--372,
  1939.

\bibitem{bose1952classification}
Ray~C Bose and Takatsugu Shimamoto.
\newblock Classification and analysis of partially balanced incomplete block
  designs with two associate classes.
\newblock {\em Journal of the American statistical association},
  47(258):151--184, 1952.

\bibitem{chen2024lectures}
Gang Chen and Ilia Ponomarenko.
\newblock Lectures on coherent configurations (update june, 2024).
\newblock
  \url{https://web.archive.org/web/20250202155408/https://www.pdmi.ras.ru/~inp/ccNOTES.pdf},
  2024.

\bibitem{coecke2017picturing}
Bob Coecke and Aleks Kissinger.
\newblock {\em Picturing quantum processes: A first course in quantum theory
  and diagrammatic reasoning}.
\newblock Cambridge University Press, 2017.

\bibitem{ding2014matrix}
Jiu Ding and Noah~H Rhee.
\newblock When a matrix and its inverse are nonnegative.
\newblock {\em Missouri Journal of Mathematical Sciences}, 26(1):98--103, 2014.

\bibitem{french2013functors}
Christopher French.
\newblock Functors from association schemes.
\newblock {\em Journal of Combinatorial Theory, Series A}, 120(6):1141--1165,
  2013.

\bibitem{godsil2010association}
Chris Godsil.
\newblock Association schemes.
\newblock Lecture notes, University of Waterloo,
  \url{https://math.uwaterloo.ca/~cgodsil/pdfs/assoc2.pdf}, 2010.
\newblock Version of June 3, 2010.

\bibitem{hanaki2010acategory}
Akihide Hanaki.
\newblock A category of association schemes.
\newblock {\em Journal of Combinatorial Theory, Series A}, 117(8):1207--1217,
  2010.

\bibitem{hanaki2019classification}
Akihide Hanaki and Izumi Miyamoto.
\newblock Classification of association schemes.
\newblock \url{http://math.shinshu-u.ac.jp/~hanaki/as/}, 2019.
\newblock Accessed: 2026-09-15.

\bibitem{hanaki2015thin}
Akihide Hanaki and Masayoshi Yoshikawa.
\newblock Thin coherent configurations and groupoids.
\newblock {\em Journal of Algebra \& Its Applications}, 14(1), 2015.

\bibitem{heunen2013relative}
Chris Heunen, Ivan Contreras, and Alberto~S Cattaneo.
\newblock Relative {F}robenius algebras are groupoids.
\newblock {\em Journal of Pure and Applied Algebra}, 217:114--124, 2013.

\bibitem{heunen2019categories}
Chris Heunen and Jamie Vicary.
\newblock {\em Categories for Quantum Theory: an introduction}.
\newblock Oxford University Press, 2019.

\bibitem{higman1975coherent}
DG~Higman.
\newblock Coherent configurations: Part i: Ordinary representation theory.
\newblock {\em Geometriae Dedicata}, 4(1):1--32, 1975.

\bibitem{higman1987coherent}
Donald~G. Higman.
\newblock Coherent algebras.
\newblock {\em Linear Algebra and its Applications}, 93:209--239, 1987.

\bibitem{hoffman1960moore}
Alan~J. Hoffman and Robert~R. Singleton.
\newblock On {M}oore graphs with diameters 2 and 3.
\newblock {\em IBM Journal of Research and Development}, 4(5):497--504, 1960.

\bibitem{kenney2011categories}
Toby Kenney and Robert Paré.
\newblock Categories as monoids in {S}pan, {R}el and {S}up.
\newblock {\em Cahiers de topologie et g{\'e}om{\'e}trie diff{\'e}rentielle
  cat{\'e}goriques}, 52(3):209--240, 2011.

\bibitem{klin2017non}
Mikhail Klin and Matan Ziv-Av.
\newblock A non-schurian coherent configuration on 14 points exists.
\newblock {\em Designs, Codes and Cryptography}, 84(1):203--221, 2017.

\bibitem{kock2003frobenius}
Joachim Kock.
\newblock {\em Frobenius Algebras and 2-D Topological Quantum Field Theories}.
\newblock Number~59 in London Mathematical Society Student Texts. Cambridge
  University Press, Cambridge, 2003.

\bibitem{kuribayashi2015association}
Katsuhiko Kuribayashi and Kentaro Matsuo.
\newblock Association schemoids and their categories.
\newblock {\em Applied Categorical Structures}, 23:107--136, 2015.

\bibitem{mac1998categories}
Saunders~Mac Lane.
\newblock {\em Categories for the Working Mathematician}.
\newblock Number~5 in Graduate Texts in Mathematics. Springer-Verlag, 1971.

\bibitem{meyer2023matrix}
Carl~D Meyer.
\newblock {\em Matrix analysis and applied linear algebra}.
\newblock SIAM, 2023.

\bibitem{nakayama1941frobeniusean}
Tadasi Nakayama.
\newblock On {F}robeniusean algebras. {II}.
\newblock {\em Annals of Mathematics}, 42(1):1--21, 1941.

\bibitem{skowronski2017frobenius}
Andrzej Skowro{\'n}ski and Kunio Yamagata.
\newblock {\em Frobenius Algebras I}.
\newblock EMS Textbooks in Mathematics. European Mathematical Society, 2011.

\bibitem{vicary2011acategorical}
Jamie Vicary.
\newblock Categorical formulation of finite-dimensional quantum algebras.
\newblock {\em Communications in Mathematical Physics}, 304(3):765 – 796,
  2011.
\newblock Cited by: 40.

\bibitem{zieschang1996}
Paul-Hermann Zieschang.
\newblock {\em An Algebraic Approach to Association Schemes}, volume 1628 of
  {\em Lecture Notes in Mathematics}.
\newblock Springer, Berlin, 1996.

\bibitem{zieschang2005theory}
Paul-Hermann Zieschang.
\newblock {\em Theory of Association Schemes}.
\newblock Springer Monographs in Mathematics. Springer, Berlin, 2005.

\bibitem{zieschang2023hypergroups}
Paul-Hermann Zieschang.
\newblock {\em Hypergroups}.
\newblock Springer, 2023.

\end{thebibliography}
\end{document}